\documentclass[ejs, preprint]{imsart}
\usepackage{amsmath, amssymb,amsthm, fullpage, color, relsize, enumerate, enumitem, hyperref, mathtools}
\usepackage{latexsym}
\usepackage{xparse}
\usepackage{array}
\usepackage{xstring}
\usepackage{indentfirst}
\usepackage{bm,dsfont} 
\usepackage{graphicx,adjustbox}
\usepackage{chngcntr}
\usepackage{import}
\usepackage{comment}
\usepackage{amsfonts}
\usepackage{lipsum}
\usepackage[OT1]{fontenc}
\usepackage[utf8]{inputenc}
\usepackage{mathtools}  
\usepackage{multirow,multicol}
\usepackage{cite}
\usepackage{subfig}
\usepackage{caption}
\usepackage[russian, english]{babel}
\numberwithin{equation}{section}
\usepackage[toc,page]{appendix}
\usepackage{pgfplots}
\pgfplotsset{compat=newest} 
\usepackage{tikzscale}
\pgfplotsset{plot coordinates/math parser=false}

\DeclareMathOperator{\tr}{tr}

\theoremstyle{definition}
\newtheorem{defn}{Definition}[section]

\newtheorem{prop}{Proposition}[section]

\newtheorem{lem}{Lemma}[section]
\newtheorem{thm}{Theorem}[section]
\newtheorem{cor}{Corollary}[section]
\newtheorem{rem}{Remark}[section]
\newtheorem{assu}{Assumption}[section]

\newtheorem{clawithinpf}{Claim}
\makeatletter
\@addtoreset{clawithinpf}{proof}
\makeatother

\makeatletter
\newcommand{\thickhline}{%
	\noalign {\ifnum 0=`}\fi \hrule height 1pt
	\futurelet \reserved@a \@xhline
}
\newcolumntype{"}{@{\hskip\tabcolsep\vrule width 1pt\hskip\tabcolsep}}
\makeatother

\makeatletter
\def\amsbb{\use@mathgroup \M@U \symAMSb}
\makeatother

\makeatletter
\g@addto@macro\@floatboxreset\centering
\makeatother

\makeatletter
\newcommand*{\rom}[1]{\expandafter\@slowromancap\romannumeral #1@}
\makeatother

\newcommand{\suchthat}{\mathop{\mathrm{s.t.}}}

\newcommand{\bb}[1]{\mathbb{#1}} 
\newcommand{\bbM}[1]{\mathds{#1}} 
\newcommand{\cc}[1]{\mathcal{#1}} 
 
\newcommand{\zero}{\boldsymbol{0}}

\newcommand{\cov}{\mathop{\mathrm{cov}}} 
\newcommand{\var}{\mathop{\mathrm{var}}} 
\newcommand{\sd}{\mathop{\mathrm{std}}} 
\newcommand{\diam}{\mathop{\mathrm{diam}}}

\newcommand{\RHS}{\mathop{\mathrm{RHS}}} 
  
\newcommand{\argmin}{\operatornamewithlimits{arg\min}}
\newcommand{\argmax}{\operatornamewithlimits{arg\max}}

\newcommand{\brac}[1]{\left[#1\right]}
\newcommand{\set}[1]{\left\{#1\right\}}
\newcommand{\abs}[1]{\left\lvert #1 \right\rvert}
\newcommand{\paren}[1]{\left(#1\right)}

\newcommand{\brapar}[1]{\left[#1\right)}
\newcommand{\Bigparen}[1]{\Big(#1\Big)}
\newcommand{\Bigbrac}[1]{\Big[#1\Big]}
\newcommand{\Bigset}[1]{\Big\{#1\Big\}}

\newcommand{\InnerProd}[2]{\langle #1,#2 \rangle}

\newcommand{\cp}[1]{\overset{#1}{\rightarrow}}    
\newcommand{\eqd}{\overset{d}{=}}  
\newcommand{\RelNum}[2]{\overset{#1}{#2}}

\newcommand{\LpNorm}[2]{\left\| #1\right\|_{\ell_{#2}}}
\newcommand{\SpNorm}[2]{\left\| #1\right\|_{\cc{S}_{#2}}}
\newcommand{\norm}[2]{\left\|#1\right\|_{#2}}
\newcommand{\OpNorm}[3]{\left\| #1\right\|_{#2\rightarrow#3}}
\newcommand{\RNum}[1]{\uppercase\expandafter{\romannumeral #1\relax}}

\begin{document}	

\begin{frontmatter}
	
\title{Local inversion-free estimation of spatial Gaussian processes\protect\thanksref{Th1, Th2}}
\runtitle{Local inversion-free estimation of spatial Gaussian processes}
\thankstext{Th1}{This research is partially supported by NSF grant ACI-1047871. Additionally, CS is partially supported by NSF Grants 1838179 and 1422157, and LN by NSF CAREER award DMS-1351362, NSF CNS-1409303, and NSF CCF-1115769.}
\thankstext{Th2}{We are grateful to Mihai Anitescu and Michael Stein for valuable discussions and suggestions.}

\begin{aug}
\author{Hossein Keshavarz\ead[label=e1]{hkeshava@umn.edu}}
\and
\author{XuanLong Nguyen\ead[label=e2]{xuanlong@umich.edu}}
	
\address{Institute for Mathematics and its Applications, University of Minnesota, Minneapolis\\Department of Statistics, University of Michigan, Ann Arbor\\ \printead{e1,e2}}
	
\author{Clayton Scott \ead[label=e3]{clayscot@umich.edu}}
		
\address{Department of Electrical Engineering and Computer Science, University of Michigan, Ann Arbor\\ \printead{e3}}	

\runauthor{F. Author et al.}
\affiliation{University of Minnesota and University of Michigan}	
\end{aug}

\begin{abstract}
Maximizing the likelihood has been widely used for estimating the unknown covariance parameters of spatial Gaussian processes. However, evaluating and optimizing the likelihood function can be computationally intractable, particularly for large number of (possibly) irregularly spaced observations, due to the need to handle the inverse of ill-conditioned and large covariance matrices.  Extending the ``inversion-free'' method of Anitescu, Chen and Stein \cite{anderes2010consistent}, we investigate a broad class of covariance parameter estimation based on inversion-free surrogate losses and block diagonal approximation schemes of the covariance structure. This class of estimators yields a spectrum for negotiating the trade-off between statistical accuracy and computational cost. We present fixed-domain asymptotic properties of our proposed method, establishing $\sqrt{n}$-consistency and asymptotic normality results for isotropic Matern Gaussian processes observed on a multi-dimensional and irregular lattice. Simulation studies are also presented for assessing the scalability and statistical efficiency of the proposed algorithm for large data sets.
\end{abstract}
	
\begin{keyword}[class=MSC]
\kwd[Primary ]{62M30}
\kwd{62M40}
\kwd[; secondary ]{60G15}
\end{keyword}
	
\begin{keyword}
\kwd{Local inversion-free covariance estimation}
\kwd{Gaussian process}
\kwd{Computationally scalable}
\kwd{Fixed-domain asymptotic analysis}
\kwd{Irregularly spaced observations}
\end{keyword}
	
	
\end{frontmatter}

\section{Introduction}\label{Intro}

Gaussian processes (GPs) are one of the most common modelling tools for the analysis of
spatiotemporal data (see e.g., \cite{cressie1992statistics,gelfand2010handbook}). A crucial aspect of GP-based
inference is the estimation of its covariance function. The covariance function
is typically specified up to a finite number of parameters, the estimation of which is
pivotal for performing interpolation and prediction tasks.

While there are a number of likelihood-based techniques for covariance estimation, they do not scale well. Indeed, exact evaluation of the Gaussian likelihood requires computing the inverse of the covariance matrix, which generally requires $\cc{O}\paren{n^3}$ operations and $\cc{O}\paren{n^2}$ space storage. A number of authors have proposed ways of getting around this challenge, by working instead with an approximate version of the likelihood function. Vecchia \cite{vecchia1988estimation} considered an approximation by ignoring the conditional correlation of distant sites given their nearest neighbours. This idea was further extended by Stein et al. \cite{stein2004approximating} who studied more flexible choices of conditioning sets. The key to evaluating the exact log-likelihood function and its partial derivatives boils down to solving large and dense systems of linear equations. To accelerate such linear solvers, e.g., using the \emph{Krylov subspace iteration method}, Furrer et al. \cite{furrer2006covariance} and Kaufman et al. \cite{kaufman2008covariance} exploit the tapering technique to sparsify the dense covariance matrix. More recently, several authors investigated a stochastic optimization technique for implementing the MLE \cite{anitescu2012matrix,stein2013stochastic}. Their proposed algorithms are statistically comparable to MLE, if the condition number of the covariance matrix has a uniform upper bound (independent of the sample size). 

An attractive alternative to likelihood based techniques is to abandon the likelihood function altogether, and consider instead surrogate loss functions which may be evaluated and optimized more efficiently. Anitescu, Chen and Stein \cite{anitescu2017inversion} proposed one such surrogate loss based method for covariance estimation, and showed that it is considerably computationally more efficient than the standard MLE, especially for irregularly spaced observations. Indeed, their loss function, which we call \emph{inversion-free (IF)} in this manuscript, does not require computing the precision matrix (covariance inverse), and so it can be evaluated in $\cc{O}\paren{n^2}$ time. It was established by the authors that when the covariance matrix has a bounded condition number, the resulting estimate possesses consistency and asymptotic normality \cite{keshavarz2016consistency}. It is noted that the boundedness of condition number holds in the increasing domain setting, where the minimum distance among the sampling points is bounded away from zero. This is in contrast to the scenarios in which the GP is observed in a fixed and bounded domain, where the observations get denser as the sample size $n$ increases. In this new regime, which is referred to as \emph{fixed-domain} (or \emph{infill}) setting, because of strong spatial correlation the condition number often grows without bound with $n$. This points to an unresolved question regarding the statistical efficiency of the inversion-free algorithm in the fixed-domain setting, including the situation of irregularly spaced observations.

In this article, we adopt and extend the basic surrogate loss based approach of \cite{anitescu2017inversion}, while looking to address the theoretical questions described above. A natural adaptation of the IF loss function is to apply it to a transformed version of the data using a transformation technique that helps to reduce the strong correlation among
the (original) observations. A fast and root-$n$ consistent estimator studied by Anderes \cite{anderes2010consistent} can be viewed this way, as it is based on squared increments of the observed Gaussian process. In his work samples are transformed using directional increments of the Gaussian process. However this method is applicable only to regularly spaced observations. A general scheme for dependence reduction, which we refer to as \emph{preconditioning}, was introduced in \cite{chen2013use, stein2012difference} and chapter $3$ of \cite{lee2012local}. The preconditioning technique is one of the building blocks of our proposed estimation algorithm. It will be shown that this preconditioner provides a suitable transformation in the case of irregularly placed observations.

The second ingredient of our approach is to apply a divide-and-conquer technique to design of the surrogate loss function, which will be referred to as the \emph{local inversion-free (LIF)} loss. Specifically, the (preconditioned) samples are divided into $b_n$ possibly overlapping clusters (bins). The LIF loss is composed by taking a weighted average of the IF loss functions over these bins. The covariance estimates are obtained by optimizing with respect to the LIF loss function. The aforementioned preconditioning technique is crucial for the statistical efficiency of the LIF algorithm as it helps reduce the correlation between distant clusters.  

The resulting LIF procedure comprises a rich and flexible class of estimation algorithms, depending on the number of bins $b_n$, and specific binning scheme determined by the size and shape of each bin. When $b_n=1$, our algorithm reduces to the inversion-free method of \cite{anitescu2017inversion}, but applied with the preconditioning scheme that we will describe. Furthermore, the quadratic variation-based approach of \cite{anderes2010consistent} is a special instance in the LIF class, specifically corresponding to the other extreme scenario of $b_n = n$. Thus, the LIF class can be viewed as a spectrum of algorithms bridging between two distinct approaches in the literature. A noted advantage of our procedure in exploiting the divide and conquer strategy is to significantly expedite the estimation procedure, while preserving favorable statistical properties. Indeed, the LIF loss can be evaluated in order $n^2/b_n \ll n^2$ operations. 

A considerable portion of this article is devoted to the investigation of the asymptotic behavior of the proposed LIF based estimation method in the fixed-domain regime. Theoretical analysis for several specific instances of LIF based estimation have been carried out before, by \cite{anderes2010consistent} on his quadratic variation based method on regularly spaced observations in the fixed-domain framework, and in the increasing domain regime by the authors \cite{keshavarz2016consistency}. The asymptotic theory for the fixed-domain regime is considerably more involved than the increasing domain regime, especially for irregularly spaced observations.

It is established by \cite{zhang2004inconsistent} that for the isotropic Matern GP, the variance $\phi$ and the range parameter $\rho$ are not identifiable when dimension $d\leq 3$. Thus we only concentrate on estimating the so-called \emph{microergodic} parameter (see page $163$ of \cite{stein2012interpolation} for the exact definition), namely $\phi\rho^{-2\nu}$ where $\nu$ quantifies the smoothness of GP. The microergodic parameter is of great interest as it determines the asymptotic mean square estimation error in the fixed-domain setting (e.g., pages $174-175$ of \cite{stein2012interpolation}). We show that under some regularity conditions and for any binning scheme, all the stationary points of the LIF objective function are concentrated around the true parameter on a ball of radius $\cc{O}(\sqrt{n^{-1}\log n})$, with high probability. We also establish the asymptotic normality of this estimate. Hence, the LIF loss does not sacrifice asymptotic rate for increasing the computational speed and memory efficiency, even for irregularly spaced observations. The treatment of observations on irregular lattices distinguishes our theoretical contribution from the previous works of \cite{anderes2010consistent,wang2011fixed,ying1991asymptotic}.

Following the theoretical study, a comprehensive set of synthetic numerical experiments are conducted for assessing the role of preconditioning, the irregularity of sampling locations, and the binning scheme in the performance of the LIF estimate. Despite the robustness of the asymptotic rate to changes of $b_n$ and the shape of the bins, such factors can still affect the bias and variance of the LIF estimator, particularly for moderate sample sizes. Our simulation studies serve to corroborate the asymptotic theory, but also reveal the stability of the LIF estimate with respect to the size and shape of the bins. We evaluate the efficiency of our method for data sets up to $2.5\times 10^5$ data points. 

\paragraph{Plan of the paper.}
Section \ref{ProbForm} describes the geometry of sampling sites, preconditioning, and the IF method. In Section \ref{LocACSSection}, we propose the family of the LIF loss functions and introduce an efficient parallel technique for evaluating such functions. Section \ref{FxddomAnal} establishes the infill asymptotic properties of the LIF algorithm such as $\sqrt{n}$-consistency and asymptotic normality, given samples in a $d$-dimensional space with $d\leq 3$. In Section \ref{SimulStud} we present a series of simulation studies to assess the performance of the LIF estimator. Section \ref{Discus} serves as the conclusion and discusses future directions. We substantiate the main results of the paper in Section \ref{Proofs}. Finally, Appendices \ref{AppendixCovMatrixIrregLat} and \ref{AuxRes} not only contain some auxiliary technicalities which are crucial in Section \ref{Proofs}, but also present a comprehensive sensitivity analysis of the correlation matrix of the preconditioned data with respect to the range parameter, which may be useful for the asymptotic analysis of other estimation algorithms in geostatistics.

\paragraph{Notation.}
For the convenience of the reader, we collect standard pieces of notation here. $j = \sqrt{-1}$ denotes the imaginary unit. Boldface symbols denote vectors. $\wedge$ and $\vee$ stand for the minimum and maximum operators. For any $m\in\bb{N}$, $\zero_m$ denotes the all zeros column vector of length $m$. Furthermore, for any $p\in\set{1,\ldots,m}$, $\boldsymbol{e}_p$ denotes the unit vector along the $p^{\mbox{th}}$ coordinate. If $\boldsymbol{u}$ and $\boldsymbol{v}$ are vectors of length $m$, then $\boldsymbol{u}^{\boldsymbol{v}}$ denotes $\Pi^m_{i=1} u^{v_i}_i$ (we define $0^0$ to be $1$). For square matrices $A$ and $B$ of the same size, by writing $A \succeq B$, we mean that $A-B$ is symmetric positive semi-definite. Furthermore, $\InnerProd{A}{B}{}\coloneqq \tr\paren{A^\top B}$ refers to their trace inner product. We use various types of matrix norms on $A\in\bb{R}^{n\times n}$ in this paper. For any $p\in\brapar{1,\infty}$, $\LpNorm{A}{p}\; \coloneqq \paren{\sum_{i,j} \abs{A_{ij}}^p}^{1/p}$ stands for the element-wise $p-$norm of $A$. We also write $\OpNorm{A}{2}{2}$ to denote the usual operator norm (largest singular value) of $A$. Moreover $\SpNorm{A}{1}$ represents the sum of the singular values of $A$, which is called the nuclear norm. We also write $\diam\paren{\Omega} = \sup_{\omega_1,\omega_2\in\Omega} \LpNorm{\omega_2-\omega_1}{2}$ to denote the diameter of a bounded set $\Omega\subset\bb{R}^m$. For a symmetric, positive semi-definite $A\in\bb{R}^{n\times n}$ with spectral decomposition $A = U\Lambda U^\top$, $\sqrt{A}\coloneqq U \Lambda^{1/2} U^\top$ represents its symmetric square root. For two non-negative sequences $\set{a_m}^{\infty}_{m=1}$ and $\set{b_m}^{\infty}_{m=1}$, we write $a_m \asymp b_m$ if there are strictly positive and bounded scalars $C_{\min},C_{\max}$ such that $C_{\min} \leq \lim\limits_{m\rightarrow\infty} a_m/b_m \leq C_{\max}$. Moreover, $a_m \lesssim b_m$ refers to the case that $a_m/b_m \leq C_{\max} < \infty$ as $m\rightarrow\infty$. Lastly, $\cc{K}_{\nu}\paren{\cdot}$ and $\Gamma\paren{\cdot}$ respectively represent the modified Bessel function of the second kind of order $\nu$ and the Gamma function.
   
\section{Preconditioning and inversion-free surrogate loss}\label{ProbForm}

\subsection{Gaussian processes observed on irregular lattices}

Consider a zero mean, real valued, and stationary Gaussian process $G$ on domain $\cc{D}$,
where $\cc{D}$ is a bounded subset of $\bb{R}^d$ such as $\brac{0,1}^d$. 
The dependence structure of $G$ is typically parametrized by a
variance parameter $\phi_0>0$ and a (correlation) range parameter $\rho_0$. Specifically, if $G$ is a geometric anisotropic process on $\cc{D}$, then there are a fully known covariance function $K$ and a matrix $\rho_0\in\bb{R}^{d\times d}$ such that
\begin{equation*}
\bb{E} G\paren{\boldsymbol{s}}G\paren{\boldsymbol{t}} = \phi_0 K\paren{ \LpNorm{ \rho^{-1}_0\paren{\boldsymbol{t}-\boldsymbol{s}} }{2} },\quad\forall\;\boldsymbol{s},\boldsymbol{t}\in\cc{D}
\end{equation*}
The objective is to estimate the microergodic parameters of the covariance function, given $n$ measurements from one realization of $G$ at locations $\cc{D}_n = \set{\boldsymbol{s}_1,\ldots,\boldsymbol{s}_n}\subset\cc{D}$. 
Throughout the paper, we assume that $\rho_0$ belongs to a compact, connected space $\Theta_0$ 
(with respect to the Euclidean distance). We also restrict $d$ to be less than or equal $3$. 

As the first step we precisely formulate $\cc{D}_n$. $\cc{D}_n$ is called a $d$-dimensional regular (rectangular) lattice with $n = N^d$ point, if $\cc{D}_n = \set{1/N,\ldots,1}^d$. In such a lattice the smallest distance between neighboring locations decreases with the rate of $N^{-1}$. This fact provides a clue for extending the notion of the regular lattice into irregular ones, which can be formalized as follows (see \cite{lee2012local}):

\begin{assu}\label{RegCondLattice}
Let $\cc{D}_n\subset\cc{D}$ be a set of size $n$. For any $\boldsymbol{s}\in\cc{D}_n$, let $r_{\boldsymbol{s},i}$ denote the distance from $\boldsymbol{s}$ to its $i^{\mbox{th}}$ closest neighbor in $\cc{D}_n\setminus\set{\boldsymbol{s}}$. There are positive scalars $C_{\min}$ and $C_{\max}$ such that
\begin{equation}\label{RegCondLatticeEq}
C_{\min}\paren{\frac{i}{n}}^{\frac{1}{d}} \leq r_{\boldsymbol{s},i}\leq C_{\max} \paren{\frac{i}{n}}^{\frac{1}{d}},\quad \forall\;\boldsymbol{s}\in\cc{D}_n,\;\mbox{and}\; i = 1,\cdots,\paren{n-1}.
\end{equation} 
\end{assu}

The properties required by the assumption enlarge the notion of regular lattice in three aspects. First, in contrast to the number of points in a regular lattice, there is no restriction on $n$. Moreover, $\cc{D}$ is not restricted to be $\brac{0,1}^d$. For instance, $\cc{D}$ might be the union of a finite number of connected components, as long as each of them satisfy condition \eqref{RegCondLatticeEq} and encompasses a non-vanishing fraction of samples, as $n$ tends to infinity. Finally, $\cc{D}_n$ needs not form a $d-$dimensional regular lattice. 

\subsection{Preconditioning}

Controlling the strong spatial dependence between the observed samples $\set{G\paren{\boldsymbol{s}_1},\ldots,G\paren{\boldsymbol{s}_n}}$ via preconditioning is essential for reducing the condition number of the 
covariance matrix. It plays a crucial role in the estimation procedure we will propose. 
Various types of preconditioners have been studied for GPs observed on regular and irregular lattices in the literature (see e.g., \cite{chen2013use, lee2012local, stein2012difference}). 

We shall adopt a preconditioning scheme proposed by Lee \cite{lee2012local} for irregularly spaced observations. Before proceeding further, it is convenient to define $N \coloneqq \lfloor n^{1/d} \rfloor$. Furthermore for any $\boldsymbol{s}\in\cc{D}_n$,  $\cc{N}_m\paren{\boldsymbol{s}}$ represents a set points (in $\cc{D}_n$) in a small neighbourhood of radius $\cc{O}\paren{N^{-1}}$ around $\boldsymbol{s}$ whose size depends on $m$. Namely, $\LpNorm{\boldsymbol{t}-\boldsymbol{s}}{2} \lesssim 1/N$ for any $\boldsymbol{t}\in \cc{N}_m\paren{\boldsymbol{s}}$.

\begin{defn}\label{DecorFltNonReglat}
Let $m\in\bb{N}$ (which does not grow with $n$). Suppose that there are  sets of real coefficients $\set{a_{m,\boldsymbol{s}}\paren{\boldsymbol{t}}:\; \boldsymbol{t}\in \cc{N}_m\paren{\boldsymbol{s}} }, \; \boldsymbol{s}\in\cc{D}_n$, satisfying the following conditions:
\begin{enumerate}
\item For any $\boldsymbol{r}\in\bb{Z}^d_{+}$ (the entries of $\boldsymbol{r}$ are non-negative ) and $\LpNorm{\boldsymbol{r}}{1}< m$, $\sum_{\boldsymbol{t}\in \cc{N}_m\paren{\boldsymbol{s}}} a_{m,\boldsymbol{s}}\paren{\boldsymbol{t}} \paren{ \boldsymbol{t}-\boldsymbol{s} }^{\boldsymbol{r}} = 0$.
\item There is a vector $\boldsymbol{r}\in\set{0,1,\ldots}^d$ with  $\LpNorm{\boldsymbol{r}}{1}\geq m$ such that $\sum_{\boldsymbol{t}\in \cc{N}_m\paren{\boldsymbol{s}}} a_{m,\boldsymbol{s}}\paren{\boldsymbol{t}} \paren{ \boldsymbol{t}-\boldsymbol{s} }^{\boldsymbol{r}} \ne 0$.
\item $\sum_{\boldsymbol{t}\in \cc{N}_m\paren{\boldsymbol{s}}} a^2_{m,\boldsymbol{s}}\paren{\boldsymbol{t}} = 1$ and $a_{m,\boldsymbol{s}}\paren{\boldsymbol{t}} \neq 0$ for all $\boldsymbol{t}\in \cc{N}_m\paren{\boldsymbol{s}}$.
\end{enumerate}
We say $G_m$ is a \emph{preconditioned process of order} $m$, if  
\begin{equation}\label{G_mIrregLat}
G_m\paren{\boldsymbol{s}} \coloneqq N^{\nu} \sum_{\boldsymbol{t}\in \cc{N}_m\paren{\boldsymbol{s}}} a_{m,\boldsymbol{s}}\paren{\boldsymbol{t}} G\paren{\boldsymbol{t}},\quad \forall\; \boldsymbol{s}\in\cc{D}_n.
\end{equation}
\end{defn}

\begin{rem}
Since $\cc{N}_m\paren{\boldsymbol{s}}$ is constructed by the nearest neighbors of $\boldsymbol{s}$, the preconditioned process is approximately proportional to the $m$-th derivative of $G$ at $\boldsymbol{s}$, for large $N$. We also normalize the coefficients $\set{a_{m,\boldsymbol{s}}\paren{\boldsymbol{t}}:\; \boldsymbol{t}\in \cc{N}_m\paren{\boldsymbol{s}} }$ by their Euclidean norm to uniformly control the magnitude of $G_m$ over $\cc{D}_n$. Moreover, for reducing ambiguity in the definition of $G_m$, $\cc{N}_m\paren{\boldsymbol{s}}$ is chosen to be a minimal set, with respect to the inclusion ordering, satisfying the conditions in Definition \ref{DecorFltNonReglat}. The cardinality of $\cc{N}_m\paren{\boldsymbol{s}}$ depends on $d,m$ and the geometric structure of neighboring observations around $\boldsymbol{s}$ in $\cc{D}_n$ and may vary across $\cc{D}_n$. The reader can deduce from a simple combinatorial argument that the first condition in Definition \ref{DecorFltNonReglat} is translated as ${d+m-1 \choose d}$ linear constraints on the set of coefficients $\set{a_{m,\boldsymbol{s}}\paren{\boldsymbol{t}}:\; \boldsymbol{t}\in \cc{N}_m\paren{\boldsymbol{s}}}$. This fact gives a rough estimate of the size of $\cc{N}_m\paren{\boldsymbol{s}}$. 
\end{rem}

\begin{rem}\label{PreCondRegLat}
A preconditioning method for the $d$-dimensional regular lattices $\cc{D}_n = \set{1/N,\ldots,1}^d$ has been studied in Stein et al. \cite{stein2012difference}. Discarding the boundary points of $\cc{D}_n$, the preconditioned process is constructed on $\cc{D}^\circ_n = \set{\paren{m+1}/N,\ldots, 1-m/N}^d$ by $m-$times application of the discrete Laplace operator. More specifically, the preconditioner is recursively defined via
\begin{align}\label{G_mRegLat}
& G_0\paren{\boldsymbol{s}} = N^\nu G\paren{\boldsymbol{s}},\quad\forall\; \boldsymbol{s}\in\cc{D}_n,\nonumber\\
& G_{2k}\paren{\boldsymbol{s}} = \sum_{r=1}^{d} \brac{G_{2k-2}\paren{\boldsymbol{s}+\frac{\boldsymbol{e}_r}{N}}-2G_{2k-2}\paren{\boldsymbol{s}} + G_{2k-2}\paren{\boldsymbol{s}-\frac{\boldsymbol{e}_r}{N}}}, \quad \boldsymbol{s}\in\cc{D}^\circ_n,\; k = 1,\ldots,m.
\end{align}	
To avoid unnecessary algebraic complexity in Eq. \eqref{G_mRegLat}, the preconditioning coefficients have not been normalized to be of norm one. It can be shown that after proper normalization, $G_{2m}$ admits the conditions of Definition \ref{DecorFltNonReglat} with order $2m$. Namely, \eqref{G_mRegLat} gives a recursive way of constructing the preconditioned process of even orders for regular lattices. It is also worth mentioning that although $G_{2m}$ defined by \eqref{G_mRegLat} is a stationary process, preconditioning does not necessarily preserve stationarity for irregular lattices.
\end{rem}

\begin{rem}
The preconditioned coefficients in Definition \ref{DecorFltNonReglat} are carefully chosen so that $G_m\paren{\cdot}$ carries no information about the directional derivatives of $G$ of order less than $m$. Strictly speaking, the Taylor expansion of $G$ around $\boldsymbol{s}$ ensures the existence of an stochastic process $\Delta_m$ such that for any $\boldsymbol{t}\in\cc{N}_m\paren{\boldsymbol{s}}$,
\begin{equation*}
G\paren{\boldsymbol{t}} = \sum_{b=0}^{m-1}\sum_{\boldsymbol{r}\in \bb{Z}^d_{+}, \;\LpNorm{r}{1}=b}\frac{1}{b!} \InnerProd{\paren{\boldsymbol{t}-\boldsymbol{s}}^{\boldsymbol{r}}}{D^{\boldsymbol{r}}G\paren{\boldsymbol{s}}} + \Delta_m\paren{\boldsymbol{t}}.
\end{equation*}
Here $D^{\boldsymbol{r}} G\paren{\cdot}$ denotes the $\boldsymbol{r}^{\mbox{th}}$ directional derivative of $G$. Replacing this representation of $G$ into $G_m$ yields
\begin{equation*}
G_m\paren{\boldsymbol{s}} = N^{\nu} \sum_{b=0}^{m-1}\sum_{\boldsymbol{r}\in \bb{Z}^d_{+}, \;\LpNorm{r}{1}=b}\frac{1}{b!} \InnerProd{\sum_{\boldsymbol{t}\in \cc{N}_m\paren{\boldsymbol{s}}} a_{m,\boldsymbol{s}}\paren{\boldsymbol{t}} \paren{ \boldsymbol{t}-\boldsymbol{s} }^{\boldsymbol{r}}}{D^{\boldsymbol{r}}G\paren{\boldsymbol{s}}} + N^{\nu} \sum_{\boldsymbol{t}\in \cc{N}_m\paren{\boldsymbol{s}}} a_{m,\boldsymbol{s}}\paren{\boldsymbol{t}}\Delta_m\paren{\boldsymbol{t}}.
\end{equation*}
The first condition in Definition \ref{DecorFltNonReglat} implies that 
\begin{equation*}
G_m\paren{\boldsymbol{s}} = N^{\nu} \sum_{\boldsymbol{t}\in \cc{N}_m\paren{\boldsymbol{s}}} a_{m,\boldsymbol{s}}\paren{\boldsymbol{t}}\Delta_m\paren{\boldsymbol{t}}.
\end{equation*}
We finally present a concrete example satisfying the conditions in Definition \ref{DecorFltNonReglat}. Note that Remark \ref{PreCondRegLat} constructs the preconditioning coefficients for regularly observed GPs. It is also easy to show that Definition is almost surely well-defined for randomly perturbed lattices (if the perturbation vector is absolutely continuous with respect to the Lebesgue measure). We refer the reader to Chapter $3$ of \cite{lee2012local} for further discussion.
\end{rem}

\subsection{The IF algorithm}\label{ACS}

Anitescu, Stein and Chen \cite{anitescu2017inversion} introduced a parameter estimation method based on an ``inversion-free'' surrogate loss for the Gaussian process that is both easy to compute and optimize. Let $Y_m$ represent the column vector of the preconditioned samples, i.e., $Y_m = \brac{G_m\paren{\boldsymbol{s}}:\; \boldsymbol{s}\in\cc{D}_n}^\top$. We use $K_m$ to denote the covariance function of $G_m$ normalized by factor $\phi_0$. $K_m$ can be easily expressed in terms of the correlation function of $G$, $K\paren{\cdot,\rho_0}$, and the preconditioning coefficients.
\begin{equation*}
K_m\paren{\boldsymbol{s}, \boldsymbol{t};\rho_0} = \frac{\bb{E}G_m\paren{\boldsymbol{s}}G_m\paren{\boldsymbol{t}}}{\phi_0} = N^{2\nu} \sum_{ \boldsymbol{s'}\in \cc{N}_m\paren{\boldsymbol{s}} }\sum_{ \boldsymbol{t'}\in \cc{N}_m\paren{\boldsymbol{t}} } a_{m,\boldsymbol{s}}\paren{\boldsymbol{s'}}a_{m,\boldsymbol{t}}\paren{\boldsymbol{t'}} K\paren{\boldsymbol{t'}-\boldsymbol{s'};\rho_0}.
\end{equation*}
We also use $\phi_0K_{n,m}\paren{\rho_0}$ to denote the covariance matrix of $Y_m$. That is
\begin{equation}\label{K_n,m}
\bb{E} Y_mY^\top_m = \phi_0K_{n,m}\paren{\rho_0} \coloneqq \phi_0 \brac{K_m\paren{\boldsymbol{s},\boldsymbol{t};\rho_0}}_{\boldsymbol{s},\boldsymbol{t}\in\cc{D}_n}.
\end{equation}
Recall that $\rho_0$ lies in a compact and connected space $\Theta_0$. The IF estimator \cite{anitescu2017inversion} of the covariance parameters $\paren{\phi_0,\rho_0}$ is given by
\begin{equation}\label{ACSOptProb}
\paren{\hat{\phi}_n, \hat{\rho}_n} = \argmax_{\phi>0, \rho\in\Theta_0} \set{ \phi Y^\top_m K_{n,m}\paren{\rho} Y_m - \frac{\phi^2}{2}\LpNorm{K_{n,m}\paren{\rho}}{2}^2 }.
\end{equation}
Note that \eqref{ACSOptProb} can be alternatively formulated as a moment matching minimization problem,
\begin{equation*}
\paren{\hat{\phi}_n, \hat{\rho}_n} = \argmin_{\phi>0, \rho\in\Theta_0} \LpNorm{Y_mY^\top_m - \phi K_{n,m}\paren{\rho} }{2}.
\end{equation*}

\begin{rem}
From a computational perspective, the loss function in \eqref{ACSOptProb} does not depend on the Cholesky factorization of $K_{n,m}$ and can be evaluated in order $n^2$ flops even for irregularly spaced observations. Moreover, storing the whole matrix $K_{n,m}$ is not necessary for computing the objective function and its directional derivatives. In particular, storing $Y_m$ and $\cc{D}_n$, which need $\cc{O}\paren{n}$ storage, suffices for estimating the covariance parameters. 
\end{rem}

\section{The local inversion-free (LIF) algorithm}\label{LocACSSection}

We are ready to present in this section a broad class of scalable covariance estimation algorithms, building on the IF surrogate loss approach and the preconditioning technique described
in the previous section. The asymptotic theory for our estimator will be presented
in the following section.

We previously used $Y_m = \brac{G_m\paren{\boldsymbol{s}}:\;\boldsymbol{s}\in\cc{D}_n}^\top$ to denote the column vector of the preconditioned samples of order $m$. Let $\cc{B} = \set{B_t:\;t=1\ldots,b_n}$ be a partition of $\cc{D}_n$ into $b_n$ bins, i.e., $B_i\cap B_j = \emptyset$ for distinct $i,j\in\set{1,\ldots,b_n}$ and $ \cup^{b_n}_{t=1} B_t = \cc{D}_n$. We write $Y_{B_t,m} = \brac{G_m\paren{\boldsymbol{s}}:\;\boldsymbol{s}\in B_t}^\top$ to represent the column vector of the preconditioned data in $B_t,\; t=1\ldots,b_n$. Furthermore let $\phi_0K_{B_t,m}\paren{\rho_0}$ denote the covariance matrix of $Y_{B_t,m}$. Namely,
\begin{equation}\label{K_Bt,m}
\bb{E} Y_{B_t,m}Y^\top_{B_t,m} = \phi_0K_{B_t,m}\paren{\rho_0} \coloneqq \phi_0 \brac{K_m\paren{\boldsymbol{s},\boldsymbol{t};\rho_0}}_{\boldsymbol{s},\boldsymbol{t}\in B_t},\quad \forall \;t=1\ldots,b_n,
\end{equation}
in which $\phi_0K_m\paren{\cdot,\cdot,\rho_0}$ stands for the covariance function of $G_m$ with the parameters $\paren{\phi_0,\rho_0}$.

The LIF objective function associated to a binning scheme $\cc{B}$ is constructed by summing the IF loss functions corresponding to the $B_t$'s over $\cc{B}$. The unknown covariance parameters are estimated by maximizing the LIF function, with
\begin{equation}\label{LocACSOptProb}
\paren{\hat{\phi}_{n,\cc{B}}, \hat{\rho}_{n,\cc{B}}} = \argmax_{\phi>0, \rho\in\Theta_0} \set{ \sum_{t=1}^{b_n} \paren{\phi Y^\top_{B_t,m} K_{B_t,m}\paren{\rho} Y_{B_t,m} - \frac{\phi^2}{2}\LpNorm{K_{B_t,m}\paren{\rho}}{2}^2} },
\end{equation}
where $\hat{\phi}_{n,\cc{B}}$ and $\hat{\rho}_{n,\cc{B}}$ respectively denote the estimated variance and range parameters. 

Several remarks are in order.

\begin{rem}
The LIF class of estimators can be enriched in two possible ways. First we can drop the assumption that $\set{B_t}^{b_n}_{t=1}$ forms a partition for $\cc{D}_n$. Namely, the distinct clusters may not be mutually exclusive. The LIF loss can also be extended by considering a weighted average of the IF functions. Given a $b_n$-dimensional vector of strictly positive entries $w\in\bb{R}^{b_n}$, we may define
\begin{equation*}
\paren{\hat{\phi}_{n,\cc{B},w}, \hat{\rho}_{n,\cc{B},w}} = \argmax_{\phi>0, \rho\in\Theta_0} \set{ \sum_{t=1}^{b_n} w_t\paren{\phi Y^\top_{B_t,m} K_{B_t,m}\paren{\rho} Y_{B_t,m} - \frac{\phi^2}{2}\LpNorm{K_{B_t,m}\paren{\rho}}{2}^2} }.
\end{equation*}
However throughout the paper and for simplifying the theoretical analysis, we only consider the case of non-overlapping bins. It will also be assumed that $w_i = 1$ for all $i\in\set{1,\ldots,b_n}$.
\end{rem}

\begin{rem}\label{BlcDiagApproxKnmRem}
It is informative to take an alternative viewpoint of the LIF objective function in \eqref{LocACSOptProb} as corresponding to a block diagonal approximation of the covariance matrix. Interestingly, as a consequence of the asymptotic theory developed in the next section, this approximation does not affect the asymptotic estimation rate, but it can substantially help to speed up the computation.

The block diagonal approximation of $K_{n,m}\paren{\rho}$ corresponding to partitioning scheme $\cc{B}$, to be denoted by $K^{\cc{B}}_{n,m}\paren{\rho}$, can be described as follows. Choose any $\boldsymbol{s},\boldsymbol{s'}\in\cc{D}_n$, and let $t,t'$ denote the index of the elements in $\cc{B}$ containing $\boldsymbol{s}$ and $\boldsymbol{s'}$, i.e., $\boldsymbol{s}\in B_t$ and $\boldsymbol{s'}\in B_{t'}$. The entries of $K^{\cc{B}}_{n,m}\paren{\rho}$ can be equivalently represented by 
\begin{equation}\label{Ktilde_n,m}
\paren{K^{\cc{B}}_{n,m}\paren{\rho}}_{\boldsymbol{s},\boldsymbol{s'}} = \brac{K_{n,m}\paren{\rho}}_{\boldsymbol{s},\boldsymbol{s'}} \bbM{1}_{\set{t=t'}}.
\end{equation}
Observe that 
\begin{equation*}
\sum_{t=1}^{b_n} \LpNorm{K_{B_t,m}\paren{\rho}}{2}^2 = \LpNorm{K^{\cc{B}}_{n,m}\paren{\rho}}{2}^2,\quad\mbox{and}\quad \sum_{t=1}^{b_n} Y^\top_{B_t,m} K_{B_t,m}\paren{\rho} Y_{B_t,m} = Y^\top_m K^{\cc{B}}_{n,m}\paren{\rho}Y_m.
\end{equation*}
These identities provide an alternative form for Eq. \eqref{LocACSOptProb} in terms of  $K^{\cc{B}}_{n,m}\paren{\rho}$, namely
\begin{equation}\label{LIFAltForm}
\paren{\hat{\phi}_{n,\cc{B}}, \hat{\rho}_{n,\cc{B}}} = \argmax_{\phi>0, \rho\in\Theta_0} \paren{\phi Y^\top_m K^{\cc{B}}_{n,m}\paren{\rho} Y_m - \frac{\phi^2}{2}\LpNorm{K^{\cc{B}}_{n,m}\paren{\rho}}{2}^2}.
\end{equation}
Simply put, any member of the LIF class is equivalent to applying the IF procedure on an appropriate block diagonal approximation of the covariance matrix. 
\end{rem}

\begin{rem}
The following equivalent formulation for the optimization problem in \eqref{LIFAltForm} is more convenient for our subsequent theoretical analysis. Due to the quadratic dependence of the LIF loss on $\phi$, $\hat{\phi}_{n,\cc{B}}$ can be explicitly expressed in terms of $\hat{\rho}_{n,\cc{B}}$ as  
\begin{equation}\label{LIFAltForm2}
\hat{\phi}_{n,\cc{B}} = \frac{Y^\top_m K^{\cc{B}}_{n,m}\paren{\hat{\rho}_{n,\cc{B}}} Y_m}{ \LpNorm{K^{\cc{B}}_{n,m}\paren{\hat{\rho}_{n,\cc{B}}}}{2}^2},\quad\mbox{where}\quad \hat{\rho}_{n,\cc{B}} = \argmax_{\rho\in\Theta_0} \frac{Y^\top_m K^{\cc{B}}_{n,m}\paren{\rho} Y_m }{\LpNorm{K^{\cc{B}}_{n,m}\paren{\rho}}{2}}.
\end{equation}
The term \emph{profile LIF loss} refers to the objective function in Eq. \eqref{LIFAltForm2}, whose maximizer is $\hat{\rho}_{n,\cc{B}}$. The profile LIF loss is indeed proportional to the angle between $K^{\cc{B}}_{n,m}\paren{\rho}$ and $Y_mY^\top_m$.
\end{rem}

Finally, the following remarks focus on computational and numerical properties of the LIF method.

\begin{rem}
For the trivial partition $\cc{B} = \set{\cc{D}_n}$, the optimization problem \eqref{LocACSOptProb} is exactly the same as the IF algorithm. Note that the objective function in Eq. \eqref{LocACSOptProb} can be evaluated in $\sum_{t=1}^{b_n} \abs{B_t}^2$ floating point operations. For instance if all $\abs{B_t}$'s have the same order (as $n$ grows), then $\sum_{t=1}^{b_n} \abs{B_t}^2 \asymp n^2/b_n$. Thus the LIF objective function can be computed almost $b_n$ times faster than the one in \eqref{ACSOptProb}. In Section \ref{SimulStud}, we numerically assess the connection between the partitioning scheme of $\cc{D}_n$ and the estimation performance of \eqref{LocACSOptProb}. 
\end{rem}

\begin{rem}
The LIF objective function is much easier to compute than the log-likelihood with a proper choice of $b_n$ and the bins. However, implementing one iteration of any gradient-based optimizer for \eqref{LocACSOptProb}, such as the \emph{Broyden–Fletcher–Goldfarb–Shanno (BFGS)} method, can still be very challenging on a single computing core, particularly for large data sets ($n \approx 10^6$ or more), as it may require multiple evaluations of the LIF loss. Thus developing effective parallel schemes for computing the LIF function is a necessity for high resolution spatial GPs. For simplicity assume that all the bins have roughly the same size and we have access to $p$ identical processor with $q$ cores. For any $t=1,\ldots,b_n$, let $f_t\paren{Y_{B_t,m};\phi,\rho}$ stand for the IF function, with the parameters $\paren{\phi,\rho}$, associated to $B_t$. In the following we introduce a distributed memory parallel scheme for evaluating the LIF function.
\begin{enumerate}
\item The master processor assigns a label in $\set{1,\ldots,p}$ to each bin (each processor roughly receives $b_n/p$ bins). More specifically if $B_t$ is labelled as $i$, then the local memory of processor $i$ stores $G_m\paren{\boldsymbol{s}}$, $\cc{N}_m\paren{\boldsymbol{s}}$, and the preconditioning coefficients $\set{ a_{m,\boldsymbol{s}}\paren{\boldsymbol{t}}: \boldsymbol{t}\in \cc{N}_m\paren{\boldsymbol{s}}}$ for any $\boldsymbol{s}\in B_t$.
\item Inside each processor, the terms $f_t\paren{Y_{B_t,m};\phi,\rho}$ can be evaluated by employing basic shared memory parallel schemes for computing $\LpNorm{K_{B_t,m}\paren{\rho}}{2}$ and $K_{B_t,m}\paren{\rho}Y_{B_t,m}$. Finally the master processor aggregates the received quantities $\set{f_t\paren{Y_{B_t,m};\phi,\rho}:\; t=1,\ldots,b_n }$ from the slave processors to compute the LIF objective function.
\end{enumerate}
\end{rem}

\section{Fixed-domain asymptotic theory}\label{FxddomAnal} 

The goal of this section is to investigate the fixed-domain asymptotic properties of the LIF estimator \eqref{LIFAltForm2}. Throughout this section we assume that $G$ is a real valued GP with \emph{isotropic Matern} covariance function observed on a bounded domain $\cc{D}\subset\bb{R}^d$ with $d\leq 3$. In particular, for any $\boldsymbol{s},\boldsymbol{s'}\in\cc{D}$
\begin{equation*}
\cov\Bigparen{G\paren{\boldsymbol{s}}, G\paren{\boldsymbol{t}} } = \frac{\phi_0}{2^{\nu-1}\Gamma\paren{\nu}} \paren{\frac{\LpNorm{\boldsymbol{s}-\boldsymbol{t}}{2}}{\rho_0}}^{\nu} \cc{K}_{\nu} \paren{\frac{\LpNorm{\boldsymbol{s}-\boldsymbol{t}}{2}}{\rho_0}}.
\end{equation*}
Recall that $\nu>0$ is a known bounded constant controlling the mean squared smoothness of $G$; larger $\nu$ corresponds to smoother GP. The strictly positive scalars $\phi_0$ and $\rho_0$ respectively stand for the variance and the range parameters of $G$. 

Recall that the Matern covariance function admits a relatively simple form for its spectral density:
\begin{equation}\label{MatSpDen}
\hat{K}\paren{\boldsymbol{\omega};\phi_0,\rho_0} = \frac{\phi_0\rho^{-2\nu}_0}{\pi^{d/2}} \paren{\frac{1}{\rho^2_0} + \LpNorm{\boldsymbol{\omega}}{2}^2 }^{-\paren{\nu+d/2}}. 
\end{equation}
It is known that (see e.g., \cite{zhang2004inconsistent, keshavarz2018optimal}) for any bounded region $\cc{D}\subset\bb{R}^d$ with $d \leq 3$, the Matern covariance models with parameters $\paren{\phi_1,\rho_1}$ and $\paren{\phi_2,\rho_2}$ yield absolutely continuous measures (with respect to each other) whenever $\phi_1\rho^{-2\nu}_1 = \phi_2\rho^{-2\nu}_2$. In this case, $\paren{\phi_1,\rho_1}$ and $\paren{\phi_2,\rho_2}$ are almost surely not distinguishable when observing a single realization of $G$. In other words, given a single realization of $G$ in $\cc{D}$, we are only able to estimate $\phi_0\rho^{-2\nu}_0$ in \eqref{MatSpDen}. The quantity $\phi_0\rho^{-2\nu}_0$, which is usually referred to as the \emph{microergodic} parameter, is sufficient for interpolation purposes \cite{zhang2004inconsistent}. Thus, it suffices to focus on the estimation rate for $\phi_0\rho^{-2\nu}_0$ in our asymptotic analysis.

Recall from Remark \ref{BlcDiagApproxKnmRem} that $K^{\cc{B}}_{n,m}\paren{\cdot}$ stands for the block diagonal approximation $K_{n,m}\paren{\cdot}$. Define a real valued (stochastic) mapping over $\Theta_0$ by
\begin{equation}\label{PhiHatnB}
\hat{\phi}_{n,\cc{B}}\paren{\rho} \coloneqq \frac{Y^\top_m K^{\cc{B}}_{n,m}\paren{\rho} Y_m}{ \LpNorm{K^{\cc{B}}_{n,m}\paren{\rho}}{2}^2 },\quad \forall\;\rho\in\Theta_0.
\end{equation}
For ease of presentation, we omit the dependence of $\hat{\phi}_{n,\cc{B}}\paren{\cdot}$ on $m$ in our notation. It is also apparent from \eqref{LIFAltForm2} that $\hat{\phi}_{n,\cc{B}} = \hat{\phi}_{n,\cc{B}}\paren{\hat{\rho}_{n,\cc{B}}}$. 

Before presenting the main results let us consider an interesting special instance in the LIF class of estimators that reveals a 
key reason behind the $\sqrt{n}$-consistency of any LIF estimation method.

\begin{rem}\label{Rem4.1}
Suppose that $\cc{B}$ comprises only singleton sets, i.e. $\abs{B_t} = 1$ for any $B_t\in\cc{B}$. In this case $\phi K_{B_t,m}\paren{\rho}$ (the covariance matrix of  $\brac{G_m\paren{\boldsymbol{s}}:\;\boldsymbol{s}\in B_t }^\top$ associated to $\phi$ and $\rho$) is a scalar which is approximately proportional to $\phi\rho^{-2\nu}$. More specifically, using a similar approach as in the proof of Proposition \ref{UppBndCovGmNonRegLatt} shows that for $B_t = \set{\boldsymbol{s}}$
\begin{equation}\label{Eq3.4}
\phi K_{B_t,m}\paren{\rho} = C_{\boldsymbol{s}} \phi\rho^{-2\nu} + \varepsilon_{n}\paren{\boldsymbol{s},\rho,\phi},
\end{equation}
in which $C_{\boldsymbol{s}}$ is a known scalar, independent of $\phi$ and $\rho$, and $\varepsilon_{n}\paren{\boldsymbol{s},\rho,\phi}$ is a vanishing sequence in $n$ (which also depends on $m,d,\nu$ as well). Substituting Eq. \eqref{Eq3.4} into Eq. \eqref{PhiHatnB} leads to
\begin{equation}\label{AnderesAlg}
\hat{\phi}_{n,\cc{B}}\paren{\rho}\rho^{-2\nu} = \paren{\frac{\sum_{\boldsymbol{s}\in\cc{D}_n } C_{\boldsymbol{s}} G^2_m\paren{\boldsymbol{s}} }{\sum_{\boldsymbol{s}\in\cc{D}_n } C^2_{\boldsymbol{s}} }} + o\paren{1},\quad\forall\;\rho\in\Theta_0.
\end{equation}
$\hat{\phi}_{n,\cc{B}}\paren{\rho}\rho^{-2\nu}$ has a simpler representation for regular lattices as $C_{\boldsymbol{s}}$ is constant over $\cc{D}^\circ_n$ ($\cc{D}^\circ_n$ has been defined in Remark \ref{PreCondRegLat} and denotes the interior of $\cc{D}_n$). Furthermore, the profile LIF loss has (roughly) no dependence on $\rho$, since
\begin{equation*}
\frac{\sum_{t=1}^{b_n} Y^\top_{B_t,m} K_{B_t,m}\paren{\rho} Y_{B_t,m} }{\sqrt{\sum_{t=1}^{b_n} \LpNorm{K_{B_t,m}\paren{\rho}}{2}^2}} = \frac{\sum_{\boldsymbol{s}\in\cc{D}_n} C_{\boldsymbol{s}} G^2_m\paren{\boldsymbol{s}} }{\sqrt{\sum_{\boldsymbol{s}\in\cc{D}_n}C^2_{\boldsymbol{s}}}} + o\paren{1}.
\end{equation*}
Simply put, there is no need to estimate $\rho$ using the profile LIF loss, for this particular scenario. For an arbitrarily chosen $\rho$, $\phi_0\rho^{-2\nu}_0$ can indeed be estimated by $\hat{\phi}_{n,\cc{B}}\paren{\rho}\rho^{-2\nu}$. The estimator in Eq. \eqref{AnderesAlg} is in fact identical to the one proposed by Anderes \cite{anderes2010consistent}. He also investigated its fixed-domain asymptotic properties for regular lattices employing some techniques for studying the quadratic variation of stationary spatial Gaussian processes
\end{rem}

The first main result of this section states that for appropriately chosen preconditioning order $m$, regardless of the choice of $\cc{B}$ and $\rho$, $\hat{\phi}_{n,\cc{B}}\paren{\rho}\rho^{-2\nu}$ is a $\sqrt{n}$-consistent estimate of $\phi_0\rho^{-2\nu}_0$.

\begin{thm}\label{CnstncyLocInvFreeThm}
Let $G$ be observed on a lattice $\cc{D}_n$ satisfying Assumption \ref{RegCondLattice}. Suppose that the preconditioning order $m$ satisfies $m\geq \paren{\nu+d/2}$. For a given binning scheme $\cc{B}$ of $\cc{D}_n$, there are bounded positive scalars $C_{\cc{B}}$ and $n_0$, depending on $m, d, \nu, \Theta_0, \cc{B}$ and the geometric structure of $\cc{D}_n$, such that 
\begin{equation}\label{CnstncyLIFRate}
\bb{P}\paren{ \sup_{\rho\in\Theta_0}\abs{ \frac{\hat{\phi}_{n,\cc{B}}\paren{\rho}\rho^{-2\nu}}{\phi_0\rho^{-2\nu}_0}-1 }\geq C_{\cc{B}}\sqrt{\frac{\log n}{n}} }\leq \frac{1}{n},\quad\forall\; n\geq n_0.
\end{equation}
\end{thm} 

Theorem \ref{CnstncyLocInvFreeThm} establishes (high probability) uniform concentration of $\hat{\phi}_{n,\cc{B}}\paren{\rho}\rho^{-2\nu}$ around $\phi_0\rho^{-2\nu}_0$ in a small ball of radius $\cc{O}(\sqrt{n^{-1}\log n})$. The $\sqrt{n}$-consistency of the global (or local) maximizers of the LIF objective function is an immediate consequence of Theorem \ref{CnstncyLocInvFreeThm}. It is known that an analogous bound as in Eq. \eqref{CnstncyLIFRate} holds for the MLE, regardless of how $m$ is chosen. Namely, the MLE is $\sqrt{n}$-consistent even for raw data, $m=0$. Thus Theorem \ref{CnstncyLocInvFreeThm} implicitly says that, for sufficiently decorrelated samples, there are surrogates losses that can be optimized considerably faster than the log-likelihood on a wide range of irregular grids, and without sacrificing the asymptotic efficiency.

In the case that $\nu$ is either known or can be rather precisely estimated, Theorem \ref{CnstncyLocInvFreeThm} gives a straightforward way of choosing $m$. For instance the choice of $m = \lceil \nu+1 \rceil$ is sufficient when $G$ is observed within a two dimensional region. Recall from Remark \ref{PreCondRegLat} that for the regular lattices, if $m'$ represents the number of times the Laplace operator is applied to the data, then the transformed process is a preconditioned GP of order $2m'$. Thus for Gaussian processes observed on $d$-dimensional regular lattices, $m = 2m'$ and so $m'$ should not be smaller than $\nu/2+d/4$.

\begin{rem}
For pedagogical reasons, we outline a brief sketch of the proof of Theorem \ref{CnstncyLocInvFreeThm}; full details are postponed to Section \ref{Proofs}. The bias-variance decomposition plays a canonical role in our analysis. In particular,
\begin{equation*}
\sup_{\rho\in\Theta_0}\abs{ \frac{\hat{\phi}_{n,\cc{B}}\paren{\rho}\rho^{-2\nu}}{\phi_0\rho^{-2\nu}_0}-1 } \leq P_1+P_2\coloneqq \sup_{\rho\in\Theta_0}\abs{\frac{\bb{E} \hat{\phi}_{n,\cc{B}}\paren{\rho}\rho^{-2\nu}}{\phi_0\rho^{-2\nu}_0} - 1 }+\sup_{\rho\in\Theta_0} \abs{\frac{\hat{\phi}_{n,\cc{B}}\paren{\rho}\rho^{-2\nu}-\bb{E}\hat{\phi}_{n,\cc{B}}\paren{\rho}\rho^{-2\nu}}{\phi_0\rho^{-2\nu}_0}}.
\end{equation*}
We show that $P_1 = o\paren{1/\sqrt{n}}$ by employing a novel approach to investigate the large sample properties of the eigenvalues of $K^{\cc{B}}_{n,m}\paren{\rho}$. On the other hand, $P_2$ is in fact the supremum of a chi-squared process over $\Theta_0$. Employing the classical chaining argument it can be shown that $P_2$ is of order $\sqrt{n^{-1}\log n}$, with high probability. We refer the reader to Appendix \ref{AppendixCovMatrixIrregLat} for further details.
\end{rem}

\begin{cor}\label{ConstncyStatnrypts}
Under the same notation and conditions as in Theorem \ref{CnstncyLocInvFreeThm}, the following inequality holds for any stationary point $(\hat{\phi}_{n,\cc{B}}, \hat{\rho}_{n,\cc{B}})$ of the LIF loss \eqref{LocACSOptProb}. 
\begin{equation*}
\bb{P}\paren{ \abs{ \frac{\hat{\phi}_{n,\cc{B}}\hat{\rho}^{-2\nu}_{n,\cc{B}}}{\phi_0\rho^{-2\nu}_0}-1}\geq C_{\cc{B}}\sqrt{\frac{\log n}{n}} } \leq \frac{1}{n},\quad\mbox{as}\; n\rightarrow\infty.
\end{equation*}
\end{cor}
It has been argued in \cite{kaufman2013role} that estimating $\rho_0$ can improve the statistical performance, especially for small $n$. The first advantage of Corollary \ref{ConstncyStatnrypts} is that it establishes the consistency of an arbitrary stationary point of the LIF objective function. Allowing the range parameter to be estimated in a large bounded space, which is crucial in practice, is another advantage of Corollary \ref{ConstncyStatnrypts}. 

Remark \ref{BlcDiagApproxKnmRem} may induce a false impression that the convergence rate of $\hat{\phi}_{n,\cc{B}}\hat{\rho}^{-2\nu}_{n,\cc{B}}$ is determined by how well the covariance matrix of the preconditioned samples $K_{n,m}\paren{\rho}$ can be approximated by $K^{\cc{B}}_{n,m}\paren{\rho}$. Yet, Corollary \ref{ConstncyStatnrypts} discloses the somewhat surprising fact that the LIF algorithm is $\sqrt{n}$-consistent, regardless of the choice of $\cc{B}$. The fast enough decay rate of the off-diagonal entries of $K_{n,m}\paren{\rho}$ is a heuristic explanation for the $\sqrt{n}$-consistency of the LIF estimator. In other words since $K_{n,m}\paren{\rho}$ can be suitably approximated by any block diagonal matrix induced by a partitioning scheme, splitting the preconditioned data into different bins does not affect the convergence rate of the LIF estimate. However the influence of the partitioning scheme may become more pronounced in practical situations with moderate sample sizes.

\begin{rem}\label{Rem4.3}
It has been discussed in \cite{anitescu2017inversion} that the global solution of the IF optimization problem, in Eq. \eqref{ACSOptProb}, has the same convergence rate as the MLE, when the covariance matrix of the preconditioned samples has a uniformly bounded condition number over $\Theta_0$. Such a restriction on the covariance matrix rarely holds in practice, unless under some strong conditions on the spectral density and the geometric structure of $\cc{D}_n$ (see \cite{stein2012interpolation}). However Corollary \ref{ConstncyStatnrypts} requires much weaker restrictions on the covariance matrix. Two sufficient conditions on $K^{\cc{B}}_{n,m}\paren{\cdot}$ can be spotted by going through our proof of Theorem \ref{ConstncyStatnrypts}.
\begin{enumerate}
\item The largest eigenvalue of $K^{\cc{B}}_{n,m}\paren{\cdot}$ should be uniformly bounded over $\Theta_0$. Namely, 
\begin{equation*}
\max_{\rho\in\Theta_0} \OpNorm{K^{\cc{B}}_{n,m}\paren{\cdot}}{2}{2} \asymp 1.
\end{equation*}
\item $K^{\cc{B}}_{n,m}\paren{\rho}$ must have $\cc{O}\paren{n}$ non-negligible positive eigenvalues, for any $\rho\in\Theta_0$. That is,
\begin{equation*}
\inf_{\rho\in\Theta_0} \LpNorm{K^{\cc{B}}_{n,m}\paren{\rho}}{2} \asymp \sqrt{n}.
\end{equation*}
\end{enumerate}
Note that the above conditions do not rule out the existence of near zero eigenvalues and so the conditions number is still allowed to diverge as $n$ tends to infinity. In this regard, our asymptotic understanding expands the applicability of inversion-free techniques.
\end{rem}

\noindent Now we establish the asymptotic distribution of all the stationary points of the LIF loss function.

\begin{thm}\label{AsympNormLocInvFreeThm}
Under the same notation and conditions as in Theorem \ref{CnstncyLocInvFreeThm}, there exists a bounded sequence $\sigma_{n,\cc{B}}$ such that for any stationary point $(\hat{\phi}_{n,\cc{B}}, \hat{\rho}_{n,\cc{B}})$ of the LIF loss 
\begin{equation*}
\frac{\sqrt{n}}{\sigma_{n,\cc{B}}}\paren{\frac{\hat{\phi}_{n,\cc{B}}\hat{\rho}^{-2\nu}_{n,\cc{B}}}{\phi_0\rho^{-2\nu}_0}-1}\cp{d} \cc{N}\paren{0,1}.
\end{equation*}
\end{thm}
Theorem \ref{AsympNormLocInvFreeThm} formulates the asymptotic distribution of the LIF estimator for joint estimation of $\phi_0$ and $\rho_0$. To our knowledge, for the MLE, such a result has only appeared in \cite{kaufman2013role}. Note that unlike the full or tapered MLE, in which $\sigma_{n,\cc{B}} = \sqrt{2}$ (see Theorem $2$ of \cite{wang2011fixed}), here $m,d,\nu$, the geometric structure and the portioning scheme of $\cc{D}_n$ also affect the asymptotic standard deviation. We could not obtain a simple closed form expression for $\sigma_{n,\cc{B}}$. A  complicated expression is stated in the proof of Theorem \ref{AsympNormLocInvFreeThm}.

\begin{rem}
We conclude this section with a succinct discussion of the role of $\Theta_0$ in the optimization problem presented in Eq. \eqref{LIFAltForm}. The main results in this section can be generalized to the following constrained optimization problem
\begin{equation*}
\paren{\hat{\phi}_{n,\cc{B}}, \hat{\rho}_{n,\cc{B}}} = \argmax_{\phi>0, \rho\in\Theta_n} \paren{\phi Y^\top_m K^{\cc{B}}_{n,m}\paren{\rho} Y_m - \frac{\phi^2}{2}\LpNorm{K^{\cc{B}}_{n,m}\paren{\rho}}{2}^2}.
\end{equation*}
Here, $\set{\Theta_n}^{\infty}_{n=1}$ represents a class of nested subsets of $\paren{0,\infty}$, i.e., $\Theta_p\subseteq \Theta_q$ $\forall\;p\leq q$, whose diameter grows polynomially in $n$. Namely, $\diam\paren{\Theta_n}\lesssim n^\zeta$ for an arbitrary bounded scalar $\zeta \geq 0$. As sample size grows, such a formulation of the LIF algorithm demands less restrictive assumptions on the range parameter and bears more resemblance to an unconstrained maximization problem.  
\end{rem}

\section{Simulation studies}\label{SimulStud}

This section is devoted to appraising the computational and statistical properties of the LIF algorithm on synthetic stationary Gaussian process data\footnote{See Section 3.5 of \cite{keshavarz2017detection} for more complete numerical studies.}. The purpose of our study is two-fold: investigating the scalability and efficiency of the proposed method in large datasets, as well as corroborating the fixed-domain asymptotic theory presented in Section \ref{FxddomAnal}. We consider two different scenarios regarding the sample size $n$. In moderate-size settings which are designed for constructing confidence intervals of unknown parameters through independent experiments, $n = 10^4$. Moreover, large-scale simulations with $n = 2.5\times 10^5$ are conducted to study the numerical capabilities of the LIF algorithm, particularly when the exact and approximate evaluation of the likelihood function are extremely challenging. The computations have been performed on a UM Flux Ivy bridge compute node with $20$ cores (Intel Xeon processor) and $3$ GB memory per core. For expediting execution times of the simulations (up to $100$ times), the LIF algorithm has been implemented in \emph{C$++$} and \emph{R} using the \emph{RcppParallel}\footnote{https://cran.r-project.org/web/packages/RcppParallel/index.html} package.

Throughout this section $G$ is a real-valued stationary Matern GP observed on an irregularly spaced lattice $\cc{D}_n$. We consider two cases of isotropy and geometric anisotropy for the covariance function. For circumventing the obstacles of computing the Cholesky factorization of the covariance matrix, spectral methods are used for constructing $G$ on $\cc{D}_n$ \cite{keshavarz2016consistency}. We now concisely describe the geometry of $\cc{D}_n$. Let $\cc{D} = \brac{0,T}^2$ be a square of side-length $T$. $\cc{D}_n$ is a two dimensional randomly perturbed lattice of size $n = N^2$ if there exists a non-negative $\delta$, representing the perturbation parameter, such that for any point $\boldsymbol{t}\in\cc{D}_n$, there are a corresponding point in the regular lattice $\boldsymbol{s}\in\set{T/N,2T/N,\ldots,T}^2$ and a randomly chosen $\boldsymbol{p}\in\brac{-T/N,T/N}^2$ (with uniform distribution) for which $\boldsymbol{t} = \boldsymbol{s} + \delta\boldsymbol{p}$. The scalar quantity $\delta$ controls the amount of irregularity in the set of sampling locations.

Partitioning $\cc{D}_n$ into $b_n$ bins is necessary for implementing the LIF algorithm. For brevity the bins are labelled $1$ to $b_n$. In the following, we elucidate three schemes for constructing the bins.
\begin{enumerate}
\item \emph{Uniformly Chosen (UC) bins}: Any $\boldsymbol{s}\in\cc{D}_n$ is randomly assigned to a bin in $\set{1,\ldots,b_n}$ with a uniform distribution. So the average size of all bins are the same.
\item \emph{Non Uniformly Chosen (NUC) bins}: The points in $\cc{D}_n$ are independently assigned to bins labelled with $\set{1,\ldots,b_n}$, according to a non-uniform distribution $Q$. Throughout this section, we assume that $Q$ is proportional to $\brac{1,\ldots,1,2,\ldots,2}^\top$. For instance in the case that $b_n = 4$, an arbitrary $\brac{1/6,1/6,1/3,1/3}^\top$. Thus on average half of the bins are twice a big as the other half.
\item \emph{Rectangular bins}: $\cc{D}_n$ is segregated into $b_n$ rectangular subregions and all the points in each subregion belong to the same bin.
\end{enumerate}
Figure \ref{Fig:Fig2} illustrates the three methods of constructing subgroups for a randomly perturbed lattice of size $100$ and $\delta = 0.5$. For illustration, $b_n$ is chosen to be $4$ for each scenario in Figure \ref{Fig:Fig2}.

\begin{figure}
\centering
\subimport{/}{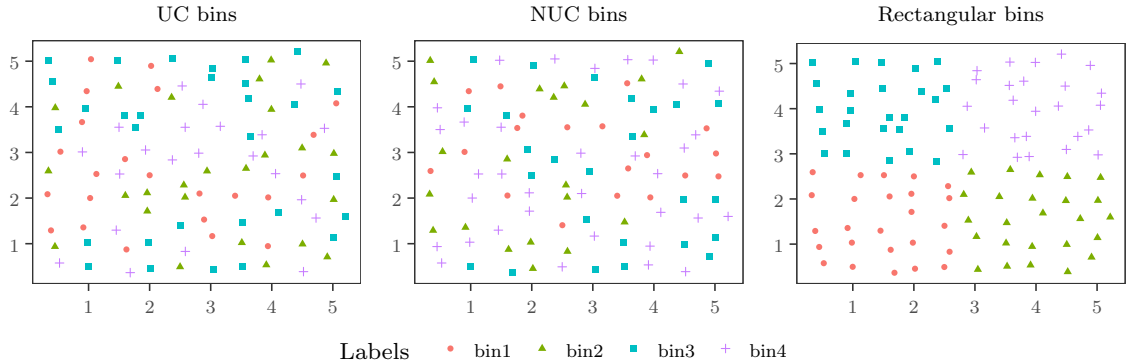}  
\vspace{-.11in}
\caption{Three partitioning schemes of $10^2$ points of a perturbed lattices on $\cc{D} = \brac{0,5}^2$ with $\delta = 0.5$}
\label{Fig:Fig2}
\end{figure}

We present three sets of simulation studies to assess the performance of the LIF algorithm. In all the experiments, $G$ is a Matern GP observed on a randomly perturbed lattice. The developed asymptotic insight in Section \ref{FxddomAnal} is rather limited, as it is restricted to isotropic GPs. Therefore we present two sets of numerical studies for evaluating the performance of our proposed method for the geometric anisotropic processes (multiple range parameters). Note that the claim in Remark \ref{Rem4.1} is not valid for geometric anistotropic GPs. In other words the profile LIF loss directly depends on range parameters and therefore needs to be numerically maximized. The L-BFGS-B (limited-memory BFGS with bound constraints \cite{byrd1995limited}) algorithm is utilized for maximizing the profile LIF loss. The finite difference approximation with step size $10^{-3}$ is used for computing the gradient. We stop the optimization procedure if either the relative change in the objective function is below $10^{-5}$ or it reaches $50$ iterations.

\subsection{Moderate-scale simulations for isotropic GPs} 

In all the experiments of this section, $\cc{D} = \brac{0,5}^2$ and $\cc{D}_n$ is a perturbed lattice with $\delta \in \set{1,3}$ and $100^2$ points, i.e. $n = 10^4$. We generate $100$ realizations of an isotropic Matern GP $G$ with parameters $\phi_0 = 1, \rho_0 = 5$, and $\nu = 0.5$ on $100$ independent realizations of $\cc{D}_n$. The preconditioning order $m = 2$ is chosen for satisfying the condition $m\geq \nu+d/2$ in the statement of Theorems \ref{CnstncyLocInvFreeThm} and \ref{AsympNormLocInvFreeThm}. Furthermore for any $\boldsymbol{s}\in\cc{D}_n$, $\cc{N}_m\paren{\boldsymbol{s}}$ consists of the seven closest points in $\cc{D}_n$ to $\boldsymbol{s}$ ($\abs{\cc{N}_m\paren{\boldsymbol{s}}} = 7$). For any $\boldsymbol{s}\in\cc{D}_n$, we adopt the following procedure for choosing the preconditioning coefficients $\set{a_{m,\boldsymbol{s}}\paren{\boldsymbol{t}}:\; \boldsymbol{t}\in\cc{N}_m\paren{\boldsymbol{s}} }$.
\begin{enumerate}
\item Let $a_{m,\boldsymbol{s}}\paren{\boldsymbol{s}} = 1$ and solve the system of linear equations introduced in the second condition of Definition \ref{DecorFltNonReglat} to compute $\set{a_{m,\boldsymbol{s}}\paren{\boldsymbol{t}}:\; \boldsymbol{t}\in\cc{N}_m\paren{\boldsymbol{s}}\setminus\boldsymbol{s} }$.
\item Each coefficient is normalized by dividing by the quantity $\sqrt{\sum_{\boldsymbol{t}\in \cc{N}_m\paren{\boldsymbol{s}}} a^2_{m,\boldsymbol{s}}\paren{\boldsymbol{t}}}$.
\end{enumerate}
The goal is to estimate  $\phi_0\rho^{-2\nu}_0$, which has the central role in the asymptotic analysis in Section \ref{FxddomAnal}. According to Theorems \ref{CnstncyLocInvFreeThm} and \ref{AsympNormLocInvFreeThm}, estimating $\rho_0$ is not necessary for the isotropic Matern covariance functions. In other words, $\rho$ can be fixed in the optimization problem in Eq. \eqref{LocACSOptProb}. Therefore we select $\rho = 10$ and maximize the LIF function with respect to $\phi$, i.e. $\hat{\rho}_{n,\cc{B}} = 10$. For each realization of $G$, $\hat{\phi}_{n,\cc{B}}$ is evaluated for $b_n\in\set{1,2,4,8,16}$ and three partitioning approaches UC, NUC, and rectangular. For brevity define  
\begin{equation}\label{StandEstimParam}
\hat{\xi}_{n,\cc{B}} = \frac{\hat{\phi}_{n,\cc{B}}\hat{\rho}^{-2\nu}_{n,\cc{B}}}{\phi_0\rho^{-2\nu}_0}.
\end{equation}
Theorem \ref{AsympNormLocInvFreeThm} suggests that $\hat{\xi}_{n,\cc{B}}$ is normally distributed centered at $1$. Figures \ref{Fig:Fig3} and \ref{Fig:Fig4} respectively exhibit the histogram of $\hat{\xi}_{n,\cc{B}}$ for the cases of $\delta = 1$ and $3$, different choices of $b_n$ and partitioning schemes. Each plot also shows a kernel density estimate (KDE) of the histogram for a simpler comparison with the normal distribution. Table \ref{Table1} presents the mean and standard deviation of each histogram in Figures \ref{Fig:Fig3} and \ref{Fig:Fig4}. According to Table \ref{Table1}, for different values of $\delta,b_n$ and bin shapes, $\hat{\xi}_{n,\cc{B}}$ is concentrated around $1$ with the bias of order $10^{-3}$ and the standard deviation near $0.04$, with a bell shaped density.

\begin{figure}[!htbp]
\centering
\subimport{/}{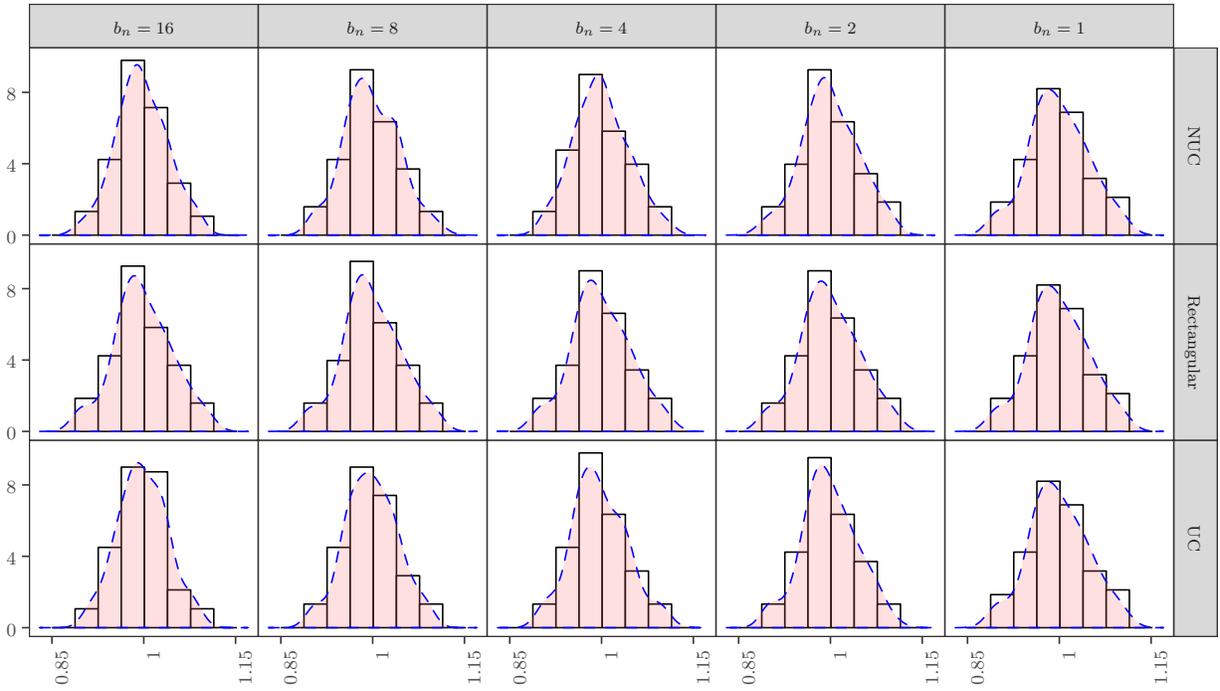}  
\vspace{-.23in}
\caption{The histogram of $\hat{\xi}_{n,\cc{B}}$ with $m=2$, $b_n=1,2,4,8,16$ and $3$ binning schemes for isotropic Matern GP with $\paren{\phi_0,\rho_0,\nu} = \paren{1,5,0.5}$ observed on a perturbed lattice with $\delta = 1$ and $n=10^4$.}
\label{Fig:Fig3}
\end{figure}

\begin{figure}[!htbp]
\centering
\subimport{/}{Histogram_NU05_Rho55_m2_Delta3.tikz}  
\vspace{-.23in}
\caption{The histogram of $\hat{\xi}_{n,\cc{B}}$ with $m=2$, $b_n=1,2,4,8,16$ and $3$ binning schemes for isotropic Matern GP with $\paren{\phi_0,\rho_0,\nu} = \paren{1,5,0.5}$ observed on a perturbed lattice with $\delta = 3$ and $n=10^4$.}
\label{Fig:Fig4}
\end{figure}

\vspace{3mm}
\begin{center} 
\begin{adjustbox}{width=0.98\columnwidth}
\begin{tabular}{cc|c|c|c|c|c|}
\cline{3-7} 
& & $b_n = 16$ & $b_n = 8$ & $b_n = 4$ & $b_n = 2$ & $b_n = 1$ \\ \cline{1-7}
\multicolumn{1}{|c}{\multirow{6}{*}{$\delta = 1$} } &
\multicolumn{1}{|c|}{\multirow{2}{*}{NUC}}  & 
\multicolumn{1}{|c|}{$\bb{E} \hat{\xi}_{n,\cc{B}} = 0.9968$} & $\bb{E} \hat{\xi}_{n,\cc{B}} = 0.9979$ & $\bb{E} \hat{\xi}_{n,\cc{B}} = 0.9993$ & $\bb{E} \hat{\xi}_{n,\cc{B}} = 0.9993$ & $\bb{E} \hat{\xi}_{n,\cc{B}} = 0.9990$  \\ 
\multicolumn{1}{|c|}{} & &
\multicolumn{1}{|c|}{$\sd \hat{\xi}_{n,\cc{B}} = 0.0417$} & $\sd \hat{\xi}_{n,\cc{B}} = 0.0442$ & $\sd \hat{\xi}_{n,\cc{B}} = 0.0448$ & $\sd \hat{\xi}_{n,\cc{B}} = 0.0459$ & $\sd \hat{\xi}_{n,\cc{B}} = 0.0481$  \\ \cline{2-7} 
\multicolumn{1}{|c|}{} &
\multicolumn{1}{|c|}{\multirow{2}{*}{Rectangular}}  & 
\multicolumn{1}{|c|}{$\bb{E} \hat{\xi}_{n,\cc{B}} = 0.9989$} & $\bb{E} \hat{\xi}_{n,\cc{B}} = 0.9990$ & $\bb{E} \hat{\xi}_{n,\cc{B}} = 0.9991$ & $\bb{E} \hat{\xi}_{n,\cc{B}} = 0.9992$ & $\bb{E} \hat{\xi}_{n,\cc{B}} = 0.9990$  \\ 
\multicolumn{1}{|c|}{} & &
\multicolumn{1}{|c|}{$\sd \hat{\xi}_{n,\cc{B}} = 0.0475$} & $\sd \hat{\xi}_{n,\cc{B}} = 0.0476$ & $\sd \hat{\xi}_{n,\cc{B}} = 0.0477$ & $\sd \hat{\xi}_{n,\cc{B}} = 0.0478$ & $\sd \hat{\xi}_{n,\cc{B}} = 0.0481$  \\ \cline{2-7} 
\multicolumn{1}{|c|}{} &
\multicolumn{1}{|c|}{\multirow{2}{*}{UC}}  & 
\multicolumn{1}{|c|}{$\bb{E} \hat{\xi}_{n,\cc{B}} = 0.9980$} & $\bb{E} \hat{\xi}_{n,\cc{B}} = 0.9980$ & $\bb{E} \hat{\xi}_{n,\cc{B}} = 0.9965$ & $\bb{E} \hat{\xi}_{n,\cc{B}} = 0.9984$ & $\bb{E} \hat{\xi}_{n,\cc{B}} = 0.9990$  \\ 
\multicolumn{1}{|c|}{} & &
\multicolumn{1}{|c|}{$\sd \hat{\xi}_{n,\cc{B}} = 0.0403$} & $\sd \hat{\xi}_{n,\cc{B}} = 0.0424$ & $\sd \hat{\xi}_{n,\cc{B}} = 0.0443$ & $\sd \hat{\xi}_{n,\cc{B}} = 0.0450$ & $\sd \hat{\xi}_{n,\cc{B}} = 0.0481$  \\ \cline{1-7} 
\multicolumn{1}{|c}{\multirow{6}{*}{$\delta = 3$} } &
\multicolumn{1}{|c|}{\multirow{2}{*}{NUC}}  & 
\multicolumn{1}{|c|}{$\bb{E} \hat{\xi}_{n,\cc{B}} = 0.9953$} & $\bb{E} \hat{\xi}_{n,\cc{B}} = 0.9962$ & $\bb{E} \hat{\xi}_{n,\cc{B}} = 0.9962$ & $\bb{E} \hat{\xi}_{n,\cc{B}} = 0.9965$ & $\bb{E} \hat{\xi}_{n,\cc{B}} = 0.9955$  \\ 
\multicolumn{1}{|c|}{} & &
\multicolumn{1}{|c|}{$\sd \hat{\xi}_{n,\cc{B}} = 0.0463$} & $\sd \hat{\xi}_{n,\cc{B}} = 0.0472$ & $\sd \hat{\xi}_{n,\cc{B}} = 0.0500$ & $\sd \hat{\xi}_{n,\cc{B}} = 0.0524$ & $\sd \hat{\xi}_{n,\cc{B}} = 0.0534$  \\ \cline{2-7} 
\multicolumn{1}{|c|}{} &
\multicolumn{1}{|c|}{\multirow{2}{*}{Rectangular}}  & 
\multicolumn{1}{|c|}{$\bb{E} \hat{\xi}_{n,\cc{B}} = 0.9955$} & $\bb{E} \hat{\xi}_{n,\cc{B}} = 0.9953$ & $\bb{E} \hat{\xi}_{n,\cc{B}} = 0.9954$ & $\bb{E} \hat{\xi}_{n,\cc{B}} = 0.9954$ & $\bb{E} \hat{\xi}_{n,\cc{B}} = 0.9955$  \\ 
\multicolumn{1}{|c|}{} & &
\multicolumn{1}{|c|}{$ \sd_{n,\cc{B}} = 0.0536$} & $\sd \hat{\xi}_{n,\cc{B}} = 0.0536$ & $\sd \hat{\xi}_{n,\cc{B}} = 0.0534$ & $\sd \hat{\xi}_{n,\cc{B}} = 0.0535$ & $\sd \hat{\xi}_{n,\cc{B}} = 0.0534$  \\ \cline{2-7} 
\multicolumn{1}{|c|}{} &
\multicolumn{1}{|c|}{\multirow{2}{*}{UC}}  & 
\multicolumn{1}{|c|}{$\bb{E} \hat{\xi}_{n,\cc{B}} = 0.9966$} & $\bb{E} \hat{\xi}_{n,\cc{B}} = 0.9954$ & $\bb{E} \hat{\xi}_{n,\cc{B}} = 0.9954$ & $\bb{E} \hat{\xi}_{n,\cc{B}} = 0.9952$ & $\bb{E} \hat{\xi}_{n,\cc{B}} = 0.9955$  \\ 
\multicolumn{1}{|c|}{} & &
\multicolumn{1}{|c|}{$\sd \hat{\xi}_{n,\cc{B}} = 0.0456$} & $\sd \hat{\xi}_{n,\cc{B}} = 0.0465$ & $\sd \hat{\xi}_{n,\cc{B}} = 0.0496$ & $\sd \hat{\xi}_{n,\cc{B}} = 0.0513$ & $\sd \hat{\xi}_{n,\cc{B}} = 0.0534$  \\ \cline{1-7}
\end{tabular}
\end{adjustbox}
\captionof{table}{The mean and standard deviation of $\hat{\xi}_{n,\cc{B}}$ exhibited in histograms in Figures \ref{Fig:Fig3} and \ref{Fig:Fig4}.}
\label{Table1}
\end{center}  

\vspace{2mm}

Next we conduct the same experiment on a smoother isotropic Matern GP with $\phi_0 = 1, \rho_0 = 2.5$, and $\nu = 1$. We seek to gauge the sensitivity of our estimation algorithm to the preconditioning order $m$ by considering two cases of $m=2$ and $3$. Notice that the condition $m\geq \nu+d/2$ holds for both choices of $m$. However evaluating the LIF loss is a more difficult task for $m = 3$ because of dealing with larger conditioning sets ($\abs{\cc{N}_3\paren{\boldsymbol{s}}} = 11$ for any $\boldsymbol{s}\in\cc{D}_n$). Table \ref{Table2} summarizes the mean and standard deviation of $\hat{\xi}_{n,\cc{B}}$ for the different choices of $m, b_n, \delta$, and partitioning schemes.

\vspace{3mm}
\begin{center} 
\begin{adjustbox}{width=0.95\columnwidth}
\begin{tabular}{ccc|c|c|c|c|c|}
\cline{4-8}
& & & $b_n = 16$ & $b_n = 8$ & $b_n = 4$ & $b_n = 2$ & $b_n = 1$ \\ \cline{1-8}
\multicolumn{1}{|c}{\multirow{12}{*}{$m=2$} } &
\multicolumn{1}{|c}{\multirow{6}{*}{$\delta = 1$} } &
\multicolumn{1}{|c}{\multirow{2}{*}{NUC}}  & 
\multicolumn{1}{|c|}{$\bb{E} \hat{\xi}_{n,\cc{B}} = 1.0465$} & $\bb{E} \hat{\xi}_{n,\cc{B}} = 1.0459$ & $\bb{E} \hat{\xi}_{n,\cc{B}} = 1.0478$ & $\bb{E} \hat{\xi}_{n,\cc{B}} = 1.0481$ & $\bb{E} \hat{\xi}_{n,\cc{B}} = 1.0489$  \\ 
\multicolumn{1}{|c}{} & \multicolumn{1}{|c}{} & \multicolumn{1}{|c}{} &
\multicolumn{1}{|c|}{$\sd \hat{\xi}_{n,\cc{B}} = 0.3188$} & $\sd \hat{\xi}_{n,\cc{B}} = 0.3222$ & $\sd \hat{\xi}_{n,\cc{B}} = 0.3315$ & $\sd \hat{\xi}_{n,\cc{B}} = 0.3439$ & $\sd \hat{\xi}_{n,\cc{B}} = 0.3555$  \\ \cline{3-8} 
\multicolumn{1}{|c}{} & \multicolumn{1}{|c}{} & 
\multicolumn{1}{|c}{\multirow{2}{*}{Rectangular}}  & 
\multicolumn{1}{|c|}{$\bb{E} \hat{\xi}_{n,\cc{B}} = 1.0491$} & $\bb{E} \hat{\xi}_{n,\cc{B}} = 1.0489$ & $\bb{E} \hat{\xi}_{n,\cc{B}} = 1.0487$ & $\bb{E} \hat{\xi}_{n,\cc{B}} = 1.0491$ & $\bb{E} \hat{\xi}_{n,\cc{B}} = 1.04889$  \\ 
\multicolumn{1}{|c}{} & \multicolumn{1}{|c}{} &  \multicolumn{1}{|c}{} &
\multicolumn{1}{|c|}{$\sd \hat{\xi}_{n,\cc{B}} = 0.3548$} & $\sd \hat{\xi}_{n,\cc{B}} = 0.3550$ & $\sd \hat{\xi}_{n,\cc{B}} = 0.3554$ & $\sd \hat{\xi}_{n,\cc{B}} = 0.3556$ & $\sd \hat{\xi}_{n,\cc{B}} = 0.3555$  \\ \cline{3-8} 
\multicolumn{1}{|c}{} & \multicolumn{1}{|c}{}  &
\multicolumn{1}{|c}{\multirow{2}{*}{UC}}  & 
\multicolumn{1}{|c|}{$\bb{E} \hat{\xi}_{n,\cc{B}} = 1.0458$} & $\bb{E} \hat{\xi}_{n,\cc{B}} = 1.0464$ & $\bb{E} \hat{\xi}_{n,\cc{B}} = 1.0470$ & $\bb{E} \hat{\xi}_{n,\cc{B}} = 1.0488$ & $\bb{E} \hat{\xi}_{n,\cc{B}} = 1.0489$  \\ 
\multicolumn{1}{|c}{} & \multicolumn{1}{|c}{} & \multicolumn{1}{|c}{} &
\multicolumn{1}{|c|}{$\sd \hat{\xi}_{n,\cc{B}} = 0.3173$} & $\sd \hat{\xi}_{n,\cc{B}} = 0.3215$ & $\sd \hat{\xi}_{n,\cc{B}} = 0.3289$ & $\sd \hat{\xi}_{n,\cc{B}} = 0.3418$ & $\sd \hat{\xi}_{n,\cc{B}} = 0.3555$  \\ \cline{2-8} 
\multicolumn{1}{|c}{} &
\multicolumn{1}{|c}{\multirow{6}{*}{$\delta = 3$} } &
\multicolumn{1}{|c}{\multirow{2}{*}{NUC}}  & 
\multicolumn{1}{|c|}{$\bb{E} \hat{\xi}_{n,\cc{B}} = 1.0302$} & $\bb{E} \hat{\xi}_{n,\cc{B}} = 1.0315$ & $\bb{E} \hat{\xi}_{n,\cc{B}} = 1.0329$ & $\bb{E} \hat{\xi}_{n,\cc{B}} = 1.0366$ & $\bb{E} \hat{\xi}_{n,\cc{B}} = 1.0393$  \\ 
\multicolumn{1}{|c}{} & \multicolumn{1}{|c}{} & \multicolumn{1}{|c}{} &
\multicolumn{1}{|c|}{$\sd \hat{\xi}_{n,\cc{B}} = 0.3790$} & $\sd \hat{\xi}_{n,\cc{B}} = 0.3847$ & $\sd \hat{\xi}_{n,\cc{B}} = 0.4926$ & $\sd \hat{\xi}_{n,\cc{B}} = 0.4075$ & $\sd \hat{\xi}_{n,\cc{B}} = 0.4105$  \\ \cline{3-8} 
\multicolumn{1}{|c}{} & \multicolumn{1}{|c}{} &
\multicolumn{1}{|c}{\multirow{2}{*}{Rectangular}}  & 
\multicolumn{1}{|c|}{$\bb{E} \hat{\xi}_{n,\cc{B}} = 1.0396$} & $\bb{E} \hat{\xi}_{n,\cc{B}} = 1.0392$ & $\bb{E} \hat{\xi}_{n,\cc{B}} = 1.0393$ & $\bb{E} \hat{\xi}_{n,\cc{B}} = 1.0394$ & $\bb{E} \hat{\xi}_{n,\cc{B}} = 1.0393$  \\ 
\multicolumn{1}{|c}{} & \multicolumn{1}{|c}{} & \multicolumn{1}{|c}{} &
\multicolumn{1}{|c|}{$ \sd_{n,\cc{B}} = 0.4196$} & $\sd \hat{\xi}_{n,\cc{B}} = 0.4196$ & $\sd \hat{\xi}_{n,\cc{B}} = 0.4201$ & $\sd \hat{\xi}_{n,\cc{B}} = 0.4204$ & $\sd \hat{\xi}_{n,\cc{B}} = 0.4105$  \\ \cline{3-8} 
\multicolumn{1}{|c}{} & \multicolumn{1}{|c}{} &
\multicolumn{1}{|c}{\multirow{2}{*}{UC}}  & 
\multicolumn{1}{|c|}{$\bb{E} \hat{\xi}_{n,\cc{B}} = 1.0304$} & $\bb{E} \hat{\xi}_{n,\cc{B}} = 1.0323$ & $\bb{E} \hat{\xi}_{n,\cc{B}} = 1.0337$ & $\bb{E} \hat{\xi}_{n,\cc{B}} = 1.0363$ & $\bb{E} \hat{\xi}_{n,\cc{B}} = 1.0393$  \\ 
\multicolumn{1}{|c}{} & \multicolumn{1}{|c}{} & \multicolumn{1}{|c}{} &
\multicolumn{1}{|c|}{$\sd \hat{\xi}_{n,\cc{B}} = 0.3789$} & $\sd \hat{\xi}_{n,\cc{B}} = 0.3846$ & $\sd \hat{\xi}_{n,\cc{B}} = 0.3927$ & $\sd \hat{\xi}_{n,\cc{B}} = 0.4048$ & $\sd \hat{\xi}_{n,\cc{B}} = 0.4105$  \\ \cline{1-8}
\multicolumn{1}{|c}{\multirow{12}{*}{$m=3$} } &
\multicolumn{1}{|c}{\multirow{6}{*}{$\delta = 1$} } &
\multicolumn{1}{|c}{\multirow{2}{*}{NUC}}  & 
\multicolumn{1}{|c|}{$\bb{E} \hat{\xi}_{n,\cc{B}} = 1.0237$} & $\bb{E} \hat{\xi}_{n,\cc{B}} = 1.0237$ & $\bb{E} \hat{\xi}_{n,\cc{B}} = 1.0262$ & $\bb{E} \hat{\xi}_{n,\cc{B}} = 1.0279$ & $\bb{E} \hat{\xi}_{n,\cc{B}} = 1.0315$  \\ 
\multicolumn{1}{|c}{} & \multicolumn{1}{|c}{} & \multicolumn{1}{|c}{} &
\multicolumn{1}{|c|}{$\sd \hat{\xi}_{n,\cc{B}} = 0.4104$} & $\sd \hat{\xi}_{n,\cc{B}} = 0.4177$ & $\sd \hat{\xi}_{n,\cc{B}} = 0.4285$ & $\sd \hat{\xi}_{n,\cc{B}} = 0.4464$ & $\sd \hat{\xi}_{n,\cc{B}} = 0.4635$  \\ \cline{3-8} 
\multicolumn{1}{|c}{} & \multicolumn{1}{|c}{} & 
\multicolumn{1}{|c}{\multirow{2}{*}{Rectangular}}  & 
\multicolumn{1}{|c|}{$\bb{E} \hat{\xi}_{n,\cc{B}} = 1.0311$} & $\bb{E} \hat{\xi}_{n,\cc{B}} = 1.0312$ & $\bb{E} \hat{\xi}_{n,\cc{B}} = 1.0313$ & $\bb{E} \hat{\xi}_{n,\cc{B}} = 1.0316$ & $\bb{E} \hat{\xi}_{n,\cc{B}} = 1.0315$  \\ 
\multicolumn{1}{|c}{} & \multicolumn{1}{|c}{} &  \multicolumn{1}{|c}{} &
\multicolumn{1}{|c|}{$\sd \hat{\xi}_{n,\cc{B}} = 0.4616$} & $\sd \hat{\xi}_{n,\cc{B}} = 0.4620$ & $\sd \hat{\xi}_{n,\cc{B}} = 0.4626$ & $\sd \hat{\xi}_{n,\cc{B}} = 0.4633$ & $\sd \hat{\xi}_{n,\cc{B}} = 0.4635$  \\ \cline{3-8} 
\multicolumn{1}{|c}{} & \multicolumn{1}{|c}{}  &
\multicolumn{1}{|c}{\multirow{2}{*}{UC}}  & 
\multicolumn{1}{|c|}{$\bb{E} \hat{\xi}_{n,\cc{B}} = 1.0232$} & $\bb{E} \hat{\xi}_{n,\cc{B}} = 1.0239$ & $\bb{E} \hat{\xi}_{n,\cc{B}} = 1.0267$ & $\bb{E} \hat{\xi}_{n,\cc{B}} = 1.0296$ & $\bb{E} \hat{\xi}_{n,\cc{B}} = 1.0315$  \\ 
\multicolumn{1}{|c}{} & \multicolumn{1}{|c}{} & \multicolumn{1}{|c}{} &
\multicolumn{1}{|c|}{$\sd \hat{\xi}_{n,\cc{B}} = 0.4096$} & $\sd \hat{\xi}_{n,\cc{B}} = 0.4156$ & $\sd \hat{\xi}_{n,\cc{B}} = 0.41275$ & $\sd \hat{\xi}_{n,\cc{B}} = 0.4463$ & $\sd \hat{\xi}_{n,\cc{B}} = 0.4635$  \\ \cline{2-8} 
\multicolumn{1}{|c}{} &
\multicolumn{1}{|c}{\multirow{6}{*}{$\delta = 3$} } &
\multicolumn{1}{|c}{\multirow{2}{*}{NUC}}  & 
\multicolumn{1}{|c|}{$\bb{E} \hat{\xi}_{n,\cc{B}} = 1.0206$} & $\bb{E} \hat{\xi}_{n,\cc{B}} = 1.0228$ & $\bb{E} \hat{\xi}_{n,\cc{B}} = 1.0223$ & $\bb{E} \hat{\xi}_{n,\cc{B}} = 1.0255$ & $\bb{E} \hat{\xi}_{n,\cc{B}} = 1.0271$  \\ 
\multicolumn{1}{|c}{} & \multicolumn{1}{|c}{} & \multicolumn{1}{|c}{} &
\multicolumn{1}{|c|}{$\sd \hat{\xi}_{n,\cc{B}} = 0.3771$} & $\sd \hat{\xi}_{n,\cc{B}} = 0.3835$ & $\sd \hat{\xi}_{n,\cc{B}} = 0.3934$ & $\sd \hat{\xi}_{n,\cc{B}} = 0.4069$ & $\sd \hat{\xi}_{n,\cc{B}} = 0.4216$  \\ \cline{3-8} 
\multicolumn{1}{|c}{} & \multicolumn{1}{|c}{} &
\multicolumn{1}{|c}{\multirow{2}{*}{Rectangular}}  & 
\multicolumn{1}{|c|}{$\bb{E} \hat{\xi}_{n,\cc{B}} = 1.0271$} & $\bb{E} \hat{\xi}_{n,\cc{B}} = 1.0276$ & $\bb{E} \hat{\xi}_{n,\cc{B}} = 1.0274$ & $\bb{E} \hat{\xi}_{n,\cc{B}} = 1.0273$ & $\bb{E} \hat{\xi}_{n,\cc{B}} = 1.0271$  \\ 
\multicolumn{1}{|c}{} & \multicolumn{1}{|c}{} & \multicolumn{1}{|c}{} &
\multicolumn{1}{|c|}{$ \sd_{n,\cc{B}} = 0.4202$} & $\sd \hat{\xi}_{n,\cc{B}} = 0.4215$ & $\sd \hat{\xi}_{n,\cc{B}} = 0.4219$ & $\sd \hat{\xi}_{n,\cc{B}} = 0.4218$ & $\sd \hat{\xi}_{n,\cc{B}} = 0.4216$  \\ \cline{3-8} 
\multicolumn{1}{|c}{} & \multicolumn{1}{|c}{} &
\multicolumn{1}{|c}{\multirow{2}{*}{UC}}  & 
\multicolumn{1}{|c|}{$\bb{E} \hat{\xi}_{n,\cc{B}} = 1.0214$} & $\bb{E} \hat{\xi}_{n,\cc{B}} = 1.0204$ & $\bb{E} \hat{\xi}_{n,\cc{B}} = 1.02037$ & $\bb{E} \hat{\xi}_{n,\cc{B}} = 1.0249$ & $\bb{E} \hat{\xi}_{n,\cc{B}} = 1.0271$  \\ 
\multicolumn{1}{|c}{} & \multicolumn{1}{|c}{} & \multicolumn{1}{|c}{} &
\multicolumn{1}{|c|}{$\sd \hat{\xi}_{n,\cc{B}} = 0.3764$} & $\sd \hat{\xi}_{n,\cc{B}} = 0.3798$ & $\sd \hat{\xi}_{n,\cc{B}} = 0.3921$ & $\sd \hat{\xi}_{n,\cc{B}} = 0.4045$ & $\sd \hat{\xi}_{n,\cc{B}} = 0.4216$  \\ \cline{1-8}
\end{tabular}
\end{adjustbox}
\captionof{table}{The mean and standard deviation of $\hat{\xi}_{n,\cc{B}}$ in experiments with $m = 2, 3$, $b_n = 1, 2, 4, 8, 16$ and $3$ binning schemes for isotropic Matern GP with $\paren{\phi_0,\rho_0,\nu} = \paren{1,2.5,1}$ observed on a perturbed lattice with $\delta = 1, 3$.} 
\label{Table2}
\end{center}

\begin{rem}\label{Rem5.1}
The above experiments explicate some aspects of the LIF method which were not thoroughly explained by the asymptotic theory. In the following we list some critical observations of the simulation studies in this section.
\begin{enumerate}[label = (\alph*),leftmargin=*]
\item In most of the entries in Tables \ref{Table1} and \ref{Table2}, the bias of $\hat{\xi}_{n,\cc{B}}$ is considerably smaller than its standard deviation. We have shown that (see the proof of Theorem \ref{CnstncyLocInvFreeThm} for further details) for isotropic Matern GPs observed in a $d$-dimensional space
\begin{equation*}
\bb{E} \hat{\xi}_{n,\cc{B}} - 1 = \cc{O}\paren{n^{-2/d}},\;\mbox{and}\quad \sd \hat{\xi}_{n,\cc{B}} = \cc{O}\paren{n^{-1/2}}.
\end{equation*}
So for $d=2$, the bias to standard deviation ratio is order $n^{-1/2}$, converging to zero as $n\rightarrow\infty$.
\item As long as $m$ is chosen to satisfy $m\geq \nu+d/2$, increasing the preconditioning order does not improve the estimation performance. On the other hand larger $m$ requires more challenging computation for evaluating the LIF loss function. So choosing $m = \lceil \nu+d/2 \rceil$ can optimally balance between statistical efficiency and computational tractability. 
\item Comparing the results in Tables \ref{Table1} and \ref{Table2} shows that $\hat{\xi}_{n,\cc{B}}$ has larger bias and standard deviation for $\nu = 1$. Namely estimating $\phi_0\rho^{-2\nu}_0$ is more difficult when $\nu = 1$. We give a qualitative justification for this phenomenon. It has been argued in Remark \ref{Rem4.3} that the LIF algorithm is consistent when the largest eigenvalue of $K^{\cc{B}}_{n,m}\paren{\cdot}$ is uniformly bounded (independent of $n$) and its Frobenius norm is of order $\sqrt{n}$. Simply put, the effective rank of $K^{\cc{B}}_{n,m}\paren{\cdot}$ should be of order $n$. Define the quantity $\Psi^{\cc{B}}_{n,m}$ as
\begin{equation*}
\Psi^{\cc{B}}_{n,m} \coloneqq \frac{ \OpNorm{K^{\cc{B}}_{n,m}}{2}{2}\sqrt{n} }{ \LpNorm{K^{\cc{B}}_{n,m}}{2} },
\end{equation*}
Observe that $\Psi^{\cc{B}}_{n,m}$ is no smaller than $1$ and attains its minimum for the identity matrix. If $K^{\cc{B}}_{n,m}\paren{\cdot}$ can be well approximated by a rank deficient matrix of rank $r_n = o\paren{n}$, then $\Psi^{\cc{B}}_{n,m}$ grows with the same rate as $\sqrt{n/r_n}$. So roughly speaking the LIF algorithm works better for smaller $\Psi^{\cc{B}}_{n,m}$. Here we compare $\Psi^{\cc{B}}_{n,m}$ for the two cases of $\nu = 0.5$ and $1$. For avoiding the computational challenges of evaluating the operator norm of large matrices, we focus on smaller size perturbed grids on $\cc{D} = \brac{0,2.5}^2$ of size $2500$ ($N = 50$) and with $\delta\in\paren{0.5,1.5}$. The range parameter of $G$ is assumed to be $\rho_0 = 1.25$. Note that $\rho_0$, the diameter of $\cc{D}$ and $\delta$ have been chosen in such a way that the lattice of size $50^2$ imitates the local neighbouring properties of $\cc{D}_n$ in Tables \ref{Table1} and \ref{Table2}. Figure \ref{Fig:Fig11} displays $\Psi^{\cc{B}}_{n,m}$ in four different scenarios of $\paren{\nu,\delta}$. It is apparent that $\Psi^{\cc{B}}_{n,m}$ is always larger for $\nu = 1$, which can explain the higher bias and variance of the LIF estimate.
\begin{figure}[!htbp]
\centering
\begin{tikzpicture}[x=1pt,y=1pt]
\definecolor{fillColor}{RGB}{255,255,255}
\path[use as bounding box,fill=fillColor,fill opacity=0.00] (0,0) rectangle (469.75,231.26);
\begin{scope}
\path[clip] (  0.00,  0.00) rectangle (469.75,231.26);
\definecolor{drawColor}{RGB}{255,255,255}
\definecolor{fillColor}{RGB}{255,255,255}

\path[draw=drawColor,line width= 0.6pt,line join=round,line cap=round,fill=fillColor] (  0.00,  0.00) rectangle (469.76,231.26);
\end{scope}
\begin{scope}
\path[clip] ( 16.51, 16.51) rectangle (464.25,225.76);
\definecolor{fillColor}{RGB}{255,255,255}

\path[fill=fillColor] ( 16.51, 16.51) rectangle (464.25,225.76);
\definecolor{drawColor}{gray}{0.92}

\path[draw=drawColor,line width= 0.6pt,line join=round] ( 16.51, 52.44) --
	(464.25, 52.44);

\path[draw=drawColor,line width= 0.6pt,line join=round] ( 16.51,105.28) --
	(464.25,105.28);

\path[draw=drawColor,line width= 0.6pt,line join=round] ( 16.51,158.13) --
	(464.25,158.13);

\path[draw=drawColor,line width= 0.6pt,line join=round] ( 16.51,210.97) --
	(464.25,210.97);

\path[draw=drawColor,line width= 0.6pt,line join=round] ( 80.47, 16.51) --
	( 80.47,225.76);

\path[draw=drawColor,line width= 0.6pt,line join=round] (187.08, 16.51) --
	(187.08,225.76);

\path[draw=drawColor,line width= 0.6pt,line join=round] (293.69, 16.51) --
	(293.69,225.76);

\path[draw=drawColor,line width= 0.6pt,line join=round] (400.29, 16.51) --
	(400.29,225.76);
\definecolor{drawColor}{RGB}{255,0,0}

\path[draw=drawColor,line width= 0.4pt,line join=round,line cap=round] ( 80.47, 56.31) circle (  1.96);

\path[draw=drawColor,line width= 0.4pt,line join=round,line cap=round] ( 80.47, 55.62) circle (  1.96);

\path[draw=drawColor,line width= 0.4pt,line join=round,line cap=round] ( 80.47, 61.71) circle (  1.96);

\path[draw=drawColor,line width= 0.4pt,line join=round,line cap=round] ( 80.47, 61.09) circle (  1.96);
\definecolor{drawColor}{gray}{0.20}

\path[draw=drawColor,line width= 0.6pt,line join=round] ( 80.47, 39.19) -- ( 80.47, 50.64);

\path[draw=drawColor,line width= 0.6pt,line join=round] ( 80.47, 28.92) -- ( 80.47, 26.65);

\path[draw=drawColor,line width= 0.6pt,line join=round,line cap=round,fill=fillColor] ( 40.50, 39.19) --
	( 40.50, 28.92) --
	(120.45, 28.92) --
	(120.45, 39.19) --
	( 40.50, 39.19) --
	cycle;

\path[draw=drawColor,line width= 1.1pt,line join=round] ( 40.50, 33.27) -- (120.45, 33.27);
\definecolor{drawColor}{RGB}{255,0,0}

\path[draw=drawColor,line width= 0.4pt,line join=round,line cap=round] (187.08, 76.44) circle (  1.96);

\path[draw=drawColor,line width= 0.4pt,line join=round,line cap=round] (187.08, 86.77) circle (  1.96);
\definecolor{drawColor}{gray}{0.20}

\path[draw=drawColor,line width= 0.6pt,line join=round] (187.08, 53.34) -- (187.08, 75.45);

\path[draw=drawColor,line width= 0.6pt,line join=round] (187.08, 38.58) -- (187.08, 30.45);

\path[draw=drawColor,line width= 0.6pt,line join=round,line cap=round,fill=fillColor] (147.10, 53.34) --
	(147.10, 38.58) --
	(227.06, 38.58) --
	(227.06, 53.34) --
	(147.10, 53.34) --
	cycle;

\path[draw=drawColor,line width= 1.1pt,line join=round] (147.10, 43.26) -- (227.06, 43.26);
\definecolor{drawColor}{RGB}{255,0,0}

\path[draw=drawColor,line width= 0.4pt,line join=round,line cap=round] (293.69,165.95) circle (  1.96);
\definecolor{drawColor}{gray}{0.20}

\path[draw=drawColor,line width= 0.6pt,line join=round] (293.69, 77.44) -- (293.69,112.67);

\path[draw=drawColor,line width= 0.6pt,line join=round] (293.69, 51.20) -- (293.69, 31.35);

\path[draw=drawColor,line width= 0.6pt,line join=round,line cap=round,fill=fillColor] (253.71, 77.44) --
	(253.71, 51.20) --
	(333.66, 51.20) --
	(333.66, 77.44) --
	(253.71, 77.44) --
	cycle;

\path[draw=drawColor,line width= 1.1pt,line join=round] (253.71, 63.46) -- (333.66, 63.46);
\definecolor{drawColor}{RGB}{255,0,0}

\path[draw=drawColor,line width= 0.4pt,line join=round,line cap=round] (400.29,187.69) circle (  1.96);

\path[draw=drawColor,line width= 0.4pt,line join=round,line cap=round] (400.29,211.04) circle (  1.96);

\path[draw=drawColor,line width= 0.4pt,line join=round,line cap=round] (400.29,203.17) circle (  1.96);

\path[draw=drawColor,line width= 0.4pt,line join=round,line cap=round] (400.29,208.19) circle (  1.96);
\definecolor{drawColor}{gray}{0.20}

\path[draw=drawColor,line width= 0.6pt,line join=round] (400.29,127.77) -- (400.29,172.26);

\path[draw=drawColor,line width= 0.6pt,line join=round] (400.29, 89.67) -- (400.29, 51.03);

\path[draw=drawColor,line width= 0.6pt,line join=round,line cap=round,fill=fillColor] (360.31,127.77) --
	(360.31, 89.67) --
	(440.27, 89.67) --
	(440.27,127.77) --
	(360.31,127.77) --
	cycle;

\path[draw=drawColor,line width= 1.1pt,line join=round] (360.31,107.53) -- (440.27,107.53);

\path[draw=drawColor,line width= 0.6pt,line join=round,line cap=round] ( 16.51, 16.51) rectangle (464.25,225.76);
\end{scope}
\begin{scope}
\path[clip] (  0.00,  0.00) rectangle (469.75,231.26);
\definecolor{drawColor}{gray}{0.30}

\node[text=drawColor,rotate= 90.00,anchor=base,inner sep=0pt, outer sep=0pt, scale=  0.88] at ( 11.56, 52.44) {5};

\node[text=drawColor,rotate= 90.00,anchor=base,inner sep=0pt, outer sep=0pt, scale=  0.88] at ( 11.56,105.28) {7};

\node[text=drawColor,rotate= 90.00,anchor=base,inner sep=0pt, outer sep=0pt, scale=  0.88] at ( 11.56,158.13) {9};

\node[text=drawColor,rotate= 90.00,anchor=base,inner sep=0pt, outer sep=0pt, scale=  0.88] at ( 11.56,210.97) {11};
\end{scope}
\begin{scope}
\path[clip] (  0.00,  0.00) rectangle (469.75,231.26);
\definecolor{drawColor}{gray}{0.20}

\path[draw=drawColor,line width= 0.6pt,line join=round] ( 13.76, 52.44) --
	( 16.51, 52.44);

\path[draw=drawColor,line width= 0.6pt,line join=round] ( 13.76,105.28) --
	( 16.51,105.28);

\path[draw=drawColor,line width= 0.6pt,line join=round] ( 13.76,158.13) --
	( 16.51,158.13);

\path[draw=drawColor,line width= 0.6pt,line join=round] ( 13.76,210.97) --
	( 16.51,210.97);
\end{scope}
\begin{scope}
\path[clip] (  0.00,  0.00) rectangle (469.75,231.26);
\definecolor{drawColor}{gray}{0.20}

\path[draw=drawColor,line width= 0.6pt,line join=round] ( 80.47, 13.76) --
	( 80.47, 16.51);

\path[draw=drawColor,line width= 0.6pt,line join=round] (187.08, 13.76) --
	(187.08, 16.51);

\path[draw=drawColor,line width= 0.6pt,line join=round] (293.69, 13.76) --
	(293.69, 16.51);

\path[draw=drawColor,line width= 0.6pt,line join=round] (400.29, 13.76) --
	(400.29, 16.51);
\end{scope}
\begin{scope}
\path[clip] (  0.00,  0.00) rectangle (469.75,231.26);
\definecolor{drawColor}{gray}{0.30}

\node[text=drawColor,anchor=base,inner sep=0pt, outer sep=0pt, scale=  0.88] at ( 80.47,  5.50) {$\left(\nu,\delta\right) = \left(0.5,0.5\right)$};

\node[text=drawColor,anchor=base,inner sep=0pt, outer sep=0pt, scale=  0.88] at (187.08,  5.50) {$\left(\nu,\delta\right) = \left(0.5,1.5\right)$};

\node[text=drawColor,anchor=base,inner sep=0pt, outer sep=0pt, scale=  0.88] at (293.69,  5.50) {$\left(\nu,\delta\right) = \left(1,0.5\right)$};

\node[text=drawColor,anchor=base,inner sep=0pt, outer sep=0pt, scale=  0.88] at (400.29,  5.50) {$\left(\nu,\delta\right) = \left(1,1.5\right)$};
\end{scope}
\end{tikzpicture} 
\vspace{-.1in}
\caption{The box-plot of $\Psi_{n,m}$ for different values of $\delta$ and $\nu$. Here $\cc{D}_n$ is a perturbed lattice of size $2500$ and $G$ is an isotropic Matern GP with $\phi_0 = 1$ and $\rho_0 = 1.25$.}
\label{Fig:Fig11}
\end{figure}
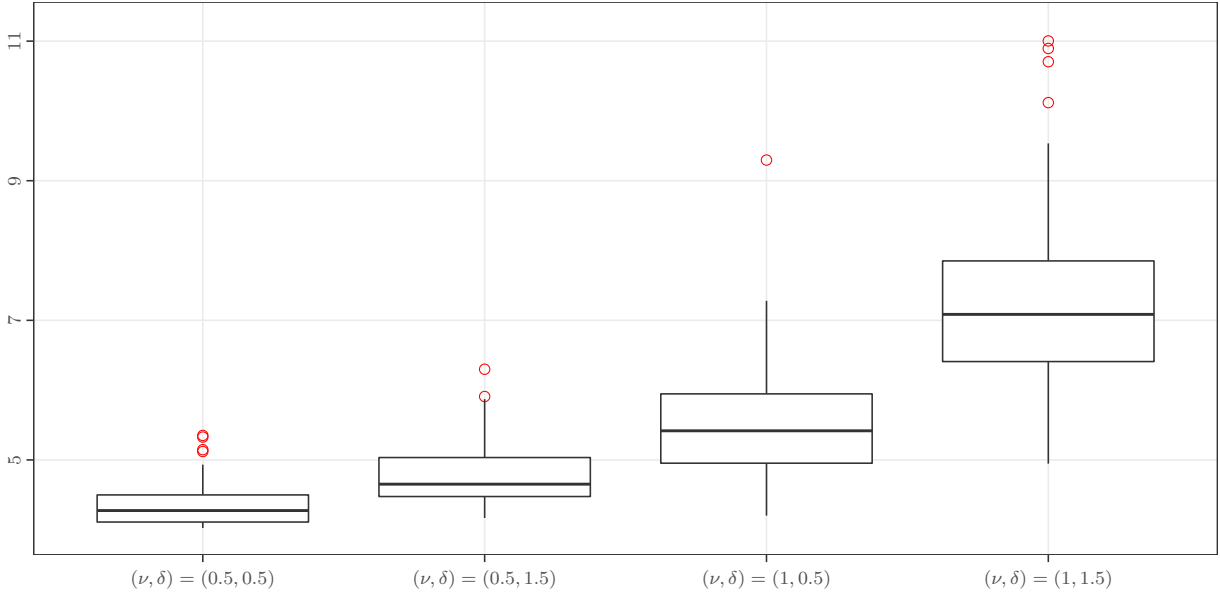 
\end{enumerate}
\end{rem} 

Now we gauge the asymptotic behaviour of the LIF estimate. For doing so we generate $100$ independent realizations of an isotropic Matern GP with $\paren{\phi_0,\rho_0,\nu} = \paren{1,5,0.5}$ on $100$ independently generated perturbed lattices of size $n = N^2$ and with $\delta\in\set{1,3}$ on $\cc{D} = \brac{0,5}^2$. The LIF loss function, with respect to the case of $b_n = 1$, is optimized with respect to $\phi$ and for a fixed $\rho = 10$. We refer the reader to Table \ref{Table3} for the sample average and standard deviation of $\hat{\xi}_{n,\cc{B}}$ for different values of $n$. The results in Table \ref{Table3} shows that the LIF estimate becomes more accurate as $n$ increases (in a fixed domain), when $b_n$ does not grow with $n$.

\begin{center} 
\begin{adjustbox}{width=0.9\columnwidth}
\begin{tabular}{cc|c|c|c|c|c|c|}
\cline{3-8} 
& & $N = 20$ & $N = 30$ & $N = 50$ & $N = 70$ & $N = 100$ & $N = 150$ \\ \cline{1-8}
\multicolumn{1}{|c|}{\multirow{2}{*}{$\delta = 1$}}  & 
\multicolumn{1}{|c|}{bias of $\hat{\xi}_{n,\cc{B}}$} & $0.8643$ & $0.5891$ & $0.2955$ & $0.1593$ & $0.0299$ & $0.0198$ \\ \cline{2-8} 
\multicolumn{1}{|c|}{} &
\multicolumn{1}{|c|}{$\sd$ of $\hat{\xi}_{n,\cc{B}}$} & $0.3716$ & $0.2305$ & $0.1093$ & $0.0700$ & $0.0480$ & $0.0233$  \\ \cline{1-8} 
\multicolumn{1}{|c|}{\multirow{2}{*}{$\delta = 3$}}  & 
\multicolumn{1}{|c|}{bias of $\hat{\xi}_{n,\cc{B}}$} & $3.2033$ & $1.0161$ & $0.5133$ & $0.2157$ & $0.0634$ & $0.0187$  \\ \cline{2-8}
\multicolumn{1}{|c|}{} &
\multicolumn{1}{|c|}{$\sd$ of $\hat{\xi}_{n,\cc{B}}$} & $1.4174$ & $0.4070$ & $0.1218$ & $0.0984$ & $0.0519$ & $0.0355$  \\ \cline{1-8} 
\end{tabular}
\end{adjustbox}
\captionof{table}{The mean and standard deviation of $\hat{\xi}_{n,\cc{B}}$ over $100$ independent experiments for isotropic Matern GP with $\paren{\phi_0,\rho_0,\nu} = \paren{1,5,0.5}$ and for different size of lattice.}
\label{Table3}
\end{center} 

\subsection{Moderate-scale simulations for geometric anisotropic GPs} 

This subsection is devoted to assess the performance of the LIF method for geometric anisotropic Matern GPs in two dimensional fixed domains. Particularly, there is $\rho_0 = \paren{\rho_{0,1},\rho_{0,2}}$ such that for any $\boldsymbol{s} = \paren{s_1,s_2}$ and $\boldsymbol{t} = \paren{t_1,t_2}$,
\begin{equation*}
\cov \Bigparen{ G\paren{\boldsymbol{s}}, G\paren{\boldsymbol{t}} } = \phi_0 f_{\nu}\paren{r},\;\mbox{in which}\; r^2 = \paren{ \frac{ t_1-s_1 }{\rho_{0,1}}}^2+\paren{\frac{ t_2-s_2 }{\rho_{0,2}}}^2.
\end{equation*}
Here $f_{\nu}$ stands for the Matern standard correlation function with the smoothness parameter $\nu$. The  quantities $\hat{\phi}_{n,\cc{B}}\in\bb{R}$ and $\hat{\rho}_{n,\cc{B}}\in\bb{R}^2$ are obtained by maximizing the LIF loss. It is known that $\phi_0$ and $\rho_0$ are not fully discernible in the infill setting (see \cite{stein2012interpolation}, p. $120$). Therefore the focus of our simulation studies is to estimate the quantities $\phi_0\rho^{-2\nu}_{0,1}$ and $\phi_0\rho^{-2\nu}_{0,2}$ (or equivalently $\phi_0\paren{\rho_{0,1}\rho_{0,2}}^{-\nu}$ and $\rho_{0,1}/\rho_{0,2}$). We refer the reader to \cite{anderes2010consistent} for a comprehensive discussion regarding the identifiability of covariance parameters in multi-dimensional geometric anisotropic Matern GPs. For brevity we reformulate $\hat{\xi}_{n,\cc{B}}$ as the following:
\begin{equation}\label{xiAnis}
\hat{\xi}_{n,\cc{B}} = \paren{ \frac{ \hat{\phi}_{n,\cc{B}}\hat{\rho}^{-2\nu}_{1,n,\cc{B}} }{ \phi_0\rho^{-2\nu}_{0,1}}, \frac{ \hat{\phi}_{n,\cc{B}}\hat{\rho}^{-2\nu}_{2,n,\cc{B}} }{ \phi_0\rho^{-2\nu}_{0,2}} } \in \brapar{0,\infty}^2.
\end{equation}
Again, we let $\cc{D}_n$ be a perturbed lattice of size $n = 10^4$ and with $\delta\in\set{1,3}$ on $\cc{D} = \brac{0,5}^2$. We simulate $100$ independent realizations of a Matern GP with $\phi_0 = 1$, $\rho_0 = \paren{1.5,4}$ and $\nu = 0.5$ on $100$ realizations of $\cc{D}_n$. The L-BFGS-B method with the initial guess $\rho = \paren{10,10}$ is used for maximizing the profile LIF loss function in a constrained box $\brac{0.1,50}^2$. In our experiments the boundary points were not touched during optimization, so the final results do not change even when the box constraints are not enforced. The scatter plots of $\hat{\xi}_{n,\cc{B}}$ is depicted in Figure \ref{Fig:Fig9} for $b_n\in\set{4,16}$ and two partitioning approaches. It appears that $\hat{\xi}_{n,\cc{B}}$ is concentrated around $\paren{1,1}$ for all the scenarios. Table \ref{Table5} also accumulates the mean and standard deviation of $\hat{\xi}_{n,\cc{B}}$ displayed in Figure \ref{Fig:Fig9}. 
%

\begin{figure}[!htbp]
\centering
\subimport{/}{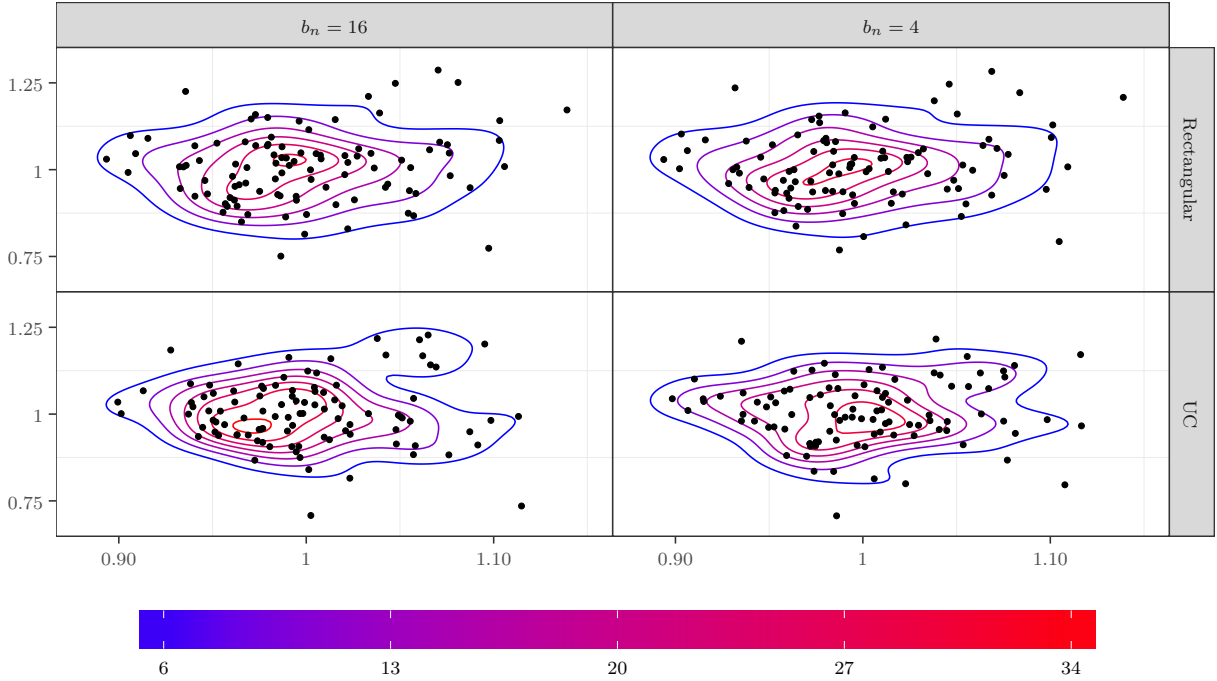}  
\vspace{-.23in}
\caption{The scatter plot and two dimensional KDE of $\hat{\xi}_{n,\cc{B}}$ for an anisotropic Matern GP with $\phi_0 = 1, \rho_0 = \paren{1.5,4}$, and $\nu_0 = 0.5$ observed on a perturbed lattice with $\delta = 1$ and $n=10^4$.}
\label{Fig:Fig9}
\end{figure}

\begin{center} 
\begin{adjustbox}{width=0.8\columnwidth}
\begin{tabular}{c|c|c|}
\cline{2-3} 
& $b_n = 16$ & $b_n = 4$  \\ \cline{1-3}
\multicolumn{1}{|c|}{\multirow{2}{*}{UC}}  & 
\multicolumn{1}{|c|}{$\bb{E} \hat{\xi}_{n,\cc{B}} = \paren{0.9996, 1.0063}$} & $\bb{E} \hat{\xi}_{n,\cc{B}} = \paren{1.0002, 1.0049}$   \\ 
\multicolumn{1}{|c|}{} &
\multicolumn{1}{|c|}{$\sd \hat{\xi}_{n,\cc{B}} = \paren{0.0467, 0.0966}$} & $\sd \hat{\xi}_{n,\cc{B}} = \paren{0.0482, 0.0932}$  \\ \cline{1-3} 
\multicolumn{1}{|c|}{\multirow{2}{*}{Rectangular}}  & 
\multicolumn{1}{|c|}{$\bb{E} \hat{\xi}_{n,\cc{B}} = \paren{0.9993, 1.0081}$} & $\bb{E} \hat{\xi}_{n,\cc{B}} = \paren{0.9994, 1.0104}$  \\ 
\multicolumn{1}{|c|}{} & 
\multicolumn{1}{|c|}{$\sd \hat{\xi}_{n,\cc{B}} = \paren{0.0507, 0.1026}$} & $\sd \hat{\xi}_{n,\cc{B}} = \paren{0.0515, 0.0998}$  \\ \cline{1-3} 
\end{tabular} 
\end{adjustbox}
\captionof{table}{The mean and standard deviation of $\hat{\xi}_{n,\cc{B}}$ exhibited in scatter plots in Figures \ref{Fig:Fig9}.}
\label{Table5}
\end{center}  

\subsection{Large-scale simulations for geometric anisotropic GPs} 

To obtain further insights into the estimation accuracy of the LIF algorithm on large data sets, we carry out a few simulation studies on Matern GPs observed on perturbed lattices. The simulations are separated into two categories described as follows.
\begin{enumerate}[leftmargin=*]
\item We fix $\cc{D} = \brac{0,25}^2$ and choose a perturbed lattice $\cc{D}_n$ of size $2.5\times 10^5$, i.e. $N = 500$, with $\delta = 5$ on $\cc{D}$. $G$ is a geometric anisotropic Matern GP with $\rho_0 = \paren{\rho_{0,1},\rho_{0,2}} = \paren{2,5}$ and $\phi_0 = 1$ observed on $\cc{D}_n$. Such simulation imitates the large-sample infill behaviour, as the diameter of $\cc{D}$ is considerably smaller than $N$. We report the LIF estimates of $\phi_0 \rho^{-2\nu}_{0,1}$ and $\phi_0 \rho^{-2\nu}_{0,2}$.
\item In the second class which emulates the increasing domain setting, we select $\cc{D} = \brac{0,500}^2$. Furthermore, the variance and range parameter of $G$ are given by $\phi_0 = 1$ and $\rho_0 = \paren{10,20}$ and $\nu = 1$. $\cc{D}_n$ is also treated the same as the first category ($N = 500$). In these simulations, the estimates of all unknown parameters will be reported.
\end{enumerate}

Recall $\hat{\xi}_{n,\cc{B}}$ from Eq. \eqref{xiAnis}. Tables \ref{Table6} encapsulates $\hat{\xi}_{n,\cc{B}}$ and the running time of maximizing the profile LIF loss in the box-constrained region $\brac{0.1,50}$ by L-BFGS-B algorithm and with the initial guess $\rho = \paren{4,8}$. Comparing to the case of $\nu = 0.5$, the optimization algorithm is three times slower for $\nu = 1$, which is due to the more complicated form of the covariance function. Furthermore the running time of the LIF loss optimizer is inversely proportional to $b_n$.

\vspace{3mm}

\begin{center} 
\begin{adjustbox}{width=0.8\columnwidth}
\begin{tabular}{cc|c|c|c|}
\cline{3-5} 
& & $b_n = 200$ & $b_n = 50$ & $b_n = 10$  \\ \cline{1-5}
\multicolumn{1}{|c}{\multirow{2}{*}{$\nu = 0.5$} } &
\multicolumn{1}{|c|}{$\hat{\xi}_{n,\cc{B}}$} & $\paren{0.9978, 1.0434}$ & $\paren{0.9988, 1.04085}$ & $\paren{1.0011, 1.0280}$ \\
\cline{2-5}  
\multicolumn{1}{|c|}{} & 
\multicolumn{1}{|c|}{Running time (hour)} & $0.5016$  &  $2.1747$ & $4.8055$ \\ 
\cline{1-5} 
\multicolumn{1}{|c}{\multirow{2}{*}{$\nu = 1$} } &
\multicolumn{1}{|c|}{$\hat{\xi}_{n,\cc{B}}$} & $\paren{0.9910, 1.1060}$ & $\paren{0.9951, 1.0858}$ & $\paren{0.9928, 1.0899}$ \\ 
\cline{2-5} 
\multicolumn{1}{|c|}{} &
\multicolumn{1}{|c|}{Running time (hour)} & $1.4128$  &  $5.4449$ & $13.2018$ \\ 
\cline{1-5}
\end{tabular} 
\end{adjustbox}
\captionof{table}{The summary of the large-sample simulations for the first category.}
\label{Table6}
\end{center}

Table \ref{Table7} presents the summary of results for the case that $\cc{D} = \brac{0,500}^2$. The L-BFGS-B optimizer starts at $\rho = \paren{25,40}$. We only consider the case that $\nu = 1$, because of the more challenging computation. Note that obtaining the estimated parameters in this setting is around twice as slow as the former case.

\vspace{3mm}

\begin{center} 
\begin{adjustbox}{width=0.8\columnwidth}
\begin{tabular}{cc|c|c|c|}
\cline{3-5} 
& & $b_n = 200$ & $b_n = 50$ & $b_n = 10$  \\ \cline{1-5}
\multicolumn{1}{|c}{\multirow{3}{*}{$\nu = 1$} } &
\multicolumn{1}{|c|}{$\hat{\phi}_{n,\cc{B}}$} & $1.0179$ & $1.0072$ & $1.0125$ \\ \cline{2-5}  
\multicolumn{1}{|c|}{} & 
\multicolumn{1}{|c|}{$\hat{\rho}_{n,\cc{B}}$} & $\paren{10.4457,19.8137}$ & $\paren{10.3789,19.8433}$ & $\paren{10.4203,19.8278}$ \\ \cline{2-5} 
\multicolumn{1}{|c|}{} & 
\multicolumn{1}{|c|}{Running time (hour)} & $2.7441$  &  $10.5585$ & $25.6577$ \\ 
\cline{1-5} 
\end{tabular} 
\end{adjustbox}
\captionof{table}{The summary of the large-sample simulations for the second category.}
\label{Table7}
\end{center}

Comparing the different columns in Table \ref{Table6} and \ref{Table7} reveals insensitivity of the LIF estimate to $b_n$. We believe that for large $n$, increasing the number of bins does not improve the statistical accuracy, as long as each bin can separately encode the local dependence structure. For instance when $n = 500^2$ and $b_n = 200$, there are more than $1000$ samples in each bin, which is roughly enough for learning the local dependence structure in a geometric anisotropic GP with two range parameters. We observe that there is a large range of $b_n$ in which decreasing the bin size (which is equivalent to increasing $b_n$) barely degrades the statistical performance of the LIF algorithm, but the computational saving is quite substantial.

Finally, for a systematic evaluation of the role of $b_n$ on the statistical accuracy of the LIF estimate we consider a Geometric anisotropic GP with $\phi_0 = 1$ and $\rho_0 = \paren{5,10}$ and $\nu\in\set{0.5,1}$ on a regular lattice ($\delta = 0$) of size $n = 40^2$ on $\cc{D} = \brac{0,10}^2$. That is, $\cc{D}_n = \set{i/5:\;i=1,\ldots,40}^2$. Similar to the results in Table \ref{Table6}, the L-BFGS-B algorithm with starting point $\rho=\paren{4,8}$ is used for estimating $\hat{\xi}_{n,\cc{B}}$. For each $b_n$, we run $100$ independent experiments for evaluating the empirical mean and standard deviation of $\hat{\xi}_{n,\cc{B}}$. The summary results in Table \ref{Table8} shows that the standard deviation of the LIF estimator increases for larger $b_n$.

\vspace{3mm}

\begin{center} 
\begin{adjustbox}{width=0.8\columnwidth}
\begin{tabular}{cc|c|c|c|c|}
\cline{3-6} 
& & $b_n = 1$ & $b_n = 2$ & $b_n = 4$ & $b_n = 8$  \\ \cline{1-6}
\multicolumn{1}{|c}{\multirow{2}{*}{$\nu = 0.5$} } &
\multicolumn{1}{|c|}{$\bb{E}\hat{\xi}_{n,\cc{B}}$} & $\paren{0.9931, 1.0214}$ & $\paren{0.9902, 1.0286}$ & $\paren{0.9914, 1.0324}$ & $\paren{0.9923, 1.0341}$\\
\cline{2-6}  
\multicolumn{1}{|c|}{} & 
\multicolumn{1}{|c|}{$\sd \hat{\xi}_{n,\cc{B}}$} & $\paren{0.0201, 0.0372}$ & $\paren{0.0239, 0.0398}$ & $\paren{0.0272, 0.0448}$ & $\paren{0.0290, 0.0482}$\\
\cline{1-6} 
\multicolumn{1}{|c}{\multirow{2}{*}{$\nu = 1$} } &
\multicolumn{1}{|c|}{$\bb{E}\hat{\xi}_{n,\cc{B}}$} & $\paren{0.9873, 1.0521}$ & $\paren{0.9821, 1.0593}$ & $\paren{0.9852, 1.0565}$ & $\paren{0.9813, 1.0591}$\\
\cline{2-6} 
\multicolumn{1}{|c|}{} &
\multicolumn{1}{|c|}{$\sd \hat{\xi}_{n,\cc{B}}$} & $\paren{0.0573, 0.1011}$ & $\paren{0.0611, 0.1098}$ & $\paren{0.0659, 0.1149}$ & $\paren{0.0682, 0.1178}$\\
\cline{1-6}
\end{tabular} 
\end{adjustbox}
\captionof{table}{The summary of simulations for assessing the role of $b_n$.}
\label{Table8}
\end{center}

\section{Discussion}\label{Discus}

In this paper we have introduced a family of scalable covariance estimation algorithms, called the local inversion-free (LIF) algorithm, by amalgamating the ideas of the inversion-free estimation procedure in \cite{anitescu2017inversion} and a block diagonal approximation of the covariance matrix of the preconditioned data. We have established $\sqrt{n}$-consistency and asymptotic normality of our method for the isotropic Matern covariance function on a $d$-dimensional irregular lattice (with $d\leq 3$). Prior to this work, it had only been asserted that the inversion-free estimator is statistically comparable to the MLE, when there exists a linear transformation to uniformly control the condition number of the covariance matrix below some constant, independent of the sample size \cite{anitescu2017inversion}. However, our analysis demonstrates that the LIF algorithm has the same convergence rate as the MLE, as long as the largest eigenvalue remains uniformly bounded and a non-negligible fraction of the eigenvalues are further away from zero. The removal of the necessity of uniformly controlling the condition number of the covariance matrix in our asymptotic theory can expand the applicability of surrogate loss maximization methods for estimating the covariance of spatial Gaussian processes. 

Despite the relatively low cost of computing the LIF estimate for GPs observed on irregularly spaced locations, it remains to investigate the applicability of LIF-based algorithms beyond parameter estimation, e.g., prediction. Furthermore, despite recent progresses in preconditioning of stationary GPs, an effective mechanism to reduce the condition number of the covariance matrix for non-stationary random fields is still obscure. However, we have only scratched the surface of scalable non-likelihood based estimation algorithms and still much needs to be done for developing an efficient class of algorithms for a broad family of spatial processes.

We end this discussion by briefly describing a potential way of adjusting the LIF loss function for non-stationary processes with smoothly varying variance and range parameters (with a known smoothness parameter). The main idea is to partition the set of sampling sites $\cc{D}_n$ into $b_n$ small bins, so that the GP inside each bin can be well approximated by a stationary process. For any $\boldsymbol{s}\in\cc{D}_n$, construct the set $\cc{N}_m\paren{\boldsymbol{s}}$ using the nearest neighbours of $\boldsymbol{s}$ inside its associated bin. The vectors of variance and range parameters, denoted by $\boldsymbol{\phi}_0 = \brac{\phi_{0,1},\ldots,\phi_{0,b_n}}^\top$ and $\boldsymbol{\rho}_0 = \brac{\rho_{0,1},\ldots,\rho_{0,b_n}}^\top$, can be simultaneously estimated by optimizing a penalized LIF objective function, namely,
\begin{equation*}
\paren{\hat{\boldsymbol{\phi}}_{n,\cc{B}}, \hat{\boldsymbol{\rho}}_{n,\cc{B}}} = \argmin_{\boldsymbol{\phi}, \boldsymbol{\rho}} \set{ \sum_{t=1}^{b_n} \LpNorm{Y_{B_t,m}Y^\top_{B_t,m}-\phi_tK_{B_t,m}\paren{\rho_t} }{2}^2 + J_{\boldsymbol{\phi}}\paren{\phi_1,\ldots, \phi_{b_n}} + J_{\boldsymbol{\rho}}\paren{\rho_1,\ldots,\rho_{b_n}} },
\end{equation*}
in which $J_{\boldsymbol{\phi}}$ and $J_{\boldsymbol{\rho}}$ are non-negative functions penalizing rapidly varying variance and range parameters. Such a penalized loss function may be optimized using the coordinate descent method.

\section{Proofs}\label{Proofs}

All the constants appearing in this section (including those implicitly defined in $\lesssim$, and $\asymp$), are bounded and depend on $m, \nu, d, \Theta_0$, and the geometric structure of the sampling locations.

\begin{proof}[Proof of Theorem \ref{CnstncyLocInvFreeThm}]
Applying the triangle inequality, we get
\begin{equation}\label{BiasVarDecmp}
\sup_{\rho\in\Theta_0}\abs{ \frac{\hat{\phi}_{n,\cc{B}}\paren{\rho}\rho^{-2\nu}}{\phi_0\rho^{-2\nu}_0}-1 } \leq \sup_{\rho\in\Theta_0}\abs{\frac{\bb{E} \hat{\phi}_{n,\cc{B}}\paren{\rho}\rho^{-2\nu}}{\phi_0\rho^{-2\nu}_0} - 1 }+\sup_{\rho\in\Theta_0} \frac{\abs{\hat{\phi}_{n,\cc{B}}\paren{\rho}\rho^{-2\nu}-\bb{E}\hat{\phi}_{n,\cc{B}}\paren{\rho}\rho^{-2\nu}}}{\phi_0\rho^{-2\nu}_0}.
\end{equation}
Let $P_1$ and $P_2$ respectively stand for the two terms in the right hand side of \eqref{BiasVarDecmp}. For clarity, we break the proof into two parts. The first part is devoted to uniformly control $P_1$. Strictly speaking, we prove that 
\begin{equation*}
P_1\lesssim \paren{\bbM{1}_{\set{d=1}} \frac{1}{n} + \bbM{1}_{\set{d=2}} \frac{\log n}{n} + \bbM{1}_{\set{d\geq 3}} n^{-2/d} } \paren{1+ \bbM{1}_{\set{m=\nu+d/2}} \log n  }.
\end{equation*}
We then show that the stochastic quadratic quantity $P_2$ is of order $\sqrt{n^{-1}\log n}$, with high probability. The concentration inequalities involving the quadratic forms (and their supremum over a bounded space) of GPs presented in \cite{keshavarz2016consistency} are crucial for bounding $P_2$ from above.

Choose an arbitrary $\paren{\phi,\rho}\in\cc{I}\times \Theta_0$. Recall $K^{\cc{B}}_{n,m}\paren{\rho}\in\bb{R}^{n\times n}$ from \eqref{Ktilde_n,m} and $\hat{\phi}_{n,\cc{B}}\paren{\rho}$ from Eq. \eqref{PhiHatnB}. For brevity, define $L^{\cc{B}}_{n,m}\paren{\rho} \coloneqq \rho^{2\nu}K^{\cc{B}}_{n,m}\paren{\rho}$. Observe that
\begin{equation*}
\frac{\bb{E}\hat{\phi}_{n,\cc{B}}\paren{\rho}\rho^{-2\nu}}{\phi_0\rho^{-2\nu}_0} = \frac{\rho^{-2\nu}}{\phi_0\rho^{-2\nu}_0} \frac{\bb{E} Y^\top K^{\cc{B}}_{n,m}\paren{\rho}Y }{\LpNorm{ K^{\cc{B}}_{n,m}\paren{\rho}}{2}^2} = \paren{\frac{\rho_0}{\rho}}^{2\nu} \frac{ \InnerProd{K^{\cc{B}}_{n,m}\paren{\rho}}{K^{\cc{B}}_{n,m}\paren{\rho_0}} }{\LpNorm{ K^{\cc{B}}_{n,m}\paren{\rho}}{2}^2} = \frac{ \InnerProd{L^{\cc{B}}_{n,m}\paren{\rho}}{L^{\cc{B}}_{n,m}\paren{\rho_0}} }{\LpNorm{ L^{\cc{B}}_{n,m}\paren{\rho}}{2}^2}.
\end{equation*} 
Thus, 
\begin{eqnarray}\label{P1UppBnd1D}
P_1 &=& \sup_{\rho\in\Theta_0} \abs{\frac{ \InnerProd{L^{\cc{B}}_{n,m}\paren{\rho}}{L^{\cc{B}}_{n,m}\paren{\rho_0}} }{\LpNorm{ L^{\cc{B}}_{n,m}\paren{\rho}}{2}^2}-1} = \sup_{\rho\in\Theta_0} \abs{\frac{ \InnerProd{L^{\cc{B}}_{n,m}\paren{\rho}-L^{\cc{B}}_{n,m}\paren{\rho_0}}{L^{\cc{B}}_{n,m}\paren{\rho}} }{\LpNorm{ L^{\cc{B}}_{n,m}\paren{\rho}}{2}^2}}\nonumber\\
&\RelNum{\paren{a}}{\leq}& \sup_{\rho\in\Theta_0} \brac{\frac{ \SpNorm{L^{\cc{B}}_{n,m}\paren{\rho}-L^{\cc{B}}_{n,m}\paren{\rho_0}}{1} \OpNorm{L^{\cc{B}}_{n,m}\paren{\rho}}{2}{2} }{\LpNorm{ L^{\cc{B}}_{n,m}\paren{\rho}}{2}^2}}.
\end{eqnarray}
Here $\paren{a}$ is implied by the generalized Cauchy-Schwartz inequality. We assess the large sample behaviour of the terms appearing in the second line of \eqref{P1UppBnd1D} in Appendix \ref{AppendixCovMatrixIrregLat}. Lemma \ref{LowBndFrobNormLnm} states that $\min_{\rho\in\Theta_0} \LpNorm{L^{\cc{B}}_{n,m}\paren{\rho}}{2} \gtrsim \sqrt{n}$. For brevity define $\Delta^{\cc{B}}\paren{\rho,\rho_0}\coloneqq L^{\cc{B}}_{n,m}\paren{\rho}-L^{\cc{B}}_{n,m}\paren{\rho_0}$.  Furthermore, Lemma \ref{SensAnalS1Norm} implies that
\begin{eqnarray}\label{UppBndNucNormDelta}
\sup_{\rho\in\Theta_0} \SpNorm{\Delta^{\cc{B}}\paren{\rho,\rho_0}}{1} &\lesssim& \paren{ \bbM{1}_{\set{d=1}} + \bbM{1}_{\set{d=2}} \log n + \bbM{1}_{\set{d\geq 3}} n^{1-2/d} } \diam\paren{\Theta_0} \nonumber\\
&\asymp& \paren{ \bbM{1}_{\set{d=1}} + \bbM{1}_{\set{d=2}} \log n + \bbM{1}_{\set{d\geq 3}} n^{1-2/d} }. 
\end{eqnarray}
Thus the upper bound on $P_1$ in \eqref{P1UppBnd1D} can be rewritten as 
\begin{equation}\label{UppBndP1GenCase}
P_1\lesssim \paren{ \frac{\bbM{1}_{\set{d=1}}}{n} + \bbM{1}_{\set{d=2}} \frac{\log n}{n} + \bbM{1}_{\set{d\geq 3}} n^{-2/d} } \sup_{\rho\in\Theta_0} \OpNorm{L^{\cc{B}}_{n,m}\paren{\rho}}{2}{2}.
\end{equation}
So it is only needed to find a uniform upper bound on the largest eigenvalue of $L^{\cc{B}}_{n,m}\paren{\rho}$ on $\Theta_0$. Notice that $L^{\cc{B}}_{n,m}\paren{\rho}$ is a block diagonalized version of $L_{n,m}\paren{\rho}$. Hence
\begin{equation*}
\OpNorm{L^{\cc{B}}_{n,m}\paren{\rho}}{2}{2}\leq \OpNorm{L_{n,m}\paren{\rho}}{2}{2},\quad\forall\;\rho\in\Theta_0
\end{equation*}
In other words, we only need to focus on the case of no partitioning. For $d$-dimensional regular lattices, the exact procedure as Theorems $2.1$ and $2.3$ of \cite{stein2012difference} demonstrates that all the eigenvalues of $L_{n,m}\paren{\rho}$ are universally bounded. Namely, 
\begin{equation}\label{UppBndOpNormLnmRegLat}
\sup_{\rho\in\Theta_0} \lambda_j\paren{ L_{n,m}\paren{\rho} } \leq \alpha_{\max},\quad\forall\; j = 1,\ldots, \abs{\cc{D}_n}
\end{equation}
for some bounded $\alpha^{\max}>0$. Thus $P_1$ admits the following inequality for regular lattices.
\begin{equation}\label{UppBndP1RegLatt}
P_1\lesssim \paren{ \frac{\bbM{1}_{\set{d=1}}}{n} + \bbM{1}_{\set{d=2}} \frac{\log n}{n} + \bbM{1}_{\set{d\geq 3}} n^{-2/d} }.
\end{equation}
However the operator norm of $L_{n,m}\paren{\rho}$ is not necessarily uniformly bound on $\Theta_0$, for a general irregular lattice satisfying Assumption \ref{RegCondLattice}. For such case, we show in Proposition \ref{UppBndCovGmNonRegLatt} that 
\begin{equation}\label{DecayRateEntrLnm}
\abs{\paren{ L_{n,m}\paren{\rho} }_{\boldsymbol{s},\boldsymbol{t}}} \lesssim \paren{1+ \lfloor n^{1/d} \rfloor\LpNorm{\boldsymbol{t}-\boldsymbol{s}}{2} }^{-2\paren{m-\nu}},\quad\; \boldsymbol{s},\boldsymbol{t}\in\cc{D}_n.
\end{equation}
Lemma \ref{OpNormPolDecayOffDiag} also introduces an upper bound on the operator norm of the matrices satisfying \eqref{DecayRateEntrLnm}. Applying Lemma \ref{OpNormPolDecayOffDiag} yields
\begin{equation}\label{UppBndOpNormLnm}
\sup_{\rho\in\Theta_0} \OpNorm{L^{\cc{B}}_{n,m}\paren{\rho}}{2}{2} \leq \sup_{\rho\in\Theta_0} \OpNorm{L_{n,m}\paren{\rho}}{2}{2} \lesssim \paren{1+ \bbM{1}_{\set{m=\nu+d/2}}\log n}.
\end{equation}
The desired bound on $P_1$ is obtained by combining \eqref{UppBndP1GenCase} and \eqref{UppBndOpNormLnm}. The next goal is control $P_2$ from above. Let $Z\in\bb{R}^{n}$ be a standard Gaussian vector and define the symmetric matrix $M^{\cc{B}}_{n,m}\paren{\rho}$ by
\begin{equation}\label{ffM}
M^{\cc{B}}_{n,m}\paren{\rho} = \sqrt{L_{n,m}\paren{\rho_0}}\brac{\frac{nL^{\cc{B}}_{n,m}\paren{\rho}}{\LpNorm{L^{\cc{B}}_{n,m}\paren{\rho}}{2}^2}}\sqrt{L_{n,m}\paren{\rho_0}},\quad\forall\;\rho\in\Theta_0.
\end{equation}
We first introduce an equivalent representation for $\hat{\phi}_{n,\cc{B}}\paren{\rho}\rho^{-2\nu}$ in terms of $Z$ and $M^{\cc{B}}_{n,m}\paren{\rho}$. Obviously, the Gaussian vectors $Y$ and $\sqrt{\phi_0K_{n,m}\paren{\rho_0}}Z = \phi^{1/2}_0\rho^{-\nu}_0\sqrt{L_{n,m}\paren{\rho_0}}Z$ have the same distribution. Thus,
\begin{equation*}
\hat{\phi}_{n,\cc{B}}\paren{\rho}\rho^{-2\nu} = \rho^{-2\nu} \frac{Y^\top K^{\cc{B}}_{n,m}\paren{\rho}Y}{\LpNorm{K^{\cc{B}}_{n,m}\paren{\rho}}{2}^2} = \frac{Y^\top L^{\cc{B}}_{n,m}\paren{\rho}Y}{\LpNorm{L^{\cc{B}}_{n,m}\paren{\rho}}{2}^2} \eqd \frac{Z^\top M^{\cc{B}}_{n,m}\paren{\rho} Z}{n} \phi_0\rho^{-2\nu}_0,
\end{equation*}
and so $P_2$ can be rewritten as the supremum of a centered $\chi^2$ process over $\Theta_0$, i.e.,
\begin{equation*}
P_2 = \frac{1}{n}\sup_{\rho\in\Theta_0} \abs{ Z^\top M^{\cc{B}}_{n,m}\paren{\rho} Z - \tr\set{ M^{\cc{B}}_{n,m}\paren{\rho} } }.
\end{equation*}
So if $M^{\cc{B}}_{n,m}\paren{\rho}$ admits the three conditions in Proposition \ref{ModalProp}, then there are bounded scalars $C$ and $n_0\in\bb{N}$ such that for any $n\geq n_0$, we have
\begin{equation}\label{P2Rate}
\bb{P}\paren{ P_2 \geq C\sqrt{\frac{\log n}{n}} } = \bb{P}\paren{ \sup_{\rho\in\Theta_0} \abs{ Z^\top M^{\cc{B}}_{n,m}\paren{\rho} Z - \tr\set{ M^{\cc{B}}_{n,m}\paren{\rho} } }\geq C\sqrt{n\log n} } \leq \frac{1}{n}.
\end{equation}
Thus we require to verify the conditions $\paren{a}-\paren{c}$ in Proposition \ref{ModalProp}. \\
\emph{Validating condition $\paren{a}$}. We should substantiate the uniform boundedness of $n^{-1/2}\LpNorm{M^{\cc{B}}_{n,m}\paren{\rho}}{2}$ over $\Theta_0$. Namely, we must prove that $U$ defined as the following is bounded.
\begin{equation*}
U\coloneqq \sup_{\rho\in\Theta_0}\frac{\LpNorm{M^{\cc{B}}_{n,m}\paren{\rho}}{2}}{\sqrt{n}} = \sup_{\rho\in\Theta_0} \frac{\sqrt{n}\LpNorm{\sqrt{L_{n,m}\paren{\rho_0}} L^{\cc{B}}_{n,m}\paren{\rho} \sqrt{L_{n,m}\paren{\rho_0}}}{2}}{\LpNorm{L^{\cc{B}}_{n,m}\paren{\rho}}{2}^2}.
\end{equation*}
We prove in Lemma \ref{LowBndFrobNormLnm} that $\min_{\rho\in\Theta_0} n^{-1}\LpNorm{L^{\cc{B}}_{n,m}\paren{\rho}}{2}^2 > 0$ for large enough $n$. Thus, $U$ can be bounded above by some $U'$ given by
\begin{equation*}
U \lesssim U'\coloneqq \sup_{\rho\in\Theta_0} \frac{\LpNorm{\sqrt{L_{n,m}\paren{\rho_0}} L^{\cc{B}}_{n,m}\paren{\rho} \sqrt{L_{n,m}\paren{\rho_0}}}{2}}{\sqrt{n}}.
\end{equation*}
Finally, Lemma \ref{FrobNormAsympNumeratMnm} ensures the boundedness of $U'$ (and consequently $U$).\\
\emph{Validating condition $\paren{b}$}. Pick two arbitrary distinct $\rho_1,\rho_2\in\Theta_0$ with $\abs{\rho_2-\rho_1} \leq 1$. Our objective is to demonstrate the Lipschitz property of $\OpNorm{M^{\cc{B}}_{n,m}\paren{\rho_2}-M^{\cc{B}}_{n,m}\paren{\rho_1}}{2}{2}$ (with a constant of order $\log^2 n$). Obviously
\begin{equation*}
\frac{\OpNorm{M^{\cc{B}}_{n,m}\paren{\rho_2}-M^{\cc{B}}_{n,m}\paren{\rho_1}}{2}{2}}{n\abs{\rho_2-\rho_1}}\leq\frac{\OpNorm{L_{n,m}\paren{\rho_0}}{2}{2}}{\abs{\rho_2-\rho_1}} \OpNorm{ \frac{L^{\cc{B}}_{n,m}\paren{\rho_2}}{\LpNorm{L^{\cc{B}}_{n,m}\paren{\rho_2}}{2}^2}-\frac{L^{\cc{B}}_{n,m}\paren{\rho_1}}{\LpNorm{L^{\cc{B}}_{n,m}\paren{\rho_1}}{2}^2} }{2}{2}.
\end{equation*}
We have argued in \eqref{UppBndOpNormLnm} that $\OpNorm{L_{n,m}\paren{\rho_0}}{2}{2}\lesssim \paren{1+\bbM{1}_{\set{m=\nu+d/2}}\log n} \leq \log n$. Hence,
\begin{equation}\label{CondBeq1}
\frac{\OpNorm{M^{\cc{B}}_{n,m}\paren{\rho_2}-M^{\cc{B}}_{n,m}\paren{\rho_1}}{2}{2}}{n\abs{\rho_2-\rho_1}\log n}\lesssim \frac{1}{\abs{\rho_2-\rho_1}}\OpNorm{ \frac{L^{\cc{B}}_{n,m}\paren{\rho_2}}{\LpNorm{L^{\cc{B}}_{n,m}\paren{\rho_2}}{2}^2}-\frac{L^{\cc{B}}_{n,m}\paren{\rho_1}}{\LpNorm{L^{\cc{B}}_{n,m}\paren{\rho_1}}{2}^2} }{2}{2}.
\end{equation}
Furthermore, we know from the triangle inequality that  
\begin{eqnarray}\label{CondBeq2}
\OpNorm{ \frac{L^{\cc{B}}_{n,m}\paren{\rho_2}}{\LpNorm{L^{\cc{B}}_{n,m}\paren{\rho_2}}{2}^2}-\frac{L^{\cc{B}}_{n,m}\paren{\rho_1}}{\LpNorm{L^{\cc{B}}_{n,m}\paren{\rho_1}}{2}^2} }{2}{2}&\leq& \frac{\OpNorm{L^{\cc{B}}_{n,m}\paren{\rho_2}-L^{\cc{B}}_{n,m}\paren{\rho_1}}{2}{2}}{\LpNorm{L^{\cc{B}}_{n,m}\paren{\rho_2}}{2}^2 } \nonumber\\
&+& \OpNorm{ \frac{L^{\cc{B}}_{n,m}\paren{\rho_1}}{\LpNorm{L^{\cc{B}}_{n,m}\paren{\rho_2}}{2}^2}-\frac{L^{\cc{B}}_{n,m}\paren{\rho_1}}{\LpNorm{L^{\cc{B}}_{n,m}\paren{\rho_1}}{2}^2} }{2}{2}.
\end{eqnarray}
Let $\varPsi^1_n\paren{\rho_1,\rho_2}$ and $\varPsi^2_n\paren{\rho_1,\rho_2}$ stand for the first and second terms in the right hand side of \eqref{CondBeq2}, which we aim to control from above. The fact that $\min_{\rho\in\Theta_0} n^{-1}\LpNorm{L^{\cc{B}}_{n,m}\paren{\rho}}{2}^2 > 0$ (see Lemma \ref{LowBndFrobNormLnm}) comes in handy for finding a simpler upper bound on $\varPsi^1_n\paren{\rho_1,\rho_2}$ and $\varPsi^2_n\paren{\rho_1,\rho_2}$.
\begin{equation*}
\varPsi^1_n\paren{\rho_1,\rho_2} \coloneqq  \frac{\OpNorm{L^{\cc{B}}_{n,m}\paren{\rho_2}-L^{\cc{B}}_{n,m}\paren{\rho_1}}{2}{2}}{\LpNorm{L^{\cc{B}}_{n,m}\paren{\rho_2}}{2}^2}\lesssim \frac{\OpNorm{L^{\cc{B}}_{n,m}\paren{\rho_2}-L^{\cc{B}}_{n,m}\paren{\rho_1}}{2}{2}}{n}.
\end{equation*}
Furthermore, Lemma \ref{SensAnalOpNorm} indicates that
\begin{equation*}
\OpNorm{L^{\cc{B}}_{n,m}\paren{\rho_2}-L^{\cc{B}}_{n,m}\paren{\rho_1}}{2}{2} \lesssim \paren{1+\bbM{1}_{\set{m=\nu+d/2}}\log n} \abs{\rho_2-\rho_1} \leq \abs{\rho_2-\rho_1} \log n.
\end{equation*}
So $\varPsi^1_n\paren{\rho_1,\rho_2} \lesssim n^{-1}\log n \abs{\rho_2-\rho_1}$. Now we consider $\varPsi^2_n\paren{\rho_1,\rho_2}$. Observe that
\begin{eqnarray*}
\varPsi^2_n\paren{\rho_1,\rho_2} &\coloneqq& \OpNorm{ L^{\cc{B}}_{n,m}\paren{\rho_1}}{2}{2}\paren{\frac{\LpNorm{ L^{\cc{B}}_{n,m}\paren{\rho_2}}{2}^2 - \LpNorm{ L^{\cc{B}}_{n,m}\paren{\rho_1}}{2}^2}{\LpNorm{ L^{\cc{B}}_{n,m}\paren{\rho_1}}{2}^2\LpNorm{ L^{\cc{B}}_{n,m}\paren{\rho_2}}{2}^2}}\\
&\leq& \OpNorm{ L^{\cc{B}}_{n,m}\paren{\rho_2}}{2}{2}\frac{\LpNorm{ L^{\cc{B}}_{n,m}\paren{\rho_2}}{2} + \LpNorm{ L^{\cc{B}}_{n,m}\paren{\rho_1}}{2}}{\LpNorm{ L^{\cc{B}}_{n,m}\paren{\rho_1}}{2}^2\LpNorm{ L^{\cc{B}}_{n,m}\paren{\rho_2}}{2}^2}\LpNorm{ L^{\cc{B}}_{n,m}\paren{\rho_2}-L^{\cc{B}}_{n,m}\paren{\rho_1}}{2}.
\end{eqnarray*}
It is known from \eqref{UppBndOpNormLnm} that $\OpNorm{ L^{\cc{B}}_{n,m}\paren{\rho_2}}{2}{2}\lesssim \log n$. Moreover, it is easy to verify that
\begin{equation*}
\frac{\LpNorm{ L^{\cc{B}}_{n,m}\paren{\rho_2}}{2} + \LpNorm{ L^{\cc{B}}_{n,m}\paren{\rho_1}}{2}}{\LpNorm{ L^{\cc{B}}_{n,m}\paren{\rho_1}}{2}^2\LpNorm{ L^{\cc{B}}_{n,m}\paren{\rho_2}}{2}^2} = \frac{1/\LpNorm{ L^{\cc{B}}_{n,m}\paren{\rho_1}}{2} + 1/\LpNorm{ L^{\cc{B}}_{n,m}\paren{\rho_2}}{2}}{\LpNorm{ L^{\cc{B}}_{n,m}\paren{\rho_1}}{2}\LpNorm{ L^{\cc{B}}_{n,m}\paren{\rho_2}}{2}} \lesssim \frac{1/\sqrt{n}+1\sqrt{n}}{\sqrt{n}\sqrt{n}} \asymp n^{-3/2}.
\end{equation*}
Thus, the upper bound on $\varPsi^2_n\paren{\rho_1,\rho_2}$ can be simplified as 
\begin{eqnarray*}
\frac{\varPsi^2_n\paren{\rho_1,\rho_2}}{\abs{\rho_2-\rho_1} } &\leq& \frac{\log n}{n^{3/2}} \paren{\frac{\LpNorm{ L^{\cc{B}}_{n,m}\paren{\rho_2}-L^{\cc{B}}_{n,m}\paren{\rho_1}}{2}}{\abs{\rho_2-\rho_1}}}\\
&\RelNum{\paren{c}}{\lesssim}&  \frac{\log n}{n^{3/2}}\paren{\bbM{1}_{\set{d=1}} + \bbM{1}_{\set{d=2}} \log n + \bbM{1}_{\set{d=3}} n^{1/3} + \bbM{1}_{\set{d\geq 4}} n^{1/2} } \\
&=& \frac{\log n}{n}\paren{\bbM{1}_{\set{d=1}}\frac{1}{\sqrt{n}}+ \bbM{1}_{\set{d=2}}\frac{\log n}{\sqrt{n}} + \bbM{1}_{\set{d=3}} n^{-1/6} + \bbM{1}_{\set{d>3}} } \lesssim \frac{\log n}{n},
\end{eqnarray*}
where the inequality $\paren{c}$ follows from Lemma \ref{SensAnalFrobNorm}. In summary, \eqref{CondBeq1} can be rewritten as
\begin{equation*}
\frac{\OpNorm{M^{\cc{B}}_{n,m}\paren{\rho_2}-M^{\cc{B}}_{n,m}\paren{\rho_1}}{2}{2}}{\abs{\rho_2-\rho_1}}\leq n\log n\paren{ \frac{\varPsi^1_n\paren{\rho_1,\rho_2} + \varPsi^2_n\paren{\rho_1,\rho_2}}{\abs{\rho_2-\rho_1} } } \lesssim n\log n \frac{\log n}{n} = \log^2 n,
\end{equation*}
showing that the condition $\paren{b}$ of Proposition \ref{ModalProp} holds.\\
\emph{Validating condition $\paren{c}$}. Choose an arbitrary $\rho\in\Theta_0$. We should prove that $V_n$, which is defined as the following, converging to zero as $n$ goes to infinity.
\begin{equation}\label{Vn}
V_n \coloneqq \OpNorm{M^{\cc{B}}_{n,m}\paren{\rho}}{2}{2} \sqrt{\frac{\log n}{n}}.
\end{equation}
$V_n$ can be equivalently written as 
\begin{equation*}
V_n = \frac{\OpNorm{\sqrt{L_{n,m}\paren{\rho_0}} L^{\cc{B}}_{n,m}\paren{\rho} \sqrt{L_{n,m}\paren{\rho_0}}}{2}{2}\sqrt{n\log n}}{\LpNorm{ L^{\cc{B}}_{n,m}\paren{\rho} }{2}^2}.
\end{equation*}
Lemma \ref{LowBndFrobNormLnm}, which says the Frobenius norm of $ L_{n,m}\paren{\rho}$ is of order $\sqrt{n}$ (uniformly on $\Theta_0$) provides a simpler asymptotic expression for $V_n$.
\begin{equation*}
V_n \asymp  \OpNorm{\sqrt{L_{n,m}\paren{\rho_0}} L^{\cc{B}}_{n,m}\paren{\rho} \sqrt{L_{n,m}\paren{\rho_0}}}{2}{2} \sqrt{\frac{\log n}{n}} \leq \OpNorm{L^{\cc{B}}_{n,m}\paren{\rho} }{2}{2} \OpNorm{L_{n,m}\paren{\rho_0} }{2}{2} \sqrt{\frac{\log n}{n}}.
\end{equation*}
We refer the reader to Eq. \eqref{UppBndOpNormLnm} for an upper bound on the operator norm of $L_{n,m}$ and $L^{\cc{B}}_{n,m}$ matrices over $\Theta_0$. So, $V_n$ can be bounded above by
\begin{equation}\label{VnRate}
V_n \lesssim \paren{1+\bbM{1}_{\set{m=\nu+d/2}}\log n}^2\sqrt{\frac{\log n}{n}}\rightarrow 0,\quad\mbox{as}\; n\rightarrow\infty.
\end{equation}
\end{proof}

\begin{proof}[Proof of Theorem \ref{AsympNormLocInvFreeThm}]
	
Let $\rho_{\max}$ and $\rho_{\min}$ respectively denote the largest and smallest element of $\Theta_0$. Recall the positive semi-definite class of matrices $L^{\cc{B}}_{n,m}\paren{\rho} \coloneqq \rho^{2\nu} K^{\cc{B}}_{n,m}\paren{\rho},\;\rho\in\Theta_0$. Moreover, define
\begin{equation}\label{EqFormffTn}
T_n\paren{\rho,Y} \coloneqq \sqrt{n}\paren{\frac{\hat{\phi}_{n,\cc{B}}\paren{\rho}\rho^{-2\nu}}{\phi_0\rho^{-2\nu}_0}- 1} = \sqrt{n}\paren{\frac{Y^\top L^{\cc{B}}_{n,m}\paren{\rho}Y}{\phi_0\rho^{-2\nu}_0\LpNorm{L^{\cc{B}}_{n,m}\paren{\rho}}{2}^2}- 1}.
\end{equation}
For notational convenience, the dependence to $\phi_0,\rho_0$ and $m$ has been dropped in $T_n$. We aim to show that $\sigma^{-1}_n T_n\paren{\hat{\rho}_n,Y}\cp{d} N\paren{0,1}$ for some scalar bounded sequence $\sigma_n$. The proof is broken into two parts for easier digestion. We first find probabilistic upper and lower bounds on $T_n\paren{\hat{\rho}_n,Y}$ in terms of $T_n\paren{\rho_{\max},Y}$ and $T_n\paren{\rho_{\min},Y}$. The precise statement of this claim is as following.

\begin{clawithinpf}\label{Claim1Thm2}
There are non-negative sequences of random variables $\set{p_n}^{\infty}_{n=1}$ and $\set{q_n}^{\infty}_{n=1}$ converging to zero in probability and scalar $n_0\in\bb{N}$ (depending on $\rho_0,m,d,\nu$, and $\Theta_0$) such that for any $n\geq n_0$
\begin{equation}\label{BndsT_n}
T_n\paren{\rho_{\min},Y}\paren{1-p_n} \leq T_n\paren{\hat{\rho}_n,Y}\leq T_n\paren{\rho_{\max},Y}\paren{1+q_n}.
\end{equation}
\end{clawithinpf}

\noindent Next, we substantiate the asymptotic normality of $T_n\paren{\rho,Y}$ for an arbitrary $\rho\in\Theta_0$. 

\begin{clawithinpf}\label{Claim2Thm2}
There is a bounded sequence $\sigma_{n,m}$ such that $\frac{1}{\sigma_{n,m}}T_n\paren{\rho,Y}\cp{d} N\paren{0,1}$, for any fixed $\rho\in\Theta_0$.
\end{clawithinpf}

\noindent As both upper and lower bounds on $\sigma^{-1}_{n,m}T_n\paren{\hat{\rho}_n,Y}$ in \eqref{BndsT_n} weakly converge to a random variable distributed as $N\paren{0,1}$, the squeeze theorem for the weak convergence (see Lemma \ref{SqueezeThm} for its rigorous statement) concludes the proof. The rest of the proof serves to establish Claims \ref{Claim1Thm2} and \ref{Claim2Thm2}.

\begin{proof}[Proof of Claim \ref{Claim1Thm2}]
Define $T'_n\paren{\rho} \coloneqq 1+ T_n\paren{\rho,Y}/\sqrt{n}$. Claim \ref{Claim2Thm2} obviously holds if we can show that 
\begin{equation}\label{ffT'_nUppLowBnd}
T'_n\paren{\rho_{\min}}\paren{1-p'_n} \leq T'_n\paren{\hat{\rho}_n}\leq T'_n\paren{\rho_{\max}}\paren{1+q'_n},
\end{equation}
for any realization of $Y$ and for sequences $\set{p'_n}^{\infty}_{n=1}, \set{q'_n}^{\infty}_{n=1}$ converging to zero faster than $n^{-1/2}$. Let $Z$ be a standard Gaussian column vector with the same length as $Y$. Define $U \coloneqq \sqrt{L_{n,m}\paren{\rho_0}}Z$, which obviously has no dependence on $\rho$. Then,
\begin{equation}\label{ffT'_nEqForm}
T'_n\paren{\rho} =  \frac{U^\top L^{\cc{B}}_{n,m}\paren{\rho}U}{\LpNorm{L^{\cc{B}}_{n,m}\paren{\rho}}{2}^2},
\end{equation}
We only prove the right hand side inequality in  Eq. \eqref{ffT'_nUppLowBnd} and the other side can be shown similarly. We separately analyze the numerator and denominator in \eqref{ffT'_nEqForm}. We know that $L^{\cc{B}}_{n,m}\paren{\rho}\preceq L^{\cc{B}}_{n,m}\paren{\rho_{\max}}$ for any $\rho\in\Theta_0$ (see \eqref{PosDefDelta} for the details). Thus, $U^\top L^{\cc{B}}_{n,m}\paren{\rho}U\leq U^\top L^{\cc{B}}_{n,m}\paren{\rho_{\max}}U$ almost surely. Namely,
\begin{equation}\label{ffT'_nrhoUppBnd}
T'_n\paren{\rho} \leq \frac{U^\top L^{\cc{B}}_{n,m}\paren{\rho_{\max}}U}{\LpNorm{L^{\cc{B}}_{n,m}\paren{\rho}}{2}^2} \quad\Leftrightarrow \quad \set{\frac{T'_n\paren{\rho}}{T'_n\paren{\rho_{\max}}}-1}\leq \frac{\LpNorm{L^{\cc{B}}_{n,m}\paren{\rho_{\max}}}{2}^2-\LpNorm{L^{\cc{B}}_{n,m}\paren{\rho}}{2}^2}{{\LpNorm{L^{\cc{B}}_{n,m}\paren{\rho}}{2}^2}}.
\end{equation}
Recall that we have defined $\Delta^{\cc{B}}\paren{\rho_2,\rho_1} \coloneqq L^{\cc{B}}_{n,m}\paren{\rho_2} - L^{\cc{B}}_{n,m}\paren{\rho_1}$, for any $\rho_1,\rho_2\in\Theta_0$. It is sufficient to show that
\begin{equation}\label{q'n}
q'_n\coloneqq \frac{\LpNorm{L^{\cc{B}}_{n,m}\paren{\rho_{\max}}}{2}^2-\LpNorm{L^{\cc{B}}_{n,m}\paren{\rho}}{2}^2}{{\LpNorm{L^{\cc{B}}_{n,m}\paren{\rho}}{2}^2}} = o\paren{\frac{1}{\sqrt{n}}},\quad\mbox{as}\;n\rightarrow\infty.
\end{equation}
As we know from Lemma \ref{LowBndFrobNormLnm} that $\LpNorm{L^{\cc{B}}_{n,m}\paren{\rho}}{2} \gtrsim \sqrt{n}$, we just need to show that 
\begin{equation*}
\psi_n\coloneqq \LpNorm{L^{\cc{B}}_{n,m}\paren{\rho_{\max}}}{2}^2-\LpNorm{L^{\cc{B}}_{n,m}\paren{\rho}}{2}^2 = o\paren{\sqrt{n}},\quad\mbox{as}\;n\rightarrow\infty.
\end{equation*}
On the other hand we have
\begin{eqnarray*}
\LpNorm{L^{\cc{B}}_{n,m}\paren{\rho_{\max}}}{2}^2-\LpNorm{L^{\cc{B}}_{n,m}\paren{\rho}}{2}^2 &=& \LpNorm{L^{\cc{B}}_{n,m}\paren{\rho_{\max}}}{2}^2- \LpNorm{L^{\cc{B}}_{n,m}\paren{\rho_{\max}} - \Delta^{\cc{B}}\paren{\rho_{\max},\rho}}{2}^2\\
&\leq&  2\InnerProd{L^{\cc{B}}_{n,m}\paren{\rho_{\max}}}{\Delta^{\cc{B}}\paren{\rho_{\max},\rho}}\\
&\leq& 2\OpNorm{L^{\cc{B}}_{n,m}\paren{\rho_{\max}}}{2}{2} \SpNorm{\Delta^{\cc{B}}\paren{\rho_{\max},\rho}}{1}.
\end{eqnarray*}
Eq. \eqref{UppBndOpNormLnm} provides an upper bounds on  $\OpNorm{L^{\cc{B}}_{n,m}\paren{\rho_{\max}}}{2}{2}$. So
\begin{eqnarray*}
\psi_n &\leq& 2\OpNorm{L^{\cc{B}}_{n,m}\paren{\rho_{\max}}}{2}{2} \SpNorm{\Delta^{\cc{B}}\paren{\rho_{\max},\rho}}{1} \lesssim \paren{1+ \bbM{1}_{\set{m=\nu+d/2}}\log n}\SpNorm{\Delta^{\cc{B}}\paren{\rho_{\max},\rho}}{1} \\
&\leq& \SpNorm{\Delta^{\cc{B}}\paren{\rho_{\max},\rho}}{1}\log n.
\end{eqnarray*}
We now employ analogous techniques as Eq. \eqref{UppBndNucNormDelta} (see also Lemma \ref{SensAnalS1Norm}) to control $\SpNorm{\Delta^{\cc{B}}\paren{\rho_{\max},\rho}}{1}$ from above. Since we only consider the case of $d\leq 3$, the bound in Eq. \eqref{UppBndNucNormDelta} can be rewritten as the following.
\begin{equation}\label{gammaDefn}
\exists\; 0<\gamma<\frac{1}{2},\;\;\suchthat\;\; \SpNorm{\Delta^{\cc{B}}\paren{\rho_{\max},\rho}}{1}\lesssim n^\gamma.
\end{equation}
Thus $\psi_n$ can be upper bounded by $\psi_n \lesssim n^\gamma\log n = o\paren{\sqrt{n}}$, which concludes the proof.   
\end{proof}

\begin{proof}[Proof of Claim \ref{Claim2Thm2}]
For brevity let $\xi_n \coloneqq T_n\paren{\rho,Y}+\sqrt{n}$. We suppress the dependence of $\rho$ and $Y$ on $\xi_n$. Let us decompose $T_n\paren{\rho,Y}$ into two parts as
\begin{eqnarray}\label{LHSEqForm}
T_n\paren{\rho,Y} &=& \paren{\frac{T_n\paren{\rho,Y}-\bb{E}T_n\paren{\rho,Y}}{\sqrt{\var T_n\paren{\rho,Y}}}} \sqrt{\var T_n\paren{\rho,Y}} + \bb{E}T_n\paren{\rho,Y}\nonumber\\
&=& \paren{\frac{\xi_n-\bb{E}\xi_n}{\sqrt{\var \xi_n}}} \sqrt{\var \xi_n} + \bb{E}T_n\paren{\rho,Y}.
\end{eqnarray}
Recall that we defined $P_1 \coloneqq \sup_{\rho\in\Theta_0} n^{-1/2}\bb{E}T_n\paren{\rho,Y}$ in the proof of Theorem \ref{CnstncyLocInvFreeThm}. A prudent look at Eqs. \eqref{UppBndP1GenCase} and \eqref{UppBndP1RegLatt} reveals that $P_1\lesssim n^{\gamma-1} \log n$ for some $\gamma<1/2$ ($\gamma$ is the same as in \eqref{gammaDefn}). Hence,
\begin{equation*}
\bb{E}T_n\paren{\rho,Y}\leq \sqrt{n}P_1 \lesssim n^{-1/2+\gamma}\log n\rightarrow 0,\quad\mbox{as}\; n\rightarrow\infty.
\end{equation*}
Namely, $\bb{E}T_n\paren{\rho,Y}$ tends to zero as $n$ grows to infinity. Thus, it is sufficient to obtain the asymptotic distribution of the first term in the right hand side of \eqref{LHSEqForm}. Now we express $\xi_n$ as a quadratic term of a Gaussian random vector. Using identity \eqref{EqFormffTn}, one can easily show that
\begin{equation}\label{xi_nEqForm}
\xi_n \eqd Z^\top \frac{M^{\cc{B}}_{n,m}\paren{\rho}}{\sqrt{n}} Z,
\end{equation}
in which $Z$ is a standard Gaussian vector of proper size and $M^{\cc{B}}_{n,m}\paren{\rho}$ has been defined in \eqref{ffM}. The explicit expressions for the expected value and standard deviation of $\xi_n$ are given by
\begin{equation*}
\bb{E}\xi_n = \sqrt{\frac{1}{n}} \tr\set{M^{\cc{B}}_{n,m}\paren{\rho}},\quad \sqrt{\var \xi_n} = \sqrt{\frac{2}{n}}\LpNorm{M^{\cc{B}}_{n,m}\paren{\rho}}{2}.
\end{equation*}
We showed in the proof of Theorem \ref{CnstncyLocInvFreeThm} that $\OpNorm{M^{\cc{B}}_{n,m}\paren{\rho}}{2}{2}/\LpNorm{M^{\cc{B}}_{n,m}\paren{\rho}}{2} \rightarrow 0$ when $n\rightarrow\infty$ (see \eqref{Vn} and \eqref{VnRate}). Thus applying Lemma $A.4$ of \cite{keshavarz2016consistency}, on asymptotic normality of the normalized generalized $\chi^2$ random variables, leads to
\begin{equation*}
\paren{\frac{\xi_n-\bb{E}\xi_n}{\sqrt{\var \xi_n}}} \cp{d} N\paren{0,1}.
\end{equation*}
Finally we study the limiting behaviour of $\sqrt{\var \xi_n}$, which is denoted by $\sigma_{n,m}\paren{\rho,\rho_0}$. Notice that
\begin{equation*}
\sigma_{n,m}\paren{\rho,\rho_0} \coloneqq \sqrt{\frac{2}{n}}\LpNorm{M^{\cc{B}}_{n,m}\paren{\rho}}{2} = \frac{\sqrt{2n}}{\LpNorm{L^{\cc{B}}_{n,m}\paren{\rho}}{2}^2}\LpNorm{\sqrt{L_{n,m}\paren{\rho_0}} L^{\cc{B}}_{n,m}\paren{\rho} \sqrt{L_{n,m}\paren{\rho_0}}}{2}.
\end{equation*}
We claim that 
\begin{equation}\label{Varofxi_n}
\lim\limits_{n\rightarrow\infty} \frac{\sigma_{n,m}\paren{\rho,\rho_0}}{\sigma_{n,m}\paren{\rho_1,\rho_2}}= 1,\quad \forall\;\rho_1,\rho_2\in\Theta_0.
\end{equation}
Thus, $\sigma_{n,m}$ has no dependence to $\rho$, $\rho_0$, and $\Theta_0$. In other words, $\sigma_{n,m}$ only depends on $m,d,\nu$, and the topology of $\cc{D}_n$ . Assuming that the claim holds, for proving the boundedness of $\sigma_{n,m}$, we just need to check that $\sigma_{n,m}\paren{\rho,\rho_0} \asymp 1$ for some $\rho'_1,\rho'_2\in\Theta_0$. Applying Lemma \ref{LowBndFrobNormLnm} on the denominator of $\sigma_{n,m}\paren{\rho'_1,\rho'_2}$, we get,
\begin{equation*}
f_{n,m}\paren{\rho'_1,\rho'_2} \lesssim   \frac{\LpNorm{\sqrt{L_{n,m}\paren{\rho'_2}} L^{\cc{B}}_{n,m}\paren{\rho'_1} \sqrt{L_{n,m}\paren{\rho'_2}}}{2}}{\sqrt{n}}.
\end{equation*}
So, $\sigma_{n,m}\paren{\rho'_1,\rho'_2}\asymp 1$ as a result of Lemma \ref{FrobNormAsympNumeratMnm}. We now turn to substantiate \eqref{Varofxi_n}. It is sufficient to verify the following identities for any $\rho_1,\rho_2\in\Theta_0$.
\begin{equation}\label{Varofxi_n2}
\lim\limits_{n\rightarrow\infty} \frac{\sigma_{n,m}\paren{\rho,\rho_0}}{\sigma_{n,m}\paren{\rho_1,\rho_0}}= 1,\quad\lim\limits_{n\rightarrow\infty} \frac{\sigma_{n,m}\paren{\rho_1,\rho_0}}{\sigma_{n,m}\paren{\rho_1,\rho_2}}= 1.
\end{equation}
To avoid repetition, we only demonstrate the left hand side identity in \eqref{Varofxi_n2} and the other one can be substantiated using analogous techniques. Observe that
\begin{equation*}
\frac{\sigma_{n,m}\paren{\rho,\rho_0}}{\sigma_{n,m}\paren{\rho_1,\rho_0}} = \brac{\frac{ \LpNorm{L^{\cc{B}}_{n,m}\paren{\rho_1}}{2} }{\LpNorm{L^{\cc{B}}_{n,m}\paren{\rho}}{2}}}^2  \frac{\LpNorm{\sqrt{L_{n,m}\paren{\rho_0}} L^{\cc{B}}_{n,m}\paren{\rho} \sqrt{L_{n,m}\paren{\rho_0}}}{2}}{\LpNorm{\sqrt{L_{n,m}\paren{\rho_0}} L^{\cc{B}}_{n,m}\paren{\rho_1} \sqrt{L_{n,m}\paren{\rho_0}}}{2}} \coloneqq a_n b_n.
\end{equation*}
We prove that both $a_n$ and $b_n$ converge to one as $n$ tends to infinity. Notice that $\abs{a_n-1}$ has the same limiting behaviour as $q'_n$ defined at \eqref{q'n}. So for avoiding the redundancy we just state that $\abs{a_n-1} \lesssim n^{\gamma-1} \log n = o\paren{n^{-1/2}}$ and refer the reader to the proof of Claim \ref{Claim1Thm2}. The last step of the proof is devoted to control $\abs{b_n-1}$ from above.

\begin{eqnarray*}
\abs{b_n-1} &=& \abs{\frac{\LpNorm{\sqrt{L_{n,m}\paren{\rho_0}} L^{\cc{B}}_{n,m}\paren{\rho} \sqrt{L_{n,m}\paren{\rho_0}}}{2}}{\LpNorm{\sqrt{L_{n,m}\paren{\rho_0}} L^{\cc{B}}_{n,m}\paren{\rho_1} \sqrt{L_{n,m}\paren{\rho_0}}}{2}}-1}\leq \frac{\OpNorm{L_{n,m}\paren{\rho_0}}{2}{2}\LpNorm{L^{\cc{B}}_{n,m}\paren{\rho}-L^{\cc{B}}_{n,m}\paren{\rho_1}} {2}}{\LpNorm{\sqrt{L_{n,m}\paren{\rho_0}} L^{\cc{B}}_{n,m}\paren{\rho_1} \sqrt{L_{n,m}\paren{\rho_0}}}{2}}\\
&=& \frac{\OpNorm{L_{n,m}\paren{\rho_0}}{2}{2}\LpNorm{\Delta^{\cc{B}}\paren{\rho,\rho_0}}{2}}{\LpNorm{\sqrt{L_{n,m}\paren{\rho_0}} L^{\cc{B}}_{n,m}\paren{\rho_1} \sqrt{L_{n,m}\paren{\rho_0}}}{2}} \leq \frac{\OpNorm{L_{n,m}\paren{\rho_0}}{2}{2}\SpNorm{\Delta^{\cc{B}}\paren{\rho,\rho_0}}{1}}{\LpNorm{\sqrt{L_{n,m}\paren{\rho_0}} L^{\cc{B}}_{n,m}\paren{\rho_1} \sqrt{L_{n,m}\paren{\rho_0}}}{2}}\\ 
&\RelNum{\paren{a}}{\lesssim}& \frac{\log n\SpNorm{\Delta^{\cc{B}}\paren{\rho,\rho_0}}{1}}{\LpNorm{\sqrt{L_{n,m}\paren{\rho_0}} L^{\cc{B}}_{n,m}\paren{\rho_1} \sqrt{L_{n,m}\paren{\rho_0}}}{2}} \RelNum{\paren{b}}{\lesssim} \frac{\SpNorm{\Delta^{\cc{B}}\paren{\rho,\rho_0}}{1}\log n}{\sqrt{n}}.
\end{eqnarray*}
Here $\paren{a}$ and $\paren{b}$ are successively implied from Eq. \eqref{UppBndOpNormLnm} and Lemma \ref{FrobNormAsympNumeratMnm}. Using similar techniques as Eq. \eqref{gammaDefn} implies that
\begin{equation*}
\abs{b_n-1} \lesssim \frac{\SpNorm{\Delta\paren{\rho,\rho_0}}{1}\log n}{\sqrt{n}} \lesssim \frac{n^\gamma\log n}{\sqrt{n}}\rightarrow 0,\quad\mbox{as}\; n\rightarrow\infty.
\end{equation*}
Namely $\lim\limits_{n\rightarrow\infty}b_n = 1$, which concludes the proof.
\end{proof}

\end{proof}

\appendix
\appendixpage
\makeatletter
\def\@seccntformat#1{\csname Pref@#1\endcsname \csname the#1\endcsname\quad}
\def\Pref@section{~}
\makeatother

\counterwithin{thm}{section}
\counterwithin{assu}{section}
\counterwithin{lem}{section}
\counterwithin{cor}{section}
\counterwithin{prop}{section}

\section{Large sample behavior of covariance matrices of GPs observed on irregular grids}\label{AppendixCovMatrixIrregLat}

Throughout this section, we put the following restrictions on the irregular lattice $\cc{D}_n$ with $n$ points. To avoid repetition, we omit these common assumptions in the statement of all the results in this section. Moreover, the scalars implicitly expressed in $\asymp$ and $\lesssim$ relations are bounded and generally depend on $m,d,\nu,\Theta_0$ and the topological structure of $\cc{D}_n$. 
\begin{itemize}
\item $\cc{D}_n$ is a $d$-dimensional grid satisfying Assumption \ref{RegCondLattice}. It is expedient to define $N \coloneqq \lfloor n^{1/d} \rfloor$.
\item The set of coefficients $\set{a_{m,\boldsymbol{s}}\paren{\boldsymbol{t}}: \boldsymbol{s}\in\cc{D}_n, \boldsymbol{t}\in\cc{N}_m\paren{\boldsymbol{s}}}$, admit the conditions in Definition \ref{DecorFltNonReglat}. 
\end{itemize}

Before jumping into stating the theoretical results in the subsequent sections, we recall some key assumptions and notations that we have used in the body of the paper. $G$ represents a centered, isotropic Matern GP whose one time realization has been observed at $\cc{D}_n$. The range parameters $\rho$ belongs to a compact $\Theta_0\subset\paren{0,\infty}$. We also write $\set{G_m\paren{\boldsymbol{s}}: \boldsymbol{s}\in\cc{D}_n }$ to denote the preconditioned process of order-$m$ (see Definition \ref{DecorFltNonReglat}). $m$ is chosen in such a way that $m\geq \paren{\nu+d/2}$. Let $\cc{B} = \set{B_t}^{b_n}_{t=1}$ be an arbitrary partition of $\cc{D}_n$. We have defined $K^{\cc{B}}_{n,m}\paren{\rho}$ in Eq. \eqref{Ktilde_n,m}, a matrix which is proportional to the block diagonal approximation of to the covariance of $\brac{G_m\paren{\boldsymbol{s}}:\; \boldsymbol{s}\in\cc{D}_n}$, associated to the partitioning scheme $\cc{B}$. We also define $L^{\cc{B}}_{n,m}\paren{\rho}\coloneqq \rho^{2\nu}K^{\cc{B}}_{n,m}\paren{\rho}$ for notational convenience. 

\subsection{How do the off-diagonal entries of $K^{\cc{B}}_{n,m}\paren{\rho}$ decay?}\label{OffDiagDecay}

The main objective of this section is to study the decay rate of the off-diagonal entries of $K^{\cc{B}}_{n,m}\paren{\rho}$, which comes in handy for analyzing the asymptotic behavior of different norms of $K^{\cc{B}}_{n,m}\paren{\rho}$ in Section \ref{Proofs}.For achieving this goal, we need a spectral representation for the entries of $K^{\cc{B}}_{n,m}\paren{\rho}$. For brevity define the complex valued function $f^N_{\boldsymbol{s}}:\bb{R}^d\setminus\set{\zero_d}\mapsto\bb{C}$, for any $\boldsymbol{s}\in\cc{D}_n$, by
\begin{equation}\label{fNs}
f^N_{\boldsymbol{s}}\paren{\boldsymbol{\omega}} \coloneqq \LpNorm{\boldsymbol{\omega}}{2}^{-\paren{\nu+d/2}} \sum_{ \boldsymbol{s'}\in \cc{N}_m\paren{\boldsymbol{s}} } a_{m,\boldsymbol{s}}\paren{\boldsymbol{s'}}\exp\paren{j\InnerProd{N\boldsymbol{\omega}}{\boldsymbol{s'}-\boldsymbol{s}} },\quad \forall\; \boldsymbol{\omega}\ne \zero_d,
\end{equation}
and the strictly increasing function $h_N:\paren{0,\infty}\mapsto\paren{0,1}$ with
\begin{equation}\label{hN}
h_N\paren{x} \coloneqq \brac{1+\paren{Nx}^{-2}}^{-\paren{\nu+d/2}}.
\end{equation}
Choose $\boldsymbol{s},\boldsymbol{t}\in\cc{D}_n$ arbitrarily. The entries of $K_{n,m}$ (corresponding to the single bin scenario) can be expressed in terms of the Matern spectral density.
\begin{eqnarray*}
\paren{K_{n,m}\paren{\rho}}_{\boldsymbol{s},\boldsymbol{t}} &=& \frac{N^{2\nu}}{\rho^{2\nu}}\sum_{ \boldsymbol{s'}\in \cc{N}_m\paren{\boldsymbol{s}} }\sum_{ \boldsymbol{t'}\in \cc{N}_m\paren{\boldsymbol{t}} } a_{m,\boldsymbol{s}}\paren{\boldsymbol{s'}}a_{m,\boldsymbol{t}}\paren{\boldsymbol{t'}} \int_{\bb{R}^d} e^{j\InnerProd{\boldsymbol{\omega}}{\boldsymbol{t'}-\boldsymbol{s'}} } \paren{\LpNorm{\boldsymbol{\omega}}{2}^2 +\frac{1}{\rho^2} }^{-\paren{\nu+d/2}} d\boldsymbol{\omega}\\
&=& \frac{N^{2\nu}}{\rho^{2\nu}}\sum_{ \boldsymbol{s'}\in \cc{N}_m\paren{\boldsymbol{s}} }\sum_{ \boldsymbol{t'}\in \cc{N}_m\paren{\boldsymbol{t}} } a_{m,\boldsymbol{s}}\paren{\boldsymbol{s'}}a_{m,\boldsymbol{t}}\paren{\boldsymbol{t'}} \int_{\bb{R}^d} \frac{\exp\paren{j\InnerProd{\boldsymbol{\omega}}{\boldsymbol{t'}-\boldsymbol{s'}} }}{\LpNorm{\boldsymbol{\omega}}{2}^{2\nu+d}}  h_N\paren{\frac{\rho\LpNorm{\boldsymbol{\omega}}{2}}{N}} d\boldsymbol{\omega}.
\end{eqnarray*} 
Change of variable method introduces an equivalent form of the above identity (replace $N\boldsymbol{\omega}$ instead of $\boldsymbol{\omega}$).
\begin{eqnarray}\label{SpctRepKnm}
\paren{K_{n,m}\paren{\rho}}_{\boldsymbol{s},\boldsymbol{t}} &=& \frac{1}{\rho^{2\nu}}\sum_{ \boldsymbol{s'}\in \cc{N}_m\paren{\boldsymbol{s}} }\sum_{ \boldsymbol{t'}\in \cc{N}_m\paren{\boldsymbol{t}} } a_{m,\boldsymbol{s}}\paren{\boldsymbol{s'}}a_{m,\boldsymbol{t}}\paren{\boldsymbol{t'}} \int_{\bb{R}^d} \frac{\exp\paren{j\InnerProd{N\boldsymbol{\omega}}{\boldsymbol{t'}-\boldsymbol{s'}} }}{\LpNorm{\boldsymbol{\omega}}{2}^{2\nu+d}}  h_N\paren{\rho\LpNorm{\boldsymbol{\omega}}{2}} d \boldsymbol{\omega} \nonumber\\
&=& \rho^{-2\nu} \int_{\bb{R}^d} \exp\paren{j\InnerProd{\boldsymbol{t}-\boldsymbol{s}}{\boldsymbol{\omega}} } f^N_{\boldsymbol{s}}\paren{\boldsymbol{\omega}}\overline{f^N_{\boldsymbol{t}}\paren{\boldsymbol{\omega}}} h_N\paren{\rho\LpNorm{\boldsymbol{\omega}}{2}} d\boldsymbol{\omega}.
\end{eqnarray}
Next we examine the behavior of $f^N_{\boldsymbol{s}}\paren{\cdot}$ for large $\boldsymbol{\omega}$. Such analysis is decisive for controlling the entries of $K^{\cc{B}}_{n,m}\paren{\rho}$ from above.

\begin{lem}\label{UppBndfmdLemma}
There exists $\beta \in\paren{1,\infty}$ (depending on $m,\nu,d$ and $\cc{D}_n$) such that
\begin{equation}\label{UppBndfmd}
\max_{\boldsymbol{s}\in\cc{D}_n}\abs{f^N_{\boldsymbol{s}}\paren{\boldsymbol{\omega}}}^2\leq \frac{\beta}{1+\LpNorm{\boldsymbol{\omega}}{2}^{2\nu+d}},\quad \forall\;\boldsymbol{\omega}\ne \zero_d.
\end{equation}
\end{lem}

\begin{proof}
Define the bounded integer $g_m$ by $g_m \coloneqq \max_{\boldsymbol{s}\in\cc{D}_n} \abs{\cc{N}_m\paren{\boldsymbol{s}}}$. Choose an arbitrary $\boldsymbol{s}\in\cc{D}_n$. $f^N_{\boldsymbol{s}}$ is trivially continuous and well defined at any $\boldsymbol{\omega}\ne \zero_d$, so is the function $\max_{\boldsymbol{s}\in\cc{D}_n}\abs{f^N_{\boldsymbol{s}}}^2$ (due to the continuity of the $\max$ operator). Thus for validating Eq. \eqref{UppBndfmd}, we only require to show that
\begin{enumerate}
\item $\max_{\boldsymbol{s}\in\cc{D}_n}\abs{f^N_{\boldsymbol{s}}\paren{\boldsymbol{\omega}}}^2 \lesssim \paren{1+\LpNorm{\boldsymbol{\omega}}{2}^{2\nu+d}}^{-1}$, for any $\boldsymbol{\omega}$ with $\LpNorm{\boldsymbol{\omega}}{2}^{2\nu+d} \geq g_m$.
\item There exists a bounded constant $\pi_m$ such that $\max_{\boldsymbol{s}\in\cc{D}_n} \limsup\limits_{\boldsymbol{\omega}\rightarrow\zero_d}\abs{f^N_{\boldsymbol{s}}\paren{\boldsymbol{\omega}}}^2\leq \pi_m$.
\end{enumerate}
The first claim is an implication of the Cauchy-Schwartz inequality. In Definition \ref{DecorFltNonReglat}, we normalize the coefficients $a_{m,\boldsymbol{s}}\paren{\boldsymbol{s'}}$'s to have unit Euclidean norm. Thus
\begin{eqnarray*}
\abs{f^N_{\boldsymbol{s}}\paren{\boldsymbol{\omega}}}^2 &\leq& \LpNorm{\boldsymbol{\omega}}{2}^{-\paren{2\nu+d}} \abs{\cc{N}_m\paren{\boldsymbol{s}}} \sum_{ \boldsymbol{s'}\in \cc{N}_m\paren{\boldsymbol{s}} } a^2_{m,\boldsymbol{s}}\paren{\boldsymbol{s'}} = \LpNorm{\boldsymbol{\omega}}{2}^{-\paren{2\nu+d}} \abs{\cc{N}_m\paren{\boldsymbol{s}}} \\
&\leq& g_m\LpNorm{\boldsymbol{\omega}}{2}^{-\paren{2\nu+d}} \leq \frac{1+g_m}{1+\LpNorm{\boldsymbol{\omega}}{2}^{2\nu+d}}.
\end{eqnarray*}
For proving the other claim we need to study the Taylor expansion of $f^N_{\boldsymbol{s}}$ near the origin. The second condition of Definition \ref{DecorFltNonReglat} implies that for any natural number $r<m$,
\begin{equation*}
\sum_{ \boldsymbol{s'}\in \cc{N}_m\paren{\boldsymbol{s}} } a_{m,\boldsymbol{s}}\paren{\boldsymbol{s'}} \paren{\InnerProd{\boldsymbol{\omega}}{\boldsymbol{s'}-\boldsymbol{s}} }^r = 0,\quad\forall\;\boldsymbol{\omega}\in\bb{R}^d,\;\forall\;\boldsymbol{s}\in\cc{D}_n.
\end{equation*}
So
\begin{eqnarray}\label{TayExpanNearZerofsN}
\limsup\limits_{\boldsymbol{\omega}\rightarrow\zero_d}\abs{f^N_{\boldsymbol{s}}\paren{\boldsymbol{\omega}}}^2 &=& \lim\limits_{\boldsymbol{\omega}\rightarrow\zero_d} \frac{1}{\LpNorm{\boldsymbol{\omega}}{2}^{2\nu+d}} \abs{\sum_{r=0}^{\infty} \frac{\paren{jN}^r}{r!}\sum_{ \boldsymbol{s'}\in \cc{N}_m\paren{\boldsymbol{s}} } a_{m,\boldsymbol{s}}\paren{\boldsymbol{s'}} \paren{\InnerProd{\boldsymbol{\omega}}{\boldsymbol{s'}-\boldsymbol{s}} }^r}^2 \nonumber\\
&=& \limsup\limits_{\boldsymbol{\omega}\rightarrow\zero_d} \frac{1}{{\LpNorm{\boldsymbol{\omega}}{2}^{2\nu+d}}}\abs{\sum_{r=m}^{\infty} \frac{\paren{jN}^r}{r!}\sum_{ \boldsymbol{s'}\in \cc{N}_m\paren{\boldsymbol{s}} } a_{m,\boldsymbol{s}}\paren{\boldsymbol{s'}} \paren{\InnerProd{\boldsymbol{\omega}}{\boldsymbol{s'}-\boldsymbol{s}} }^r}^2\nonumber\\
&=& \frac{N^{2m}}{m!} \limsup\limits_{\boldsymbol{\omega}\rightarrow\zero_d} \frac{1}{\LpNorm{\boldsymbol{\omega}}{2}^{2\nu+d}}\abs{\sum_{ \boldsymbol{s'}\in \cc{N}_m\paren{\boldsymbol{s}} } a_{m,\boldsymbol{s}}\paren{\boldsymbol{s'}} \paren{\InnerProd{\boldsymbol{\omega}}{\boldsymbol{s'}-\boldsymbol{s}} }^m}^2.
\end{eqnarray}
Cauchy-Schwartz inequality helps to further simplify the complex expressions in Eq. \eqref{TayExpanNearZerofsN}.
\begin{eqnarray*}
\limsup\limits_{\boldsymbol{\omega}\rightarrow\zero_d}\abs{f^N_{\boldsymbol{s}}\paren{\boldsymbol{\omega}}}^2 
&\leq& \limsup\limits_{\boldsymbol{\omega}\rightarrow\zero_d} \frac{N^{2m}\LpNorm{\boldsymbol{\omega}}{2}^{2m-2\nu-d}}{m!} \sum_{ \boldsymbol{s'}\in \cc{N}_m\paren{\boldsymbol{s}} } a^2_{m,\boldsymbol{s}}\paren{\boldsymbol{s'}} \sum_{ \boldsymbol{s'}\in \cc{N}_m\paren{\boldsymbol{s}} } \LpNorm{\boldsymbol{s'}-\boldsymbol{s}}{2}^{2m}\\
&=& \frac{\sum_{ \boldsymbol{s'}\in \cc{N}_m\paren{\boldsymbol{s}} } \LpNorm{N\paren{\boldsymbol{s'}-\boldsymbol{s}}}{2}^{2m}}{m!} \bbM{1}_{\set{2m = 2\nu+d}}.
\end{eqnarray*} 
Since $N\LpNorm{\boldsymbol{s}'-\boldsymbol{s}}{2} \asymp 1$ for any $\boldsymbol{s}'\in\cc{N}_m\paren{\boldsymbol{s}}$, then
\begin{equation*}
\exists \;\pi_m\in\paren{0,\infty}\;\; \suchthat \;\;\max_{s\in\cc{D}_n}\paren{\frac{\sum_{ \boldsymbol{s'}\in \cc{N}_m\paren{\boldsymbol{s}} } \LpNorm{N\paren{\boldsymbol{s'}-\boldsymbol{s}}}{2}^{2m}}{m!}} \leq \pi_m.
\end{equation*}
Hence, 
\begin{equation*}
\limsup\limits_{\boldsymbol{\omega}\rightarrow\zero_d}\abs{f^N_{\boldsymbol{s}}\paren{\boldsymbol{\omega}}}^2 \leq Q_m \bbM{1}_{\set{2m = 2\nu+d}} \leq Q_m.
\end{equation*}
It is straightforward to find a closed form expression for $\beta$ in terms of $g_m$ and $\pi_m$.
\end{proof}

\begin{prop}\label{UppBndCovGmNonRegLatt}
For any pair $\boldsymbol{s},\boldsymbol{t}\in\cc{D}_n$ and any partition $\cc{B}$ of $\cc{D}_n$,
\begin{equation}\label{UppBndCovGmNonRegLattIneq}
\abs{ \paren{K^{\cc{B}}_{n,m}\paren{\rho}}_{\boldsymbol{s},\boldsymbol{t}} }\lesssim \rho^{-2\nu} \paren{1+N\LpNorm{\boldsymbol{t}-\boldsymbol{s}}{2} }^{-2\paren{m-\nu}}.
\end{equation}	
\end{prop}

\begin{proof}
Without loss of generality we can assume that $\cc{B}$ has only a single bin, i.e. $\cc{B} = \set{\cc{D}_n}$. In other words, we just need to validate Eq. \eqref{UppBndCovGmNonRegLattIneq} for the entries of $K_{n,m}\paren{\rho}$. For simplicity, let $f_{\nu,\rho}$ denotes the Matern correlation function with parameters $\paren{\rho,\nu}$. Notice that $f_{\nu,\rho}\paren{x} = f_{\nu,1}\paren{x/\rho}$. We first prove the inequality \eqref{UppBndCovGmNonRegLattIneq} for the case of $\LpNorm{\boldsymbol{t}-\boldsymbol{s} }{2} = O\paren{N^{-1}}$. It suffices to show that the largest diagonal entry of $K_{n,m}\paren{\rho}$ is of order $\rho^{-2\nu}$. That is,
\begin{equation*}
\rho^{2\nu} \max_{\boldsymbol{s}\in \cc{D}_n} \abs{ \paren{K_{n,m}\paren{\rho}}_{\boldsymbol{s},\boldsymbol{s}} } \lesssim 1.
\end{equation*}
The proof of this result hinges on the inequality \eqref{SpctRepKnm} for $\boldsymbol{s} = \boldsymbol{t}$. Trivially,
\begin{equation*}
\rho^{2\nu} \max_{\boldsymbol{s}\in \cc{D}_n} \abs{ \paren{K_{n,m}\paren{\rho}}_{\boldsymbol{s},\boldsymbol{s}} } =  \max_{\boldsymbol{s}\in \cc{D}_n} \int_{\bb{R}^d}  \abs{f^N_{\boldsymbol{s}}\paren{\boldsymbol{\omega}}}^2 h_N\paren{\rho\LpNorm{\boldsymbol{\omega}}{2}} d\boldsymbol{\omega} \leq \max_{\boldsymbol{s}\in \cc{D}_n} \int_{\bb{R}^d} \abs{f^N_{\boldsymbol{s}}\paren{\boldsymbol{\omega}}}^2 d\boldsymbol{\omega}.
\end{equation*}
We finish the proof of this part by using Lemma \ref{UppBndfmdLemma}.
\begin{equation*}
\rho^{2\nu} \max_{\boldsymbol{s}\in \cc{D}_n} \abs{ \paren{K_{n,m}\paren{\rho}}_{\boldsymbol{s},\boldsymbol{s}} } \leq \max_{\boldsymbol{s}\in \cc{D}_n} \int_{\bb{R}^d} \abs{f^N_{\boldsymbol{s}}\paren{\boldsymbol{\omega}}}^2 d\boldsymbol{\omega} \lesssim \int_{\bb{R}^d}  \frac{d\boldsymbol{\omega}}{1+\LpNorm{\boldsymbol{\omega}}{2}^{2\nu+d} } \asymp  \int_{0}^{\infty} \frac{x^{d-1}}{1+x^{2\nu+d}} dx \asymp 1.
\end{equation*}
So without loss of generality we can assume that $\LpNorm{\boldsymbol{t}-\boldsymbol{s} }{2}>h/N$, for some large enough $h$. Trivially, 
\begin{equation*}
\psi\coloneqq \frac{\paren{K_{n,m}\paren{\rho}}_{\boldsymbol{s},\boldsymbol{t}}}{N^{2\nu}} = \sum_{ \boldsymbol{s'}\in \cc{N}_m\paren{\boldsymbol{s}} }\sum_{ \boldsymbol{t'}\in \cc{N}_m\paren{\boldsymbol{t}} } a_{m,\boldsymbol{s}}\paren{\boldsymbol{s'}}a_{m,\boldsymbol{t}}\paren{\boldsymbol{t'}} f_{\nu,\rho}\paren{\boldsymbol{t'}-\boldsymbol{s'}}.
\end{equation*}
The key step of the proof is to replace $f_{\nu,\rho}\paren{\cdot}$ with its exact Taylor expansion of order $2m$. Strictly speaking, we have
\begin{eqnarray*}
f_{\nu,\rho}\paren{\boldsymbol{t'}-\boldsymbol{s'}} &=& \sum_{\abs{r}<2m} \frac{D^r f_{\nu,\rho}\paren{\boldsymbol{t}-\boldsymbol{s}}}{r!} \brac{  \paren{\boldsymbol{t'}-\boldsymbol{t}} - \paren{\boldsymbol{s'}-\boldsymbol{s}} }^r \\
&+& \sum_{\abs{r}=2m} R_r\paren{\boldsymbol{t}-\boldsymbol{s}}\frac{\brac{  \paren{\boldsymbol{t'}-\boldsymbol{t}} - \paren{\boldsymbol{s'}-\boldsymbol{s}} }^r}{r!},
\end{eqnarray*}
in which $R_r$ denotes the residual function given by 
\begin{equation}\label{ResidFunc}
R_r\paren{\boldsymbol{t}-\boldsymbol{s}} = 2m\int_{0}^{1} \paren{1-x}^{2m-1} D^r f_{\nu,\rho}\Bigparen{\paren{\boldsymbol{t}-\boldsymbol{s}} + x\brac{  \paren{\boldsymbol{t'}-\boldsymbol{t}} - \paren{\boldsymbol{s'}-\boldsymbol{s}} }} dx.
\end{equation}
Thus, 
\begin{eqnarray}\label{CovGmTylrExp}
\psi &=&  \sum_{\abs{r}<2m} \frac{ D^r f_{\nu,\rho}\paren{\boldsymbol{t}-\boldsymbol{s}}}{r!} \sum_{ \boldsymbol{s'}\in \cc{N}_m\paren{\boldsymbol{s}} }\sum_{ \boldsymbol{t'}\in \cc{N}_m\paren{\boldsymbol{t}} } a_{m,\boldsymbol{s}}\paren{\boldsymbol{s'}}a_{m,\boldsymbol{t}}\paren{\boldsymbol{t'}}\brac{  \paren{\boldsymbol{t'}-\boldsymbol{t}} - \paren{\boldsymbol{s'}-\boldsymbol{s}} }^r \nonumber\\
&+& \sum_{\abs{r}=2m} \frac{ R_r\paren{\boldsymbol{t}-\boldsymbol{s}}}{r!} \sum_{ \boldsymbol{s'}\in \cc{N}_m\paren{\boldsymbol{s}} }\sum_{ \boldsymbol{t'}\in \cc{N}_m\paren{\boldsymbol{t}} } a_{m,\boldsymbol{s}}\paren{\boldsymbol{s'}}a_{m,\boldsymbol{t}}\paren{\boldsymbol{t'}}\brac{  \paren{\boldsymbol{t'}-\boldsymbol{t}} - \paren{\boldsymbol{s'}-\boldsymbol{s}} }^r.
\end{eqnarray}
The first constraint on $\set{a_{m,\boldsymbol{s}}\paren{\boldsymbol{t}}: \boldsymbol{s}\in\cc{D}_n, \boldsymbol{t}\in\cc{N}_m\paren{\boldsymbol{s}}}$ in Definition \ref{DecorFltNonReglat} easily implies that 
\begin{equation*}
\sum_{ \boldsymbol{s'}\in \cc{N}_m\paren{\boldsymbol{s}} }\sum_{ \boldsymbol{t'}\in \cc{N}_m\paren{\boldsymbol{t}} } a_{m,\boldsymbol{s}}\paren{\boldsymbol{s'}}a_{m,\boldsymbol{t}}\paren{\boldsymbol{t'}}\brac{  \paren{\boldsymbol{t'}-\boldsymbol{t}} - \paren{\boldsymbol{s'}-\boldsymbol{s}} }^r = 0.
\end{equation*}
for any $\abs{r}<2m$. So the first term in the right hand side of \eqref{CovGmTylrExp} vanishes. Henceforth, we only need control the second term from above. Observe that
\begin{equation}\label{Rom2UppBnd}
\abs{\psi}\leq \sum_{\abs{r}=2m} \abs{ \sum_{ \boldsymbol{s'}\in \cc{N}_m\paren{\boldsymbol{s}} }\sum_{ \boldsymbol{t'}\in \cc{N}_m\paren{\boldsymbol{t}} } a_{m,\boldsymbol{s}}\paren{\boldsymbol{s'}}a_{m,\boldsymbol{t}}\paren{\boldsymbol{t'}}\brac{  \paren{\boldsymbol{t'}-\boldsymbol{t}} - \paren{\boldsymbol{s'}-\boldsymbol{s}} }^r} \max_{\abs{r} = 2m} \abs{\frac{R_r\paren{\boldsymbol{t}-\boldsymbol{s}}}{r!}}.
\end{equation}
The next step is to introduce a uniform upper bound on the residual functions using Eq. \eqref{ResidFunc} and the chain rule of derivative.
\begin{eqnarray}
\max_{\abs{r} = 2m} \abs{R_r\paren{\boldsymbol{t}-\boldsymbol{s}}} &\leq& \max_{\abs{r} = 2m} \max_{x\in\brac{0,1}} \abs{D^r f_{\nu,\rho}\Bigset{\paren{\boldsymbol{t}-\boldsymbol{s}} + x\brac{  \paren{\boldsymbol{t'}-\boldsymbol{t}} - \paren{\boldsymbol{s'}-\boldsymbol{s}} }}}\nonumber\\
&\leq& \rho^{-2m} \max_{\abs{r} = 2m} \max_{x\in\brac{0,1}} \abs{D^r f_{\nu,1}\Bigset{\frac{\paren{\boldsymbol{t}-\boldsymbol{s}} + x\brac{  \paren{\boldsymbol{t'}-\boldsymbol{t}} - \paren{\boldsymbol{s'}-\boldsymbol{s}} }}{\rho}}}.
\end{eqnarray}
As the maximum distance between $\boldsymbol{s}$ and the points $\boldsymbol{s'}\in\cc{N}_m\paren{\boldsymbol{s}}$ is of order $1/N$, so we can choose $h$ large enough such that 
\begin{equation}\label{LowBndNormofResArg}
\min_{x\in\brac{0,1}}\LpNorm{\paren{\boldsymbol{t}-\boldsymbol{s}} + x\brac{  \paren{\boldsymbol{t'}-\boldsymbol{t}}}}{2}\geq \frac{\LpNorm{\boldsymbol{t}-\boldsymbol{s}}{2}}{2}.
\end{equation}
Now we apply Lemma $4$ of \cite{anderes2010consistent} to get an upper bound on $D^r f_{\nu,1}\paren{\cdot}$ in terms of the Euclidean norm of its argument. So for any $x\in\brac{0,1}$, we have
\begin{equation*}
\abs{D^r f_{\nu,1}\Bigset{\frac{\paren{\boldsymbol{t}-\boldsymbol{s}} + x\brac{  \paren{\boldsymbol{t'}-\boldsymbol{t}} - \paren{\boldsymbol{s'}-\boldsymbol{s}} }}{\rho}}} \lesssim \LpNorm{\frac{\paren{\boldsymbol{t}-\boldsymbol{s}} + x\brac{  \paren{\boldsymbol{t'}-\boldsymbol{t}} - \paren{\boldsymbol{s'}-\boldsymbol{s}} }}{\rho}}{2}^{2\paren{\nu-m}}.
\end{equation*}
Combining this inequality and Eq. \eqref{LowBndNormofResArg} shows that for any pair $\paren{\boldsymbol{s},\boldsymbol{t}}$ with $\LpNorm{\boldsymbol{t}-\boldsymbol{s} }{2}\geq h/N$
\begin{equation}\label{Rom2UppBnd2}
\max_{\abs{r}=2m}\abs{R_r\paren{\boldsymbol{t}-\boldsymbol{s}}}\lesssim  \rho^{-2m}\paren{ \frac{\LpNorm{\boldsymbol{t}-\boldsymbol{s} }{2}}{\rho} }^{2\paren{\nu-m}} \lesssim \rho^{-2\nu} \paren
{\frac{1}{N} + \LpNorm{\boldsymbol{t}-\boldsymbol{s} }{2}}^{2\paren{\nu-m}}.
\end{equation}
Substituting \eqref{Rom2UppBnd2} into \eqref{Rom2UppBnd} yields (in which $\hat{C}^{\rho,\nu}_{m,d}$ is another bounded scalar)
\begin{equation*}
\abs{\psi}\lesssim \rho^{-2\nu}\paren{\frac{1}{N} + \LpNorm{\boldsymbol{t}-\boldsymbol{s} }{2} }^{2\paren{\nu-m}} \sum_{\abs{r}=2m} \underbrace{\abs{ \sum_{ \boldsymbol{s'}\in \cc{N}_m\paren{\boldsymbol{s}} }\sum_{ \boldsymbol{t'}\in \cc{N}_m\paren{\boldsymbol{t}} } a_{m,\boldsymbol{s}}\paren{\boldsymbol{s'}}a_{m,\boldsymbol{t}}\paren{\boldsymbol{t'}}\brac{  \paren{\boldsymbol{t'}-\boldsymbol{t}} - \paren{\boldsymbol{s'}-\boldsymbol{s}} }^r}}_{\varpi_r}.
\end{equation*}
In the sequel, we prove that $\varpi_r = \cc{O}\paren{N^{-2m}}$ for any $\abs{r} = 2m$ using the following series of inequalities. 
\begin{eqnarray*}
\varpi_r&\RelNum{\paren{a}}{\leq}& \paren{\sum_{ \boldsymbol{s'}\in \cc{N}_m\paren{\boldsymbol{s}} }a^2_{m,\boldsymbol{s}}\paren{\boldsymbol{s'}}}^{1/2}\paren{\sum_{ \boldsymbol{t'}\in \cc{N}_m\paren{\boldsymbol{t}} }a^2_{m,\boldsymbol{t}}\paren{\boldsymbol{t'}}}^{1/2} \max\set{\abs{\paren{\boldsymbol{t'}-\boldsymbol{t}} - \paren{\boldsymbol{s'}-\boldsymbol{s}}}^r :\;\begin{array}{c}
\boldsymbol{s'}\in \cc{N}_m\paren{\boldsymbol{s}}\\
\boldsymbol{t'}\in \cc{N}_m\paren{\boldsymbol{t}}
\end{array} }\\
&\RelNum{\paren{b}}{=}& \max\set{\abs{\paren{\boldsymbol{t'}-\boldsymbol{t}} - \paren{\boldsymbol{s'}-\boldsymbol{s}}}^r :\;\begin{array}{c}
\boldsymbol{s'}\in \cc{N}_m\paren{\boldsymbol{s}}\\
\boldsymbol{t'}\in \cc{N}_m\paren{\boldsymbol{t}}
\end{array} } \RelNum{\paren{c}}{=} \cc{O}\paren{N^{-2m}}.
\end{eqnarray*}
Here, $\paren{a}$ is an obvious implication of the Holder inequality. The identity $\paren{b}$ is exactly same as the second condition in Definition \ref{DecorFltNonReglat} and $\paren{c}$ holds for the class of non-regular lattices satisfying Assumption \ref{RegCondLattice}. Hence
\begin{eqnarray*}
\abs{\paren{K_{n,m}\paren{\rho}}_{\boldsymbol{s},\boldsymbol{t}}} &=& N^{2\nu}\abs{\psi}\lesssim \paren{\frac{N}{\rho}}^{2\nu} \paren{\frac{1}{N} + \LpNorm{\boldsymbol{t}-\boldsymbol{s} }{2} }^{2\paren{\nu-m}} \sum_{\abs{r}=2m} N^{-2m} \\
&\lesssim& \rho^{-2\nu} \paren{1+N\LpNorm{\boldsymbol{t}-\boldsymbol{s} }{2}}^{-2\paren{m-\nu}}
\end{eqnarray*}
\end{proof}

\subsection{Sensitivity of $L^{\cc{B}}_{n,m}\paren{\rho}$ with respect to $\rho$}

Recall that we defined $L^{\cc{B}}_{n,m}\paren{\rho}$ as the block diagonal approximation of $L_{n,m}\paren{\rho} = \rho^{2\nu} K_{n,m}\paren{\rho}$, corresponding to the partitioning scheme $\cc{B} = \set{B_t}^{b_n}_{t=1}$ of $\cc{D}_n$. This section is dedicated to study the sensitivity of $L^{\cc{B}}_{n,m}\paren{\rho}$ with respect to $\rho$, for large $n$. In other words, we are interested to study the quantity
\begin{equation*}
\frac{\norm{L^{\cc{B}}_{n,m}\paren{\rho_2}-L^{\cc{B}}_{n,m}\paren{\rho_1}}{}}{\abs{\rho_2-\rho_1}},\quad\rho_1,\rho_2\in\Theta_0,
\end{equation*}
as $n$ tends to infinity. Here $\norm{\cdot}{}$ represents either nuclear, Frobenius or operator norm. The presented results are decisive in Section \ref{Proofs}. The quantity $Q_N$, which will be defined in the next lemma, appears numerous times in this section.  

\begin{lem}\label{UppBndffQLem}
Let $\rho_1,\rho_2$ be distinct points in $\Theta_0$ such that $\rho_2> \rho_1$. Define 
\begin{equation*}
Q_N\coloneqq \int_{\bb{R}^d} \abs{f^N_{\boldsymbol{s}}\paren{\boldsymbol{\omega}} f^N_{\boldsymbol{t}}\paren{\boldsymbol{\omega}}} \abs{ h_N\paren{\rho_2\LpNorm{\boldsymbol{\omega}}{2}} - h_N\paren{\rho_1\LpNorm{\boldsymbol{\omega}}{2}} } d\boldsymbol{\omega}
\end{equation*}
Choose an arbitrary pairs of $\boldsymbol{s},\boldsymbol{t}\in\cc{D}_n$. 
\begin{equation*}
\frac{Q_N}{\rho_2-\rho_1}\lesssim \frac{\paren{\bbM{1}_{\set{d\geq 3}}+ \bbM{1}_{\set{d=2}}\log N + \bbM{1}_{\set{d=1}} N }}{N^2}.
\end{equation*}
\end{lem}

\begin{proof}
Lemma \ref{UppBndfmdLemma} provides an upper bound on the term $f^N_{\boldsymbol{s}}\paren{\boldsymbol{\omega}} f^N_{\boldsymbol{t}}\paren{\boldsymbol{\omega}}$. 
\begin{equation}\label{UppBnd1Term}
\abs{f^N_{\boldsymbol{s}}\paren{\boldsymbol{\omega}} f^N_{\boldsymbol{t}}\paren{\boldsymbol{\omega}}}\lesssim \paren{1+\LpNorm{\boldsymbol{\omega}}{2}^{2\nu+d}}^{-1}.
\end{equation}
For controlling  the other term of the integrand from above, we employ the following inequality, which will be justified later.
\begin{equation}\label{ModalIneq1Prop1}
\paren{1+x}^{-\alpha}- \paren{1+y}^{-\alpha} < \brac{\alpha\paren{y-x}}\wedge \paren{x^{-\alpha}-y^{-\alpha}},\quad\forall\;\; 0<x<y<\infty,\;\; \alpha>0.
\end{equation}
Using \eqref{ModalIneq1Prop1} (with $\alpha = \nu+\frac{d}{2}$) yields
\begin{eqnarray*}
\abs{ \frac{h_N\paren{\rho_2\LpNorm{\boldsymbol{\omega}}{2}} - h_N\paren{\rho_1\LpNorm{\boldsymbol{\omega}}{2}}}{\rho_2-\rho_1} }\leq \brac{\paren{N\LpNorm{\boldsymbol{\omega}}{2}}^{2\nu+d}\paren{\frac{\rho^{2\nu+d}_2-\rho^{2\nu+d}_1}{\rho_2-\rho_1}}}\wedge \brac{ \frac{\paren{\nu+d/2}\paren{1/\rho^2_1-1/\rho^2_2}}{\paren{ N\LpNorm{\boldsymbol{\omega}}{2}}^2\paren{\rho_2-\rho_1}}}.
\end{eqnarray*}
The fact that $\Theta_0$ is compact and does not contain zero simplify the last inequality as the following.
\begin{equation}\label{UppBnd2Term}
\abs{ \frac{h_N\paren{\rho_2\LpNorm{\boldsymbol{\omega}}{2}} - h_N\paren{\rho_1\LpNorm{\boldsymbol{\omega}}{2}}}{\rho_2-\rho_1} } \lesssim \brac{\paren{N\LpNorm{\boldsymbol{\omega}}{2}}^{2\nu+d} \wedge \paren{N\LpNorm{\boldsymbol{\omega}}{2}}^{-2} }.
\end{equation}
Combining \eqref{UppBnd1Term} and \eqref{UppBnd2Term} leads to
\begin{eqnarray}\label{UppBndffQ}
\frac{Q_N}{\paren{\rho_2-\rho_1}}&\lesssim& \int_{\bb{R}^d} \brac{\paren{N\LpNorm{\boldsymbol{\omega}}{2}}^{2\nu+d} \wedge \paren{N\LpNorm{\boldsymbol{\omega}}{2}}^{-2}} \frac{d\boldsymbol{\omega}}{1+\LpNorm{\boldsymbol{\omega}}{2}^{2\nu+d}}\nonumber\\
&\RelNum{\paren{b}}{\asymp}& \int_{0}^{\infty} \brac{\paren{Nu}^{2\nu+d} \wedge \paren{Nu}^{-2}} \frac{u^{d-1} du}{1+u^{2\nu+d}}\nonumber\\
&=& N^{2\nu+d} \int_{0}^{\frac{1}{N}} \frac{u^{2\nu+2d-1}}{1+u^{2\nu+d}} du + \frac{1}{N^2} \int_{\frac{1}{N}}^{\infty} \frac{u^{d-3}}{1+u^{2\nu+d}} du
\end{eqnarray}
The change of variable $u = \LpNorm{\boldsymbol{\omega}}{2}$ in the integral validates $\RelNum{\paren{b}}{\asymp}$. For brevity, let $\psi_1$ and $\psi_2$ stand for the two expressions in the last line of \eqref{UppBndffQ}, respectively from left to right. We ultimately introduce tight upper bounds on $\psi_1$ and $\psi_2$. Observe that
\begin{equation*}
\psi_1 = N^{2\nu+d} \int_{0}^{\frac{1}{N}} \frac{u^{2\nu+2d-1}}{ 1+u^{2\nu+d} } du\leq N^{2\nu+d}\int_{0}^{\frac{1}{N}} u^{2\nu+2d-1} du \asymp N^{2\nu+d} N^{-2\paren{\nu+d}} = N^{-d}.
\end{equation*}
Furthermore,
\begin{eqnarray*}
\psi_2 &=& \frac{1}{N^2} \int_{\frac{1}{N}}^{\infty} \frac{u^{d-3}}{ 1+u^{2\nu+d} } du = \frac{1}{N^2} \brac{\int_{\frac{1}{N}}^{1} \frac{u^{d-3}}{ 1+u^{2\nu+d} } du + \int_{1}^{\infty} \frac{u^{d-3}}{ 1+u^{2\nu+d} } du}\\
&\leq& \frac{1}{N^2} \brac{ \int_{\frac{1}{N}}^{1} u^{d-3} du+ \int_{1}^{\infty} u^{-\paren{2\nu+3}} du }\lesssim \frac{1}{N^2}\paren{\int_{\frac{1}{N}}^{1} u^{d-3} du+1}\\
&\asymp& \frac{\Bigparen{1+ \bbM{1}_{\set{d=2}}\log N + \bbM{1}_{\set{d=1}} N}}{N^2}.
\end{eqnarray*}
Replacing the upper bounds on $\psi_1$ and $\psi_2$ into \eqref{UppBndffQ} yields
\begin{equation*}
\frac{Q_N}{\paren{\rho_2-\rho_1}}\lesssim \frac{\Bigparen{1+ \bbM{1}_{\set{d=2}}\log N + \bbM{1}_{\set{d=1}} N}}{N^2} + N^{-d} \asymp \frac{\Bigparen{1+ \bbM{1}_{\set{d=2}}\log N + \bbM{1}_{\set{d=1}} N}}{N^2}
\end{equation*}
In the sequel, we prove Eq. \eqref{ModalIneq1Prop1}. Choose an arbitrary $\alpha>0$ and define $g_1,g_2:\paren{0,\infty}\mapsto\bb{R}$ by
\begin{equation*}
g_1\paren{u} = \alpha u-\paren{1+u}^{-\alpha},\quad g_2\paren{u} = u^{-\alpha} - \paren{1+u}^{-\alpha}.
\end{equation*}
Notice that \eqref{ModalIneq1Prop1} is equivalent to the two inequalities $g_1\paren{x}< g_1\paren{y}$ and $g_2\paren{y}< g_2\paren{x}$. Namely, we need to show that both $g_1$ and $-g_2$ are strictly increasing function. For any $u\in\paren{0,\infty}$, we have
\begin{equation*}
g'_1\paren{u} = \alpha\paren{1- \paren{1+u}^{-\paren{\alpha+1}} } > 0,\quad g'_2\paren{u} = -\alpha \paren{ u^{-\paren{\alpha+1}} - \paren{1+u}^{-\paren{\alpha+1}} } < 0,
\end{equation*}
which concludes the proof.  
\end{proof}

For notational convenience and from now on define, $\Delta^{\cc{B}}\paren{\rho_1,\rho_2} \coloneqq L^{\cc{B}}_{n,m}\paren{\rho_2} - L^{\cc{B}}_{n,m}\paren{\rho_1}$, for any $\rho_1,\rho_2\in\Theta_0$. When we deal with a single bin (no partitioning), $\Delta$ and $L$ respectively refer to $\Delta^{\cc{B}}$ and $L^{\cc{B}}$. 

\begin{lem}\label{SensAnalS1Norm}
Choose $\rho_1,\rho_2\in\Theta_0$ such that $\rho_2\ne \rho_1$. Then
\begin{equation}\label{SensAnalLnm}
\frac{\SpNorm{\Delta^{\cc{B}}\paren{\rho_1,\rho_2}}{1}}{\abs{\rho_2-\rho_1}}\lesssim \paren{\bbM{1}_{\set{d=1}} + \bbM{1}_{\set{d=2}} \log N + \bbM{1}_{\set{d\geq 3}} N^{d-2}}.
\end{equation}
Furthermore for any $d\geq 3$, $\SpNorm{\Delta^{\cc{B}}\paren{\rho_1,\rho_2}}{1} \asymp N^{d-2}\abs{\rho_2-\rho_1}$.
\end{lem}

\begin{proof}
Without loss of generality assume that $\rho_2>\rho_1$. We claim that $\Delta^{\cc{B}}\paren{\rho_1,\rho_2}$ is a positive semi-definite matrix. If such property holds then $\cc{S}_1$ norm and trace are the same. Namely the absolute sum of eigenvalues can be expressed only in terms of the diagonal entries. To see this is so begin by obtaining the spectral representation for the entries of $\Delta^{\cc{B}}$. Recall $f^N_{\boldsymbol{s}}\paren{\cdot}$ and $h_N\paren{\cdot}$ from Eq. \eqref{fNs} and \eqref{hN}, respectively. Now choose an arbitrary unit norm vector $v\in \bb{R}^{n}$ ($n = \abs{D_n}$). Observe that 
\begin{eqnarray}\label{PosDefDelta}
v^\top\Delta^{\cc{B}}\paren{\rho_1,\rho_2}v &=& \sum_{\boldsymbol{s},\boldsymbol{t}\in\cc{D}_n} v_{\boldsymbol{s}}v_{\boldsymbol{t}} \paren{\Delta^{\cc{B}}\paren{\rho_1,\rho_2}}_{\boldsymbol{s},\boldsymbol{t}} =  \sum_{\boldsymbol{s},\boldsymbol{t}\in\cc{D}_n} v_{\boldsymbol{s}}v_{\boldsymbol{t}} \brac{\rho^{2\nu}_2 \paren{K^{\cc{B}}_{n,m}\paren{\rho_2}}_{\boldsymbol{s},\boldsymbol{t}} - \rho^{2\nu}_1\paren{K^{\cc{B}}_{n,m}\paren{\rho_1}}_{\boldsymbol{s},\boldsymbol{t}} }\nonumber\\
&\RelNum{\paren{a}}{=}& \sum_{t=1}^{b_n} \sum_{\boldsymbol{s}\in B_t}v_{\boldsymbol{s}}v_{\boldsymbol{t}}\brac{\rho^{2\nu}_2 \paren{K_{n,m}\paren{\rho_2}}_{\boldsymbol{s},\boldsymbol{t}} - \rho^{2\nu}_1\paren{K_{n,m}\paren{\rho_1}}_{\boldsymbol{s},\boldsymbol{t}} }\nonumber\\
&\RelNum{\paren{b}}{=}& \sum_{t=1}^{b_n}\int_{\bb{R}^d} \sum_{\boldsymbol{s},\boldsymbol{t}\in B_t} v_{\boldsymbol{s}}v_{\boldsymbol{t}} e^{j\InnerProd{\boldsymbol{t}-\boldsymbol{s}}{N\boldsymbol{\omega}}}f^N_{\boldsymbol{s}}\paren{\boldsymbol{\omega}}\overline{f^N_{\boldsymbol{t}}\paren{\boldsymbol{\omega}}} \Bigbrac{ h_N\paren{\rho_2\LpNorm{\boldsymbol{\omega}}{2}}-h_N\paren{\rho_1\LpNorm{\boldsymbol{\omega}}{2}} } d\boldsymbol{\omega}\nonumber\\
&=& \sum_{t=1}^{b_n} \int_{\bb{R}^d} \abs{\sum_{\boldsymbol{s}\in B_t}v_{\boldsymbol{s}}e^{j\InnerProd{\boldsymbol{s}}{N\boldsymbol{\omega}}}f^N_{\boldsymbol{s}}\paren{\boldsymbol{\omega}} }^2\Bigbrac{ h_N\paren{\rho_2\LpNorm{\boldsymbol{\omega}}{2}}-h_N\paren{\rho_1\LpNorm{\boldsymbol{\omega}}{2}} } d\boldsymbol{\omega} \RelNum{\paren{c}}{>} 0.
\end{eqnarray}
in which $\paren{a}$ follows from the fact that $(K^{\cc{B}}_{n,m}\paren{\rho_2})_{\boldsymbol{s},\boldsymbol{t}}=0$ when $\boldsymbol{s}$ and $\boldsymbol{t}$ belong to distinct bins. The identity $\paren{b}$ is a simple application of Eq. \eqref{SpctRepKnm}. Furthermore, inequality $\paren{c}$ follows from the monotonicity of $h_N$. Now obviously we have
\begin{equation*}
\abs{\cc{D}_n}\min_{ \boldsymbol{s}\in\cc{D}_n } \abs{\paren{\Delta^{\cc{B}}\paren{\rho_1,\rho_2}}_{\boldsymbol{s},\boldsymbol{s}}}\leq\SpNorm{\Delta^{\cc{B}}\paren{\rho_1,\rho_2}}{1} = \tr\paren{ \Delta^{\cc{B}}\paren{\rho_1,\rho_2} } \leq  \abs{\cc{D}_n}\max_{ \boldsymbol{s}\in\cc{D}_n } \abs{\paren{\Delta^{\cc{B}}\paren{\rho_1,\rho_2}}_{\boldsymbol{s},\boldsymbol{s}}}.
\end{equation*}
The rest of the proof is devoted to study the behavior of the diagonal entries of $\Delta^{\cc{B}}\paren{\rho_1,\rho_2}$. We need to show that
\begin{align*}
&\abs{\frac{\paren{\Delta^{\cc{B}}\paren{\rho_1,\rho_2}}_{\boldsymbol{s},\boldsymbol{s}}}{\rho_2-\rho_1}} \lesssim N^{-2}\paren{\bbM{1}_{\set{d\geq 3}}+ \bbM{1}_{\set{d=2}}\log N + \bbM{1}_{\set{d=1}} N },\quad \forall\; \boldsymbol{s}\in\cc{D}_n,\\
&\abs{\frac{\paren{\Delta^{\cc{B}}\paren{\rho_1,\rho_2}}_{\boldsymbol{s},\boldsymbol{s}}}{\rho_2-\rho_1}} \gtrsim N^{-2},\quad \forall\; \boldsymbol{s}\in\cc{D}_n,\mbox{and}\; \forall\; d\geq 3.
\end{align*}
Applying similar techniques as \eqref{PosDefDelta} as well as Lemma \ref{UppBndffQLem} yields
\begin{eqnarray*}
\max_{ \boldsymbol{s}\in\cc{D}_n } \abs{\frac{\paren{\Delta^{\cc{B}}\paren{\rho_1,\rho_2}}_{\boldsymbol{s},\boldsymbol{s}}}{\rho_2-\rho_1}}
&=& \max_{ \boldsymbol{s}\in\cc{D}_n } \abs{\int_{\bb{R}^d}  \abs{f^N_{\boldsymbol{s}}\paren{\boldsymbol{\omega}}}^2 \Bigbrac{ \frac{h_N\paren{\rho_2\LpNorm{\boldsymbol{\omega}}{2}}-h_N\paren{\rho_1\LpNorm{\boldsymbol{\omega}}{2}}}{\rho_2-\rho_1} } d\boldsymbol{\omega}}\nonumber\\
&\lesssim& N^{-2}\paren{\bbM{1}_{\set{d\geq 3}}+ \bbM{1}_{\set{d=2}}\log N + \bbM{1}_{\set{d=1}} N }.
\end{eqnarray*}
We now proceed to establish the desired lower bound on $\tr(\Delta^{\cc{B}}\paren{\rho_1,\rho_2})$. Choose any $s\in\cc{D}_n$. Then, 
\begin{eqnarray}\label{LowBndDiagDelta}
\paren{\Delta^{\cc{B}}\paren{\rho_1,\rho_2}}_{\boldsymbol{s},\boldsymbol{s}} &=& \int_{\bb{R}^d}  \abs{f^N_{\boldsymbol{s}}\paren{\boldsymbol{\omega}}}^2 \Bigbrac{ h_N\paren{\rho_2\LpNorm{\boldsymbol{\omega}}{2}}-h_N\paren{\rho_1\LpNorm{\boldsymbol{\omega}}{2}} } d\boldsymbol{\omega} \nonumber\\
&\geq& \int_{\LpNorm{\boldsymbol{\omega}}{2} \geq 1}  \abs{f^N_{\boldsymbol{s}}\paren{\boldsymbol{\omega}}}^2 \Bigbrac{ h_N\paren{\rho_2\LpNorm{\boldsymbol{\omega}}{2}}-h_N\paren{\rho_1\LpNorm{\boldsymbol{\omega}}{2}} } d\boldsymbol{\omega}
\end{eqnarray}
Let us control $h_N\paren{\rho_2\LpNorm{\boldsymbol{\omega}}{2}}-h_N\paren{\rho_1\LpNorm{\boldsymbol{\omega}}{2}}$ from below. Due to the fact that (its proof is similar to \eqref{ModalIneq1Prop1} and we left it to the reader)
\begin{equation*}
\paren{1+x}^{-\alpha}-\paren{1+y}^{-\alpha}\geq \frac{\alpha\paren{y-x}}{2},\quad\forall\;0< x\leq y < 2^{1/\paren{\alpha+1}}-1,
\end{equation*}
it is possible to write 
\begin{equation}\label{LowBndhN}
\frac{h_N\paren{\rho_2\LpNorm{\boldsymbol{\omega}}{2}}-h_N\paren{\rho_1\LpNorm{\boldsymbol{\omega}}{2}}}{\rho_2-\rho_1} \geq \frac{\paren{\nu+\frac{d}{2}}}{2N^2\LpNorm{\boldsymbol{\omega}}{2}^2} \frac{\rho_1+\rho_2}{\rho^2_1\rho^2_2} \gtrsim \paren{N\LpNorm{\boldsymbol{\omega}}{2}}^{-2}.
\end{equation} 
for large enough $N$. Moreover, the class of functions $\set{f^N_{\boldsymbol{s}}\paren{\boldsymbol{\omega}}}_{\boldsymbol{s}\in\cc{D}_n}$ are nonzero (in a large enough neighborhood of the origin), continuously differentiable, with a uniformly bounded derivative when $\LpNorm{\boldsymbol{\omega}}{2}\geq 1$, and decay with the polynomial rate given in Lemma \ref{UppBndfmdLemma}. So
\begin{equation}\label{PositIntffNs}
\int_{\LpNorm{\boldsymbol{\omega}}{2} \geq 1}  \abs{\frac{f^N_{\boldsymbol{s}}\paren{\boldsymbol{\omega}}}{\LpNorm{\boldsymbol{\omega}}{2}}}^2 d\boldsymbol{\omega} \asymp 1, \quad\forall\; \boldsymbol{s}\in\cc{D}_n.
\end{equation}
Replacing \eqref{PositIntffNs} and \eqref{LowBndhN} into Eq. \eqref{LowBndDiagDelta} gives the desirable lower bound.
\end{proof}

\begin{lem}\label{SensAnalOpNorm}
Let $\rho_1,\rho_2\in\Theta_0$. Then 
\begin{equation}\label{SensAnalOpLnmNonReg}
\OpNorm{\Delta^{\cc{B}}\paren{\rho_1,\rho_2}}{2}{2}\lesssim \Bigparen{1 \wedge \abs{\rho_2-\rho_1}} \paren{1+ \bbM{1}_{\set{m=\nu+d/2}}\log N}.
\end{equation}
Moreover, if $\cc{D}_n$ be a $d$-dimensional regular lattice, then
\begin{equation}\label{SensAnalOpLnmReg}
\OpNorm{\Delta^{\cc{B}}\paren{\rho_1,\rho_2}}{2}{2}\lesssim \Bigparen{1 \wedge \abs{\rho_2-\rho_1}}.
\end{equation}
\end{lem}

\begin{proof}
Consider any arbitrary partitioning $\cc{B}$. We know that $\Delta^{\cc{B}}\paren{\rho_1,\rho_2}$ is a block diagonal approximation of $\Delta\paren{\rho_1,\rho_2}$. The basic properties of operator norm implies that
\begin{equation*}
\OpNorm{\Delta^{\cc{B}}\paren{\rho_1,\rho_2}}{2}{2}\leq \OpNorm{\Delta\paren{\rho_1,\rho_2}}{2}{2}.
\end{equation*}
Hence, we just need to find an upper bound on $\OpNorm{\Delta\paren{\rho_1,\rho_2}}{2}{2}$. Without loss of generality, suppose that $\rho_2>\rho_1$. If $\rho_2-\rho_1>1$ then the positive definiteness of $\Delta\paren{\rho_1,\rho_2}$ (see \eqref{PosDefDelta}) implies that
\begin{equation}\label{LMinusDeltaPosDef}
\OpNorm{\Delta\paren{\rho_1,\rho_2}}{2}{2}\leq \OpNorm{L_{n,m}\paren{\rho_2}}{2}{2}.
\end{equation}
Now assume that $\paren{\rho_2-\rho_1}$ is strictly less than $1$. We also showed that for any unit norm column vector $v$ (of the proper size)
\begin{equation*}
\frac{v^\top \Delta\paren{\rho_1,\rho_2} v}{\rho_2-\rho_1} = \int_{\bb{R}^d} \abs{\sum_{\boldsymbol{s}\in\cc{D}_n}v_{\boldsymbol{s}}e^{j\InnerProd{\boldsymbol{s}}{N\boldsymbol{\omega}}}f^N_{\boldsymbol{s}}\paren{\boldsymbol{\omega}} }^2\set{ \frac{h_N\paren{\rho_2\LpNorm{\boldsymbol{\omega}}{2}}-h_N\paren{\rho_1\LpNorm{\boldsymbol{\omega}}{2}}}{\rho_2-\rho_1} } d\boldsymbol{\omega}.
\end{equation*}
The mean value theorem gives an alternative form for $h_N\paren{\rho_2\LpNorm{\boldsymbol{\omega}}{2}}-h_N\paren{\rho_1\LpNorm{\boldsymbol{\omega}}{2}}$.
\begin{equation*}
\exists\; \rho\in\paren{\rho_1,\rho_2}\;\;\suchthat\;\; \frac{h_N\paren{\rho_2\LpNorm{\boldsymbol{\omega}}{2}}-h_N\paren{\rho_1\LpNorm{\boldsymbol{\omega}}{2}}}{\rho_2-\rho_1} = \dot{h}_N\paren{\rho\LpNorm{\boldsymbol{\omega}}{2}} = \frac{2\nu+d}{\rho} \frac{h_N\paren{\rho\LpNorm{\boldsymbol{\omega}}{2}}}{1+ \paren{N\rho\LpNorm{\boldsymbol{\omega}}{2}}^2}.
\end{equation*}
In following identity we show that $\sup_{\rho\in\brac{\rho_1,\rho_2}} \dot{h}_N\paren{\rho\LpNorm{\boldsymbol{\omega}}{2}}\lesssim h_N\paren{\rho_2\LpNorm{\boldsymbol{\omega}}{2}}$.
\begin{eqnarray}\label{LipsConshN}
\frac{h_N\paren{\rho_2\LpNorm{\boldsymbol{\omega}}{2}}-h_N\paren{\rho_1\LpNorm{\boldsymbol{\omega}}{2}}}{\rho_2-\rho_1} &\leq& \frac{2\nu+d}{\rho_1} \frac{h_N\paren{\rho\LpNorm{\boldsymbol{\omega}}{2}}}{1+ \paren{N\rho\LpNorm{\boldsymbol{\omega}}{2}}^2} \leq \frac{2\nu+d}{\rho_1}h_N\paren{\rho\LpNorm{\boldsymbol{\omega}}{2}}\\
&\lesssim& h_N\paren{\rho_2\LpNorm{\boldsymbol{\omega}}{2}}.
\end{eqnarray}
The last inequality in \eqref{LipsConshN} is an easy consequence of the fact that $\inf\paren{\Theta_0} > 0$. Thus,
\begin{equation*}
0\leq \frac{v^\top \Delta\paren{\rho_1,\rho_2} v}{\rho_2-\rho_1} \lesssim \int_{\bb{R}^d} \abs{\sum_{\boldsymbol{s}\in\cc{D}_n}v_{\boldsymbol{s}}e^{j\InnerProd{\boldsymbol{s}}{N\boldsymbol{\omega}}}f^N_{\boldsymbol{s}}\paren{\boldsymbol{\omega}} }^2 h_N\paren{\rho_2\LpNorm{\boldsymbol{\omega}}{2}} d\boldsymbol{\omega} =  v^\top L_{n,m}\paren{\rho_2}v.
\end{equation*}
In other words, there is a bounded constant $c>1$ for which
\begin{equation}\label{LMinusDeltaPosDef2}
\frac{\Delta\paren{\rho_1,\rho_2}}{\rho_2-\rho_1}\preceq cL_{n,m}\paren{\rho_2}\quad \Rightarrow \quad \frac{\OpNorm{\Delta\paren{\rho_1,\rho_2}}{2}{2}}{\rho_2-\rho_1}\lesssim \OpNorm{L_{n,m}\paren{\rho_2}}{2}{2}.
\end{equation}
Combining \eqref{LMinusDeltaPosDef} and \eqref{LMinusDeltaPosDef2} leads to 
\begin{equation*}
\OpNorm{\Delta\paren{\rho_1,\rho_2}}{2}{2} \lesssim \Bigparen{1 \wedge \abs{\rho_2-\rho_1}}\OpNorm{L_{n,m}\paren{\rho_2}}{2}{2}.
\end{equation*}
In the case that $\cc{D}_n$ is a regular lattice, $\OpNorm{L_{n,m}\paren{\rho_2}}{2}{2}$ is known to be less than some bounded scalar $C$ (see \cite{stein2012interpolation}, Theorem $3.1$), which justifies \eqref{SensAnalOpLnmReg}. For arbitrary irregular lattices satisfying Assumption \ref{RegCondLattice}, Proposition \ref{UppBndCovGmNonRegLatt} characterizes the decay rate of the off diagonal entries of $L_{n,m}\paren{\rho_2}$. Thus, applying Lemma \ref{OpNormPolDecayOffDiag} immediately substantiates \eqref{SensAnalOpLnmReg} and ends the proof.
\end{proof}

\begin{lem}\label{SensAnalFrobNorm}
Let $N \coloneqq \lfloor n^{1/d} \rfloor$ and select two distinct $\rho_1$ and $\rho_2$ in $\Theta_0$. Then,
\begin{equation*}
\frac{\LpNorm{\Delta^{\cc{B}}\paren{\rho_1,\rho_2}}{2}}{\abs{\rho_2-\rho_1}}\lesssim \paren{\bbM{1}_{\set{d=1}} + \bbM{1}_{\set{d=2}}\log n + \bbM{1}_{\set{d=3}} n^{1/3} + \bbM{1}_{\set{d\geq 4}} n^{1/2} }.
\end{equation*}
\end{lem}

\begin{proof}
The same logic as in the proof of Lemma \ref{SensAnalOpNorm} leads to
\begin{equation*}
\LpNorm{\Delta^{\cc{B}}\paren{\rho_1,\rho_2}}{2}\leq \LpNorm{\Delta\paren{\rho_1,\rho_2}}{2}.
\end{equation*}
So it suffices to control $\LpNorm{\Delta\paren{\rho_1,\rho_2}}{2}$ from above. When $d\leq 4$, it is trivial that
\begin{equation*}
\LpNorm{\Delta\paren{\rho_1,\rho_2}}{2}\leq \SpNorm{\Delta\paren{\rho_1,\rho_2}}{1}.
\end{equation*}
Substituting the bound on $\SpNorm{\Delta\paren{\rho_1,\rho_2}}{1}$ from Lemma \ref{SensAnalS1Norm} in the above inequality leads to the desired result. Now suppose that $d\geq 5$. In this case, $1-2/d > 1/2$ and so we inevitably need new proof techniques. Without loss of generality assume that $\rho_2\geq \rho_1$. In \eqref{LMinusDeltaPosDef2}, we showed that 
\begin{equation*}
\LpNorm{\Delta\paren{\rho_1,\rho_2}}{2}\leq \LpNorm{L_{n,m}\paren{\rho_2}}{2}\paren{\rho_2-\rho_1}.
\end{equation*}
We also know from Proposition \ref{UppBndCovGmNonRegLatt} that 
\begin{equation}\label{UppBndLoffDiagEntr}
\abs{\paren{L_{n,m}\paren{\rho_2}}_{\boldsymbol{s},\boldsymbol{t}}}\lesssim \paren{1+N\LpNorm{\boldsymbol{t}-\boldsymbol{s}}{2} }^{-2\paren{m-\nu}},
\end{equation}
which means that $\LpNorm{L_{n,m}\paren{\rho_2}}{2}\lesssim \sqrt{n}$ (see the second part of Lemma \ref{OpNormPolDecayOffDiag}). In summary for $d\geq 5$,
\begin{equation*}
\LpNorm{\Delta\paren{\rho_1,\rho_2}}{2}\leq \LpNorm{L_{n,m}\paren{\rho_2}}{2} \abs{\rho_2-\rho_1} \lesssim n^{1/2} \abs{\rho_2-\rho_1}.
\end{equation*}
\end{proof}

\begin{lem}\label{LowBndFrobNormLnm}
There exists a large enough $N_0$ such that for any $N\geq N_0$,
\begin{equation*}
\min_{\rho\in\Theta_0}\frac{ \LpNorm{L^{\cc{B}}_{n,m}\paren{\rho}}{2} }{\sqrt{n}}> 0.
\end{equation*}
\end{lem}

\begin{proof}
Let $\rho_{\min}$ represents the smallest member of $\Theta_0$. We have shown in the proof of Lemma \ref{SensAnalS1Norm} (inequality \eqref{PosDefDelta}) that
\begin{equation*}
L^{\cc{B}}_{n,m}\paren{\rho} \succcurlyeq L^{\cc{B}}_{n,m}\paren{\rho_{\min}}, \quad\forall\;\rho\in\Theta_0
\end{equation*}
Henceforth, all the eigenvalues of $L^{\cc{B}}_{n,m}\paren{\rho}$ are greater than or equal to the corresponding eigenvalues of $L^{\cc{B}}_{n,m}\paren{\rho_{\min}}$. So $n^{-1/2}\LpNorm{L^{\cc{B}}_{n,m}\paren{\rho}}{2}$ attains its minimum at $\rho=\rho_{\min}$, due to the positive definiteness of $L_{n,m}\paren{\rho}$ and $L^{\cc{B}}_{n,m}\paren{\rho_{\min}}$. As $L^{\cc{B}}_{n,m}\paren{\rho_{\min}}$ is a square matrix of size $n$, it suffices to show that all of its diagonal entries are bounded away from zero.
\begin{equation*}
\LpNorm{L^{\cc{B}}_{n,m}\paren{\rho_{\min}}}{2}^2\geq \sum_{\boldsymbol{s}\in\cc{D}_n} \abs{ \paren{L^{\cc{B}}_{n,m}\paren{\rho_{\min}}}_{\boldsymbol{s},\boldsymbol{s}} }^2 = \sum_{\boldsymbol{s}\in\cc{D}_n} \abs{ \paren{L_{n,m}\paren{\rho_{\min}}}_{\boldsymbol{s},\boldsymbol{s}} }^2.
\end{equation*}
Recall the two functions $f^N_{\boldsymbol{s}}$ and $h_N$ from Eq. \eqref{fNs} and \eqref{hN}, respectively. Now choose an arbitrary $\boldsymbol{s}\in\cc{D}_n$ and a large enough $R\in\paren{0,\infty}$. From the identity \eqref{SpctRepKnm}, we have a closed form expression for the diagonal entries of $L_{n,m}\paren{\rho_{\min}}$.
\begin{equation*}
\paren{L_{n,m}\paren{\rho_{\min}}}_{\boldsymbol{s},\boldsymbol{s}} = \int_{\bb{R}^d} \abs{f^N_{\boldsymbol{s}}\paren{\boldsymbol{\omega}}}^2 h_N\paren{\rho_{\min}\LpNorm{\boldsymbol{\omega}}{2} } d\boldsymbol{\omega} > \int_{\LpNorm{\boldsymbol{\omega}}{2}\leq R} \abs{f^N_{\boldsymbol{s}}\paren{\boldsymbol{\omega}}}^2 h_N\paren{\rho_{\min}\LpNorm{\boldsymbol{\omega}}{2} } d\boldsymbol{\omega}.
\end{equation*}
We trivially can choose $N_0$ (depending on $\Theta_0$ and $R$) such that $\inf_{\LpNorm{\boldsymbol{\omega}}{2}\leq R} h_N\paren{\rho_{\min}\LpNorm{\boldsymbol{\omega}}{2} } \geq \frac{1}{2}$ for any $N\geq N_0$. Thus,
\begin{equation*}
\abs{ \paren{L_{n,m}\paren{\rho_{\min}}}_{\boldsymbol{s},\boldsymbol{s}} } > \frac{1}{2}\int_{\LpNorm{\boldsymbol{\omega}}{2}\leq R} \abs{f^N_{\boldsymbol{s}}\paren{\boldsymbol{\omega}}}^2 d\boldsymbol{\omega}.
\end{equation*}
\end{proof}

\begin{lem}\label{FrobNormAsympNumeratMnm}
There exist a strictly positive scalars $C_1$ and $C_2$ such that
\begin{equation}\label{Eq1FrobNormAsympNumeratMnm}
C_1\sqrt{n}\geq\LpNorm{ \sqrt{L_{n,m}\paren{\rho_1}} L^{\cc{B}}_{n,m}\paren{\rho_2}\sqrt{L_{n,m}\paren{\rho_1}} }{2}\geq C_2\sqrt{n},\quad\forall\;\rho_1,\rho_2\in\Theta_0.
\end{equation}
\end{lem}

\begin{proof}
For brevity we use $Q$ to refer the Frobenius norm in Eq. \eqref{Eq1FrobNormAsympNumeratMnm}. The cyclic permutation property of trace operator implies that
\begin{equation*}
\LpNorm{ \sqrt{L_{n,m}\paren{\rho_1}} L^{\cc{B}}_{n,m}\paren{\rho_2}\sqrt{L_{n,m}\paren{\rho_1}} }{2} = \LpNorm{ \sqrt{L^{\cc{B}}_{n,m}\paren{\rho_2}} L_{n,m}\paren{\rho_1}\sqrt{L^{\cc{B}}_{n,m}\paren{\rho_2}} }{2}.
\end{equation*}
The inequality \eqref{PosDefDelta} indicates that $L_{n,m}\paren{\rho_1 \vee \rho_2} \succcurlyeq L_{n,m}\paren{\rho_1}$ and $L^{\cc{B}}_{n,m}\paren{\rho_1 \vee \rho_2} \succcurlyeq L^{\cc{B}}_{n,m}\paren{\rho_2}$. So
\begin{eqnarray*}
\LpNorm{ \sqrt{L^{\cc{B}}_{n,m}\paren{\rho_2}} L_{n,m}\paren{\rho_1}\sqrt{L^{\cc{B}}_{n,m}\paren{\rho_2}} }{2}^2 &\leq& \LpNorm{ \sqrt{L^{\cc{B}}_{n,m}\paren{\rho_2}} L_{n,m}\paren{\rho_1 \vee \rho_2}\sqrt{L^{\cc{B}}_{n,m}\paren{\rho_2}} }{2}^2 \\
&=& \LpNorm{ \sqrt{L_{n,m}\paren{\rho_1 \vee \rho_2}} L^{\cc{B}}_{n,m}\paren{\rho_2}\sqrt{L_{n,m}\paren{\rho_1 \vee \rho_2}} }{2}^2 \\
&\leq& \LpNorm{ \sqrt{L_{n,m}\paren{\rho_1 \vee \rho_2}} L^{\cc{B}}_{n,m}\paren{\rho_1 \vee \rho_2}\sqrt{L_{n,m}\paren{\rho_1 \vee \rho_2}} }{2}^2.
\end{eqnarray*}
Thus we may suppose that $\rho_2\geq \rho_1$ without losing the generality. Namely $\rho_1 \vee \rho_2 = \rho_2$. In summary, so far we have
\begin{equation*}
Q \leq \LpNorm{ \sqrt{L_{n,m}\paren{\rho_2}} L^{\cc{B}}_{n,m}\paren{\rho_2}\sqrt{L_{n,m}\paren{\rho_2}} }{2}.
\end{equation*}
On the other hand,
\begin{equation*}
\LpNorm{ \sqrt{L_{n,m}\paren{\rho_2}} L^{\cc{B}}_{n,m}\paren{\rho_2}\sqrt{L_{n,m}\paren{\rho_2}} }{2}^2 = \RHS \coloneqq \tr\set{ L_{n,m}\paren{\rho_2}L^{\cc{B}}_{n,m}\paren{\rho_2}L_{n,m}\paren{\rho_2}L^{\cc{B}}_{n,m}\paren{\rho_2}}.
\end{equation*}
For any matrix $A$, define its absolute value by $\abs{A} = \brac{\abs{A_{\boldsymbol{s},\boldsymbol{t}}}}$. The triangle inequality says that for matrices $A_1,\ldots,A_b$, for some $b\in\bb{N}$, we have
\begin{equation*}
\tr\paren{A_1\ldots A_b}\leq \tr\paren{\abs{A_1}\ldots\abs{A_b}}.
\end{equation*}
This fact help us to find an upper bound on $\RHS$.
\begin{equation*}
\RHS \leq \tr\set{ \abs{L_{n,m}\paren{\rho_2}}\abs{L^{\cc{B}}_{n,m}\paren{\rho_2}}\abs{L_{n,m}\paren{\rho_2}}\abs{L^{\cc{B}}_{n,m}\paren{\rho_2}}}.
\end{equation*} 
Finally, since $\abs{L^{\cc{B}}_{n,m}\paren{\rho_2}}$ is the block diagonalized version of $\abs{L_{n,m}\paren{\rho_2}}$ and both of these matrices have non-negative entries, we get
\begin{eqnarray*}
\tr\set{ \abs{L_{n,m}\paren{\rho_2}}\abs{L^{\cc{B}}_{n,m}\paren{\rho_2}}\abs{L_{n,m}\paren{\rho_2}}\abs{L^{\cc{B}}_{n,m}\paren{\rho_2}}} &\leq& \tr\set{ \abs{L_{n,m}\paren{\rho_2}}\abs{L_{n,m}\paren{\rho_2}}\abs{L_{n,m}\paren{\rho_2}}\abs{L_{n,m}\paren{\rho_2}}}\\
&=& \LpNorm{ \abs{L_{n,m}\paren{\rho_2}}^2 }{2}^2.
\end{eqnarray*}
Combining the above inequalities yields
\begin{equation*}
Q \leq \LpNorm{ \abs{L_{n,m}\paren{\rho_2}}^2 }{2}.
\end{equation*}
Notice that the off-diagonal entries of $L_{n,m}\paren{\rho_2}$ and $\abs{L_{n,m}\paren{\rho_2}}$ decay with the same rate. Thus applying Lemma \ref{UppBndOffDiagASquare} can determine an bound on the entries of $\abs{L_{n,m}\paren{\rho_2}}^2$ as the following.
\begin{equation*}
\abs{\paren{\abs{L_{n,m}\paren{\rho_2}}^2}_{\boldsymbol{s},\boldsymbol{t}}}\lesssim \paren{1+N\LpNorm{\boldsymbol{t}-\boldsymbol{s}}{2} }^{-2\paren{m-\nu}} \set{1+\bbM{1}_{\set{m=\nu+d/2}}\log \paren{1+N\LpNorm{\boldsymbol{t}-\boldsymbol{s}}{2} } }.
\end{equation*}
Finally, Lemma \ref{FrobNormPolDecayOffDiag} guarantees the existence of a bounded scalar $c$ for which $\LpNorm{L^2_{n,m}\paren{\rho_2}}{2}\leq c\sqrt{n}$, finishing the proof of the first part. We now turn to the proof of the other side. Using the same trick as before implies that
\begin{equation*}
Q \geq \LpNorm{ \sqrt{L_{n,m}\paren{\rho_1}} L^{\cc{B}}_{n,m}\paren{\rho_1}\sqrt{L_{n,m}\paren{\rho_1}} }{2}.
\end{equation*} 
\end{proof}

\section{Auxiliary results}\label{AuxRes}

In this section we collect the auxiliary propositions and lemmas which come in handy to substantiate the results in Section \ref{Proofs} and Appendix \ref{AppendixCovMatrixIrregLat}. 

\subsection{The basic properties of matrices with polynomial decaying off-diagonals}

We showed in Appendix \ref{OffDiagDecay} that the off-diagonal entries of $K^{\cc{B}}_{n,m}\paren{\rho}$ decay polynomially in terms of the distance to the main diagonal. In this section, we show that such class of matrices are close to multiplication. We also investigate the large sample properties of their norms.

\begin{lem}\label{UppBndOffDiagASquare}
Let $N = \lfloor n^{1/d} \rfloor$ and suppose that $A_n\in\bb{R}^{n\times n}$ whose entries satisfy
\begin{equation}\label{DecayRateEntrA}
\abs{A_{\boldsymbol{s},\boldsymbol{t}}} \leq C\paren{1+ N\LpNorm{\boldsymbol{t}-\boldsymbol{s}}{2}}^{-\paren{d+\zeta}},\quad\forall\;\boldsymbol{s},\boldsymbol{t}\in\cc{D}_n.
\end{equation}
for some bounded $C>0$ and $\zeta\geq 0$. Then, the entries of $B=A^2$ are bounded above by
\begin{equation}\label{DecayRateEntrB}
\abs{B_{\boldsymbol{s},\boldsymbol{t}}} \lesssim \paren{1+ N\LpNorm{\boldsymbol{t}-\boldsymbol{s}}{2}}^{-\paren{d+\zeta}}\set{1+\bbM{1}_{\set{\zeta=0}}\log\paren{1+ N\LpNorm{\boldsymbol{t}-\boldsymbol{s}}{2}} },\quad\forall\boldsymbol{s},\boldsymbol{t}\in\cc{D}_n.
\end{equation}
\end{lem}

\begin{proof}
For simplicity let $\Delta = N\paren{\boldsymbol{t}-\boldsymbol{s}}$. Without loss of generality assume that $C=1$. We first justify Eq. \eqref{DecayRateEntrB} for the special case of $\Delta = \zero_d$ (associated to the diagonal entries of $B$). Indeed we need to show that all the diagonal entries of $B$ are smaller than some bounded scalar $C'$, which depends on $d$, $C$, and $\cc{D}_n$, i.e., $\abs{B_{\boldsymbol{s},\boldsymbol{s}}} \leq C'$ for any $\boldsymbol{s}\in\cc{D}_n$. Notice that the pairwise distances among two points in $\cc{D}_n$ have a similar behaviour to that of a $d$-dimensional regular lattice. Thus,
\begin{eqnarray*}
\abs{B_{\boldsymbol{s},\boldsymbol{s}}} &=& \abs{\sum_{\boldsymbol{r}\in\cc{D}_n} A^2_{\boldsymbol{s},\boldsymbol{r}} } \leq \sum_{\boldsymbol{r}\in\cc{D}_n} \paren{1+ N\LpNorm{\boldsymbol{r}-\boldsymbol{s}}{2}}^{-2\paren{d+\zeta}} \lesssim \int_{0}^{\infty} x^{d-1}\paren{1+x}^{-2\paren{d+\zeta}} dx \\
&\lesssim& \int_{1}^{\infty} x^{-\paren{d+1+2\zeta}} dx \asymp 1.
\end{eqnarray*}
Now suppose that $\Delta$ is a non-zero vector. Clearly $1\lesssim \LpNorm{\Delta}{2} \lesssim N$ and so $1+\LpNorm{\Delta}{2}^{d+\zeta} \asymp \LpNorm{\Delta}{2}^{d+\zeta}$. We replace Eq. \eqref{DecayRateEntrA} with the following more algebraically convenient alternative form.
\begin{equation*}
\abs{A_{\boldsymbol{s},\boldsymbol{t}}} \lesssim \brac{1+ \LpNorm{\Delta}{2}^{d+\zeta}}^{-1},\quad\forall\boldsymbol{s},\boldsymbol{t}\in\cc{D}_n,\quad\paren{\boldsymbol{t} = \boldsymbol{s}+\frac{\Delta}{N}}.
\end{equation*}
Next we obtain an upper bound on $\abs{B_{\boldsymbol{s},\boldsymbol{t}}}$ as the sum of two terms.
\begin{eqnarray*}
\abs{B_{\boldsymbol{s},\boldsymbol{t}}} &\lesssim& \sum_{\boldsymbol{r}\in\cc{D}_n} \frac{1}{\paren{1+ \LpNorm{N\paren{\boldsymbol{s}-\boldsymbol{r}}}{2}^{d+\zeta}} \paren{1+ \LpNorm{N\paren{\boldsymbol{t}-\boldsymbol{r}}}{2}^{d+\zeta}}}\\
&=& \sum_{\boldsymbol{r}\in\cc{D}_n} \frac{ \paren{1+ \LpNorm{N\paren{\boldsymbol{s}-\boldsymbol{r}}}{2}^{d+\zeta}}^{-1} }{2+ \LpNorm{N\paren{\boldsymbol{s}-\boldsymbol{r}}}{2}^{d+\zeta}+\LpNorm{N\paren{\boldsymbol{t}-\boldsymbol{r}}}{2}^{d+\zeta}} + \sum_{\boldsymbol{r}\in\cc{D}_n} \frac{ \paren{1+ \LpNorm{N\paren{\boldsymbol{t}-\boldsymbol{r}}}{2}^{d+\zeta}}^{-1} }{2+ \LpNorm{N\paren{\boldsymbol{s}-\boldsymbol{r}}}{2}^{d+\zeta}+\LpNorm{N\paren{\boldsymbol{t}-\boldsymbol{r}}}{2}^{d+\zeta}}.
\end{eqnarray*}
We write $\xi_1$ and $\xi_2$ to denote the first and second terms in the last line of the above expression. The next step serves as controlling $\xi_1$ from above. A similar upper bound can be found on $\xi_2$. For doing so, we introduce a lower bound on the expression in the denominator of $\xi_1$. Define $c = 2^{d+\zeta-1} \geq 1$. Applying Jensen's inequality on the convex univariate function $f\paren{x} = x^{d+\zeta}$ implies that
\begin{eqnarray*}
\LpNorm{N\paren{\boldsymbol{s}-\boldsymbol{r}}}{2}^{d+\zeta}+\LpNorm{N\paren{\boldsymbol{t}-\boldsymbol{r}}}{2}^{d+\zeta} &\geq& \frac{\LpNorm{N\paren{\boldsymbol{s}-\boldsymbol{r}}}{2}^{d+\zeta}}{c+1} + \frac{c}{c+1} \paren{\LpNorm{N\paren{\boldsymbol{s}-\boldsymbol{r}}}{2}^{d+\zeta}+\LpNorm{N\paren{\boldsymbol{t}-\boldsymbol{r}}}{2}^{d+\zeta}}\\
&\geq& \frac{\LpNorm{N\paren{\boldsymbol{s}-\boldsymbol{r}}}{2}^{d+\zeta}}{c+1} + \frac{\paren{\LpNorm{N\paren{\boldsymbol{s}-\boldsymbol{r}}}{2} + \LpNorm{N\paren{\boldsymbol{t}-\boldsymbol{r}}}{2}}^{d+\zeta}}{c+1}\\
&\geq&\frac{\LpNorm{N\paren{\boldsymbol{s}-\boldsymbol{r}}}{2}^{d+\zeta} + \LpNorm{\Delta}{2}^{d+\zeta} }{c+1}.
\end{eqnarray*}
Thus
\begin{equation}\label{UppBndxi1}
\xi_1\lesssim \sum_{\boldsymbol{r}\in\cc{D}_n}  \frac{1}{\paren{1+\LpNorm{N\paren{\boldsymbol{s}-\boldsymbol{r}}}{2}^{d+\zeta}}\paren{1+\LpNorm{N\paren{\boldsymbol{s}-\boldsymbol{r}}}{2}^{d+\zeta} + \LpNorm{\Delta}{2}^{d+\zeta}}}.
\end{equation}
Notice that the points in $\set{N\paren{\boldsymbol{s}-\boldsymbol{r}},\; \boldsymbol{r}\in\cc{D}_n}$ belong to a scaled (with the factor $N$) and translated version of $\cc{D}_n$. Assumption \ref{RegCondLattice} states that the pairwise distances in $\cc{D}_n$ and a regular lattice look alike. Hence, the summation in the right hand side of Eq. \eqref{UppBndxi1}, which only depends on the norm of the elements in $\cc{D}_n-\boldsymbol{s}$, can be upper bounded by an integral. Strictly speaking (in the following $x$ represents $\LpNorm{N\paren{\boldsymbol{s}-\boldsymbol{r}}}{2}$)
\begin{eqnarray*}
\xi_1&\lesssim& \int_{0}^{N} \frac{x^{d-1} dx}{ \paren{1+x^{d+\zeta}} \paren{1+x^{d+\zeta} + \LpNorm{\Delta}{2}^{d+\zeta}} }= \frac{1}{\LpNorm{\Delta}{2}^{d+\zeta}} \int_{0}^{N} \paren{\frac{x^{d-1}}{1+x^{d+\zeta}} - \frac{x^{d-1}}{1+x^{d+\zeta} + \LpNorm{\Delta}{2}^{d+\zeta}}} dx\\
&\lesssim& \LpNorm{\Delta}{2}^{-\paren{d+\zeta}} \brac{1 + \bbM{1}_{\set{\zeta = 0}} \log \paren{\frac{N^d \LpNorm{\Delta}{2}^d}{ N^d + \LpNorm{\Delta}{2}^d} }} \asymp \LpNorm{\Delta}{2}^{-\paren{d+\zeta}} \paren{1 + \bbM{1}_{\set{\zeta = 0}} \log \LpNorm{\Delta}{2}^d }.
\end{eqnarray*}
An analogous bound holds for $\xi_2$. Replacing these upper bounds in $\abs{B_{\boldsymbol{s},\boldsymbol{t}}}\lesssim \xi_1+\xi_2$ ends the proof.
\end{proof}

\begin{lem}\label{OpNormPolDecayOffDiag}
Let $\cc{D}_n$ be a irregular lattice of size $n$ satisfying Assumption \ref{RegCondLattice}. Define $N\coloneqq \lfloor n^{1/d}\rfloor$ and let $\Psi^n\in\bb{R}^{n\times n}$ be a symmetric matrix associated to $\cc{D}_n$ whose entries satisfy
\begin{equation*}
\abs{\Psi^n_{\boldsymbol{s},\boldsymbol{t}}}\leq C\paren{1+N\LpNorm{\boldsymbol{s}-\boldsymbol{t}}{2} }^{-\paren{d+\zeta}}, \quad\forall\; \boldsymbol{s},\boldsymbol{t}\in\cc{D}_n
\end{equation*}	
for some non-negative $\zeta$ and $C\in\paren{0,\infty}$. Then there exist bounded scalar $A,A'>0$ (depending on $C,d$ and $\zeta$) for which
\begin{enumerate}
\item $\OpNorm{\Psi^n}{2}{2}\leq A\Bigparen{1+ \bbM{1}_{\set{\zeta = 0 }}\log n}.$
\item $\LpNorm{\Psi^n}{2}\leq A'\sqrt{n}.$
\end{enumerate}
\end{lem}

\begin{proof}
We first focus on the operator norm of $\Psi^n$. The symmetry of $\Psi^n$ implies that
\begin{eqnarray}\label{GreshCircThm}
\OpNorm{\Psi^n}{2}{2} &\leq& \sqrt{ \OpNorm{\Psi^n}{1}{1} \OpNorm{\Psi^n}{\infty}{\infty}} = \OpNorm{\Psi^n}{1}{1} = \max_{ \boldsymbol{s}\in\cc{D}_n} \sum_{\boldsymbol{t}\in\cc{D}_n} \abs{\Psi^n_{\boldsymbol{s},\boldsymbol{t}}} \nonumber\\
&\leq& C\max_{ \boldsymbol{s}\in\cc{D}_n} \sum_{\boldsymbol{t}\in\cc{D}_n}\paren{1+N\LpNorm{\boldsymbol{s}-\boldsymbol{t}}{2} }^{-\paren{d+\zeta}}.
\end{eqnarray}
Choose $\boldsymbol{s}\in\cc{D}_n$. Reorder the points in $\cc{D}_n$ based on their distance from $\boldsymbol{s}$. Define the non-overlapping sets $\Pi_{\boldsymbol{s},l}$ by 
\begin{equation*}
\Pi_{\boldsymbol{s},l} = \set{\boldsymbol{t}\in\cc{D}_n: \frac{l}{N}\leq \LpNorm{\boldsymbol{s}-\boldsymbol{t}}{2}< \frac{l+1}{N} }, \quad\forall\; l\in\bb{N} \cup \set{0}.
\end{equation*}
The following facts are trivial implications of Assumption \ref{RegCondLattice}.
\begin{itemize}
\item There exists a bounded constant $a>0$ such that $\Pi_{\boldsymbol{s},l} = \emptyset$ for any $l>aN$.  
\item $\abs{\Pi_{\boldsymbol{s},l}}\lesssim \paren{l+1}^d - l^d \lesssim \paren{l+1}^{d-1}$ for any $l\leq aN$.
\end{itemize}
Thus,
\begin{equation}\label{1OpNormUppBnd}
\sum_{\boldsymbol{t}\in\cc{D}_n}\paren{1+N\LpNorm{\boldsymbol{s}-\boldsymbol{t}}{2} }^{-\paren{d+\zeta}} \leq \sum_{l=0}^{\infty}\abs{\Pi_{\boldsymbol{s},l}} \paren{l+1}^{-\paren{d+\zeta}} \lesssim \sum_{l=0}^{aN} \paren{l+1}^{-\paren{1+\zeta}}.
\end{equation}
We conclude the proof by substituting Eq. \eqref{1OpNormUppBnd} into Eq. \eqref{GreshCircThm}. Now we turn into finding an upper bound on $n^{-1}\LpNorm{\Psi^n}{2}^2$. Using similar techniques as \eqref{1OpNormUppBnd} yields
\begin{eqnarray}\label{1FrobNormUppBnd}
n^{-1}\LpNorm{\Psi^n}{2}^2 &\leq& n^{-1}\sum_{\boldsymbol{s}\in\cc{D}_n}\sum_{l=0}^{\infty} \abs{\Pi_{\boldsymbol{s},l}} \sup_{\boldsymbol{t}\in \Pi_{\boldsymbol{s},l}}\abs{\Psi^n_{\boldsymbol{s},\boldsymbol{t}}}^2 \leq \sum_{l=0}^{\infty} \abs{\Pi_{\boldsymbol{s},l}} \sup_{\boldsymbol{t}\in \Pi_{\boldsymbol{s},l}}\abs{\Psi^n_{\boldsymbol{s},\boldsymbol{t}}}^2\nonumber\\
&\leq& C^2 \sum_{l=0}^{aN} \abs{\Pi_{\boldsymbol{s},l}}\paren{l+1 }^{-2\paren{d+\zeta}}
\lesssim \sum_{l=0}^{\infty} \paren{l+1 }^{-\paren{d+1+2\zeta}} \asymp 1.
\end{eqnarray}
\end{proof}

The next result has a similar flavor as the second part of Lemma \ref{OpNormPolDecayOffDiag}. We omit its proof for avoiding the repetition. 

\begin{lem}\label{FrobNormPolDecayOffDiag}
Let $\cc{D}_n$ be a irregular lattice of size $n$ satisfying Assumption \ref{RegCondLattice}. Define $N\coloneqq \lfloor n^{1/d}\rfloor$ and let $\Psi^n\in\bb{R}^{n\times n}$ be a symmetric matrix associated to $\cc{D}_n$ whose entries satisfy
\begin{equation*}
\abs{\Psi^n_{\boldsymbol{s},\boldsymbol{t}}}\leq C\paren{1+N\LpNorm{\boldsymbol{s}-\boldsymbol{t}}{2} }^{-\paren{d+\zeta}}\set{1+\bbM{1}_{\set{\zeta = 0}} \log\paren{1+N\LpNorm{\boldsymbol{s}-\boldsymbol{t}}{2}} }, \quad\forall\; \boldsymbol{s},\boldsymbol{t}\in\cc{D}_n
\end{equation*}	
for some non-negative $\zeta$ and $C\in\paren{0,\infty}$. Then there exists a bounded scalar $A>0$ (depending on $C,d$ and $\zeta$) for which
\begin{equation*}
\LpNorm{\Psi^n}{2}\leq A\sqrt{n}.
\end{equation*}
\end{lem}

\subsection{Probabilistic inequalities}

We first extend Proposition $A.3$ of \cite{keshavarz2016consistency} regarding the uniform concentration of generalized $\chi^2$ random processes around its mean. It provides a powerful tool in the proof of Theorems \ref{CnstncyLocInvFreeThm} and \ref{AsympNormLocInvFreeThm}.

\begin{prop}\label{ModalProp}
Let $\Theta_0\subset\bb{R}^b,\;\forall\;n\in\bb{N}$ be a compact space with respect to the Euclidean metric. Consider the class of $n\times n$ matrices $\set{\Pi_n\paren{\theta}}_{\theta\in\Theta_0}$ parametrized by $\theta\in\Theta_0$. Suppose that the following conditions hold
	
\begin{enumerate}[label = (\alph*),leftmargin=*]
\item The normalized Frobenius norm of $\Pi_n\paren{\theta}$ is uniformly bounded on $\Theta_0$, i.e.,
\begin{equation*}
J_{\max}\coloneqq\sup_{n}\sup_{\theta\in\Theta_0} n^{-1/2}\LpNorm{\Pi_n\paren{\theta}}{2} < \infty.
\end{equation*}
\item The mapping $\paren{\theta,\LpNorm{\cdot}{2}}\mapsto \paren{\Pi_n\paren{\theta},\OpNorm{\cdot}{2}{2} }$ is Lipschitz with constant of order $\log^2 n$. Namely, there is $C>0$ for which
\begin{equation}\label{LipscCond}
\OpNorm{\Pi_n\paren{\theta_2}-\Pi_n\paren{\theta_1}}{2}{2} \leq C\log^2 n\LpNorm{\theta_2-\theta_1}{2},\quad\forall\;\theta_1,\theta_2\in\Theta_0\;\;\suchthat\abs{\theta_2-\theta_1}\leq 1.
\end{equation}
\item 
\begin{equation*}
\lim\limits_{n\rightarrow\infty}\OpNorm{\Pi_n\paren{\theta}}{2}{2}\sqrt{\frac{\log n}{n}}= 0, \quad\forall\;\theta\in\Theta_0.
\end{equation*}
\end{enumerate}
Then, there is a finite positive constant $C'$, depending on $C$, $J_{\max}$ and $b$, such that 
\begin{equation}\label{UnifDeviationBound}
\bb{P}\paren{ \sup_{\theta\in\Theta_0}\abs{Z^\top\Pi_n\paren{\theta}Z - \tr\set{\Pi_n\paren{\theta}} } \geq C'\sqrt{n\log n} } \leq \frac{1}{n},\quad\mbox{as}\;n\rightarrow\infty.
\end{equation}
\end{prop}

\begin{proof}
Let $r_n = 1/(C\sqrt{n\log^3 n})$ for $C$ defined in Eq. \eqref{LipscCond}. For large enough $n$, we have $r_n\leq 1$. Let $\cc{N}_{r_n}\paren{\Theta_0}$ represents the $r_n-$covering number of $\Theta_0$. The simple volume argument implies that 
\begin{equation}\label{CovNumberRate}
\abs{\cc{N}_{r_n}\paren{\Theta_0}} \lesssim \paren{\frac{\diam\paren{\Theta_0}}{r_n}}^b = \cc{O}\set{\paren{n\log^3 n}^{b/2}} .
\end{equation}
The key idea is to reduce the supremum over $\Theta_0$ in \eqref{UnifDeviationBound} to the discrete finite space $\cc{N}_{r_n}\paren{\Theta_0}$. Applying union bounded provides an upper bound on a probabilistic statement over $\cc{N}_{r_n}\paren{\Theta_0}$. Using the Hanson-Wright concentration inequality \cite{rudelson2013hanson} concludes the proof. 

For any $\theta\in\Theta_0$, let $\gamma_{\theta}$ stands for the closest element of $\cc{N}_{r_n}\paren{\Theta_0}$ to $\theta$. Thus, $\LpNorm{\theta-\gamma_{\theta}}{2}\leq r_n$. Observe that
\begin{align*}
&\RHS \coloneqq \abs{Z^\top\Pi_n\paren{\theta}Z - \tr\set{\Pi_n\paren{\theta}} - Z^\top\Pi_n\paren{\gamma_{\theta}}Z + \tr\set{\Pi_n\paren{\gamma_{\theta}}}} = \abs{\InnerProd{\Pi_n\paren{\theta}- \Pi_n\paren{\gamma_{\theta}} }{ZZ^\top+I_n}}\\
&\leq \OpNorm{\Pi_n\paren{\theta}- \Pi_n\paren{\beta_{\theta}}}{2}{2}\SpNorm{ZZ^\top+I_n}{1}
\RelNum{\paren{a}}{\leq} C\log^2 n\LpNorm{\theta-\beta_{\theta}}{2}\SpNorm{ZZ^\top+I_n}{1}\\
&\leq Cr_n\log^2 n\SpNorm{ZZ^\top+I_n}{1} = \sqrt{\frac{\log n}{n}}\paren{n+\LpNorm{Z}{2}^2}.
\end{align*}
Here $\paren{a}$ is implied from Eq. \eqref{LipscCond}. The Bernestein's inequality for the sub-exponential random variables states that
\begin{equation}\label{BernesteinIneq}
\bb{P}\paren{\LpNorm{Z}{2}^2\geq n + nt}\leq e^{-\frac{nt^2}{8}},\quad\forall\;t>0.
\end{equation} 
Choosing $t = 1$ in \eqref{BernesteinIneq} shows that $\RHS\geq 3\sqrt{n\log n}$ with probability at most $\exp\paren{-n/8}$. Hence,
\begin{equation*}
\bb{P}\paren{ \sup_{\theta\in\Theta_0}\abs{Z^\top\Pi_n\paren{\theta}Z - \tr\paren{\Pi_n\paren{\theta}} } \geq \sup_{\theta\in\cc{N}_{r_n}\paren{\Theta_0}}\abs{Z^\top\Pi_n\paren{\theta}Z - \tr\paren{\Pi_n\paren{\theta}} }+3\sqrt{n\log n} } \leq e^{-n/8}.
\end{equation*}
Recall $J_{\max}$ from the condition $\paren{a}$. Choose an arbitrary bounded $\xi$ such that $\xi> 1+b/2$. Eq. \eqref{CovNumberRate} can be rewritten as $\abs{\cc{N}_{r_n}\paren{\Theta_0}} n^{-\xi} = o\paren{n^{-1}}$, when $n$ tends to infinity. The proof will be terminated if we show that (for some bounded scalar $C_0$)
\begin{equation*}
\bb{P}\paren{ \sup_{\theta\in\cc{N}_{r_n}\paren{\Theta_0}}\abs{Z^\top\Pi_n\paren{\theta}Z - \tr\set{\Pi_n\paren{\theta}}} \geq  C_0J_{\max}\sqrt{n\log n}}\leq \abs{\cc{N}_{r_n}\paren{\Theta_0}} n^{-\xi} = o\paren{\frac{1}{n}},
\end{equation*}
as $n$ goes to infinity. For proving this claim, it suffices to obtain an appropriate probabilistic upper bound on $\abs{Z^\top\Pi_n\paren{\theta}Z - \tr\set{\Pi_n\paren{\theta}}}$ for any $\theta\in \cc{N}_{r_n}\paren{\Theta_0}$ and then exploiting the union bound trick. Hanson-Wright inequality \cite{rudelson2013hanson} says that for some $C_0 < \infty$ (depending on $\xi$), we have
\begin{equation}\label{HRIneq}
\bb{P}\brac{ \abs{Z^\top\Pi_n\paren{\theta}Z - \tr\set{\Pi_n\paren{\theta}}} \geq  C_0\paren{ \LpNorm{\Pi_n\paren{\theta}}{2}\sqrt{\log n} \vee \OpNorm{\Pi_n\paren{\theta}}{2}{2}\log n}}\leq n^{-\xi}.
\end{equation}
The condition $\paren{c}$ means that, $\OpNorm{\Pi_n\paren{\theta}}{2}{2}\log n = o\paren{\sqrt{n\log n}}$ as $n$ tends to infinity. So
\begin{eqnarray*}
\paren{ \LpNorm{\Pi_n\paren{\theta}}{2}\sqrt{\log n} \vee \OpNorm{\Pi_n\paren{\theta}}{2}{2}\log n} &=& \paren{ \LpNorm{\Pi_n\paren{\theta}}{2}\sqrt{\log n} \vee o\paren{\sqrt{n\log n}}}\\
&\leq& J_{\max}\sqrt{n\log n},\quad\;\mbox{as}\;n\rightarrow\infty,
\end{eqnarray*}
due to the condition $\paren{a}$. Thus Eq. \eqref{HRIneq} can be rewritten as 
\begin{equation*}
\bb{P}\paren{ \abs{Z^\top\Pi_n\paren{\theta}Z - \tr\set{\Pi_n\paren{\theta}}} \geq  C_0J_{\max}\sqrt{n\log n}}\leq n^{-\xi},\quad\forall\;\theta\in\cc{N}_{r_n}\paren{\Theta_0},
\end{equation*}
ending the proof of the claim.
\end{proof}

Next we rigorously state the squeeze theorem for weak convergence. It is beneficial in the proof of Theorem \ref{AsympNormLocInvFreeThm}.
		
\begin{lem}\label{SqueezeThm}
Let $\set{X_n}^{\infty}_{n=1},\set{Y_n}^{\infty}_{n=1}$ be two real valued sequences converging to $U$ in distribution. Suppose that $\set{Z_n}^{\infty}_{n=1}$ satisfies the following inequality
\begin{equation}\label{SqueezeIneq}
X'_n\coloneqq X_n\paren{1-p_n} \leq Z_n \leq Y'_n\coloneqq Y_n\paren{1+q_n}, \quad \forall\;n\in\bb{N},
\end{equation}
in which $p_n,q_n\cp{\bb{P}} 0$. Then $Z_n\cp{d} U$.
\end{lem}
			
\begin{proof}
Let $t\in\bb{R}$ be a continuity point of $U$. It suffices to show that $\bb{P}\paren{Z_n\geq t}\rightarrow\bb{P}\paren{U\geq t}$ as $n$ tends to infinity. Eq. \eqref{SqueezeIneq} obviously means that
\begin{equation*}
\bb{P}\paren{X'_n\geq t} \leq \bb{P}\paren{Z_n\geq t}\leq \bb{P}\paren{Y'_n\geq t},\quad \forall\;n\in\bb{N}.
\end{equation*}
Both $X'_n$ and $Y'_n$ weakly converge to $U$ by \emph{Slutsky's theorem}. Hence, $\bb{P}\paren{Y'_n\geq t}\rightarrow\bb{P}\paren{U\geq t}$ and $\bb{P}\paren{X'_n\geq t}\rightarrow\bb{P}\paren{U\geq t}$ as $n\rightarrow\infty$. Namely, both upper and lower bounds on $\bb{P}\paren{Z_n\geq t}$ converge to the same limit. Thus, $\lim\limits_{n\rightarrow\infty}\bb{P}\paren{Z_n\geq t}\rightarrow\bb{P}\paren{U\geq t}$ as a result of the usual \emph{squeeze theorem}.
\end{proof}

\bibliographystyle{abbrv}
\bibliography{Ref}

\end{document}